\documentclass{amsart}
\usepackage[T1]{fontenc}
\usepackage[english]{babel}
\usepackage{amsmath,amsthm,amssymb, latexsym,geometry,stackrel,enumerate}
\usepackage[all]{xy}
\usepackage{hyperref}

\usepackage[dvips]{graphicx}
\DeclareGraphicsExtensions{.bmp,.png,.pdf,.jpgg}
\usepackage[all]{xy}
\usepackage{verbatim}
\usepackage[mathcal]{euscript}
\usepackage{epsfig}
\usepackage{multicol}
\usepackage{color}
\usepackage{verbatim}
\usepackage{graphicx,float}
\usepackage{tikz-cd}
\usepackage{adjustbox}

\date{\today}

\swapnumbers
\theoremstyle{plain}
\newtheorem{teo}{Theorem}[section]
\newtheorem{lema}[teo]{Lemma}
\newtheorem{prop}[teo]{Proposition}
\newtheorem{coro}[teo]{Corollary}

\theoremstyle{definition}
\newtheorem{defi}[teo]{Definition}
\newtheorem{obs}[teo]{Remark}

\newtheorem{ejem}[teo]{Example}

\def\End{\mathop{\rm End_{\mathcal{C}_m}}\nolimits}

\input{xy}
\xyoption{poly}
\xyoption{2cell}
\xyoption{all}

\usepackage{url}

\title[The coloured mutation class of $\mathbb{D}_n$-quivers ]{The coloured mutation class of $\mathbb{D}_n$- quivers and their application to $m$-cluster tilted algebras}

\author[V. Gubitosi]{Viviana Gubitosi}
\address{Instituto de Matem\'{a}tica y Estad\'{\i}stica Rafael Laguardia, Facultad de Ingenier\'{\i}a - UdelaR, Montevideo, Uruguay, 11200 }
\email{gubitosi@fing.edu.uy}

\author[C. Qureshi]{Claudio Qureshi}
\address{Instituto de Matem\'{a}tica y Estad\'{\i}stica Rafael Laguardia, Facultad de Ingenier\'{\i}a - UdelaR, Montevideo, Uruguay, 11200 }
\email{cqureshi@fing.edu.uy}


\begin{document}
\maketitle

\begin{abstract}
In this paper, we present an explicit and purely combinatorial characterization of the $m$-coloured quivers that appear within the $m$-coloured mutation class of a quiver of type $\mathbb{D}_n$. The $m$-coloured mutation, as defined by Buan and Thomas in \cite{BT}, generalises the well-known quiver mutation introduced by Fomin and Zelevinsky \cite{FZ}. Consequently, we derive a comprehensive description of the Gabriel quivers associated with $m$-cluster-tilted algebras of type $\mathbb{D}_n$. Notably, our characterization extends a result by Vatne, \cite{Va}, which we recover when $m=1$.
\end{abstract}

\section*{Introduction}

For a given integer $m \geq 1$, $m$-coloured quivers, or simply coloured quivers, were introduced by Buan and Thomas in 2008 \cite{BT}. A coloured quiver $Q$ is defined as a finite quiver where each arrow is assigned a colour $c$ from the set $\{0, \dots, m\}$, subject to the conditions that $Q$ contains no loops and satisfies two additional properties: monochromaticity and skew-symmetry. Notably, any acyclic quiver $Q$ can be interpreted as a coloured quiver by assigning to each arrow of $Q$ the colour $0$ and subsequently adding an arrow of colour $m$ in the reverse direction.\\

Buan and Thomas also defined an operation, termed coloured quiver mutation, on these coloured quivers. They demonstrated that this operation is compatible with the mutation of $m$-cluster tilting objects. In the specific case of $m=1$, the coloured mutation of the induced coloured quiver corresponds to the mutation defined by Fomin and Zelevinsky in \cite{FZ}, commonly referred to as FZ-mutation. Since its introduction, quiver mutation has been the subject of extensive research, as evidenced by works such as \cite{BMR,BPRS,BR,BTo}. The problem of characterizing FZ-mutation classes for quivers of various types has been addressed by Vatne for type $\mathbb{D}_n$ \cite{Va}, Bastian for type $\mathbb{\tilde{A}}_n$ \cite{Ba}, and Buan and Vatne for type $\mathbb{A}_n$ \cite{BV}. Notably, a comparable study focusing on coloured mutations has only been conducted for type $\mathbb{A}_n$  in \cite{GQP} and for type $\mathbb{\tilde{A}}_n$ in \cite{GR}.\\

Two coloured quivers, $Q_1$ and $Q_2$, are considered mutation equivalent if $Q_1$ can be derived from $Q_2$ through a sequence of coloured mutations, and vice versa. An equivalence class defined by this relation is termed a coloured mutation class. Torkildsen demonstrated in \cite{Tork} that the coloured mutation class of a connected acyclic quiver $Q$ is finite if and only if $Q$ is either of Dynkin or extended Dynkin type, or possesses at most two vertices.\\

From an algebraic perspective, coloured quiver mutation provides insights into $m$-cluster tilted algebras, specifically algebras of the form $\text{End}(T)$, where $T$ represents an $m$-cluster tilting object within an $m$-cluster category $\mathcal{C}_m$.\\

It is observed in \cite{BT}  that the $0$-coloured part of the quiver $Q_T$ coincides with the Gabriel quiver of the $m$-cluster-tilted algebra $\End(T)$. Therefore, as a consequence we have that  Gabriel  quivers of  $m$-cluster-tilted algebras
can be combinatorially determined by applying repeated coloured mutations. See \cite[ Corollary 7.2]{BT}.\\

The primary objective of this paper is to provide an explicit characterization of the coloured mutation class of $\mathbb{D}_n$ coloured quivers. Specifically, we aim to present the complete set of coloured quivers derivable through iterative coloured mutations applied to a quiver whose underlying graph is of Dynkin type $\mathbb{D}_n$. Notably, we demonstrate that the quivers within this mutation class are readily identifiable.\\

In Definition \ref{la clase}, we introduce a class, denoted $\mathcal{Q}_{n,m}^D$, of $m$-coloured quivers with $n$ vertices, encompassing all colourations of a quiver of Dynkin type $\mathbb{D}_n$. One of the primary results of this paper is the following.

\subsection*{Theorem A}\textit{ A connected  $m$-coloured quiver $Q$ is mutation equivalent to  $\mathbb{D}_n$ if and only if  $Q$ belongs to the class $\mathcal{Q}_{n,m}^D$. }\\

\medskip

Specifically, by specializing to the case $m=1$, we recover established results presented in \cite{Va}. Furthermore, our characterization of the coloured mutation class of $\mathbb{D}_n$ coloured quivers provides a comprehensive description of quivers associated with $m$-cluster tilted algebras of type $\mathbb{D}_n$. Notably, these quivers remain largely unknown.  Our main result provides a complete combinatorial characterization of these algebras by describing their Gabriel quivers. These quivers are distinguished by the presence of a central cyclic structure, leading to the following classification:\\

\subsection*{Theorem B} \textit{ Let $Q$ be a connected  quiver of an $m$-cluster tilted algebra of type $\mathbb{D}_n$. Then,  $Q$ is an induced subquiver of a quiver $\tilde{Q}$ that satisfies one of the following structural configurations:}

\begin{itemize}
    \item[(a)] \textit{$\tilde{Q}$ contains an oriented cycle $S$ of length $m+3$, where exactly two non-adjacent vertices $a, b \in S_0$ have all their neighbours contained in $S$.}

    \item[(b)] \textit{$\tilde{Q}$ contains a non-oriented cycle $\mathcal{C}$ such that its $(m+1)r-(m-1)$ clockwise arrows are partitioned into $r$ disjoint blocks of length $m+1-k_i$, with $\sum_{i=1}^r k_i = m-1$. In this case, $Q$ is obtained by potentially removing $t$ of these $r$ blocks ($0 \leq t \leq r$).}
\end{itemize}

\textit{In both cases, $\tilde{Q}$ must satisfy that every vertex has in-degree and out-degree at most two, and any additional cycles are oriented of length $m+2$ and do not share arrows with $S$ or the clockwise arrows of $\mathcal{C}$, respectively.}\\

This paper is organised as follows. After a preliminary section to fix notations and recall essential
concepts, Section 2 reviews the definitions of coloured quivers and coloured  mutations. Section 3 is
 devoted to the definition and fundamental properties of mutation classes. In Section 4, we describe a
  special class of coloured quivers with $n$ vertices, denoted by $\mathcal{Q}_{n,m}^D$,  which is
  shown to be the coloured mutation class of type $\mathbb{D}_n$, while Section 5 presents its  main
  properties. Section 6 is dedicated to the proof of Theorem A. Finally, we conclude in Section 7 with
   our main result: a description of the Gabriel quivers of $m$-cluster tilted algebras of type $\mathbb{D}_n$.

\section{Preliminaries}

A  \textit{quiver} (or digraph) $Q$ is the data of two sets, $Q_0$ (the vertices) and $Q_1$ (the arrows); and two maps $s, t : Q_1 \rightarrow Q_0$ that assign to each arrow $\alpha$ its source $s(\alpha)$ and its target $t(\alpha)$. We write $\alpha: s(\alpha)\longrightarrow   t(\alpha)$ for the arrow $\alpha$ from $s(\alpha)$ to $t(\alpha)$. If either $s(\alpha)=i$ or $t(\alpha)=i$, we say that $\alpha$ is incident with the  vertex $i$. For any vertex $i$ in $Q_0$, the \textit{valency} of $i$ (in $Q$) is the number of neighbouring vertices,  i.e., the number of vertices $j\neq i$ such that there exists an arrow $\alpha \in Q_1$ with either $s(\alpha)=i$ and  $t(\alpha)=j$ or $s(\alpha)=j$ and $t(\alpha)=i$.\\

We say that $Q$ is \textit{simple} if there is at most one arrow between two distinct vertex.  In this case, each arrow $\alpha$ is determined by its source $s(\alpha)=i$ and its target $t(\alpha)=j$ and it is usual to denote $\alpha = ij$. A \textit{path} of length $k\geq 0$ in $Q$ is a sequence of arrows $\alpha_1\alpha_2\cdots \alpha_{k}$ such that $\alpha_i=x_i \rightarrow x_{i+1}\in Q_1$ for $1\leq i \leq k$.  A path is called {\it simple} when the vertices $x_1, x_2, \ldots, x_{k+1}$ are pairwise distinct. The path is called {\it closed} when $x_1=x_{k+1}$. A \textit{$k$-cycle} ($k\geq 3$) in  $Q$  is a closed path $\alpha_1\alpha_2\cdots \alpha_{k}$; i.e., $s(\alpha_1)=t(\alpha_{k})$ such that $\alpha_1\alpha_2\cdots \alpha_{k-1}$ is a simple path.
In the case where $Q$ is a simple quiver, we denote by $x_1x_2\cdots x_{k+1}$ the path $\alpha_1\alpha_2\cdots \alpha_{k}$ and by $(x_1x_2\cdots x_{k})$ the $k$-cycle $x_1x_2\cdots x_{k}x_1$. \\

Remember that a subquiver $Q'$ of $Q$ is called {\it induced} if every $\alpha \in Q_1$ such that $s(\alpha),t(\alpha)\in Q'_{0}$ satisfies $\alpha \in Q'_1$.  An \textit{induced cycle} is a cycle which is also an induced subquiver. A \textit{hole} is an induced cycle of length at least four. A quiver $Q$ is called \textit{free-hole}  if it does not contains holes. \

Let $I=\{x_1,\ldots,x_k\} \subseteq Q_0$. We denote by $Q[x_1,\ldots,x_k]$ (or just by $Q[I]$) the subquiver of $Q$ induced by $I$.  A \textit{complete} quiver is a quiver in which every pair of distinct vertices is connected by a pair of unique arrows (one in each direction).

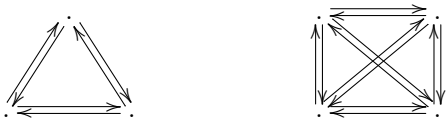
\begin{figure}[H]

\[
\begin{array}{ccc}
  \xymatrix@R=30pt@C=15pt{ &&& . \ar@<0.3ex>[rd] \ar@<0.3ex>[dl] &&  &   \\   && . \ar@<0.3ex>[ur] \ar@<0.3ex>[rr] && . \ar@<0.3ex>[ll] \ar@<0.3ex>[lu] } &  & \xymatrix@R=30pt@C=15pt{ . \ar@<0.3ex>[d] \ar@<0.3ex>[rr] \ar@<0.3ex>[rrd] && . \ar@<0.3ex>[ll] \ar@<0.3ex>[dll] \ar@<0.3ex>[d]  \\   . \ar@<0.3ex>[urr] \ar@<0.3ex>[rr] \ar@<0.3ex>[u] &&  . \ar@<0.3ex>[u] \ar@<0.3ex>[ull] \ar@<0.3ex>[ll] }
\end{array}
\]

\caption{Complete quivers  with $3$ and $4$ vertices respectively.}
\end{figure}

Given two vertices $a$ and $b$ of $Q$, we will say that $Q$ is an \textit{$[a,b]$-quasi-complete} quiver if $Q$ is a quiver in which every pair of distinct vertices is connected by a pair of unique arrows (one in each direction) except for the vertices $a$ and $b$ that they are not connected to each other.

\begin{figure}[H]

\[
\begin{array}{ccccc}
  \xymatrix@R=30pt@C=15pt{  && \\  a   && b   } && \xymatrix@R=30pt@C=15pt{ &&& . \ar@<0.3ex>[rd] \ar@<0.3ex>[dl] &&  &   \\   && a \ar@<0.3ex>[ur]  && b  \ar@<0.3ex>[lu] } &  & \xymatrix@R=30pt@C=15pt{ . \ar@<0.3ex>[d] \ar@<0.3ex>[rr] \ar@<0.3ex>[rrd] && . \ar@<0.3ex>[ll] \ar@<0.3ex>[dll] \ar@<0.3ex>[d]  \\   a \ar@<0.3ex>[urr] \ar@<0.3ex>[u] &&  b \ar@<0.3ex>[u] \ar@<0.3ex>[ull]  }
\end{array}
\]

\caption{$[a,b]$-quasi-complete quivers  with $2$, $3$ and $4$ vertices respectively.}
\end{figure}
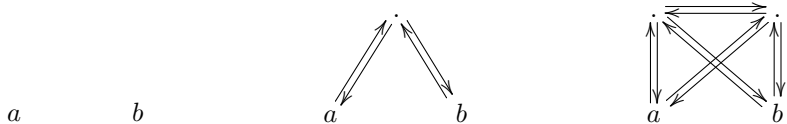

Let $I=\{x_1,\ldots,x_k\} \subseteq Q_0$. We denote by ${}^aQ^b[x_1,\ldots,x_k]$ (or just by ${}^aQ^b[I]$) the $[a,b]$-quasi-complete  subquiver of $Q$ induced by $I$. \\

A \textit{clique} is a subset of vertices $I\subseteq Q_0$ such that $Q[I]$ is a complete quiver. A clique with $k$ vertices is called a $k$-clique. In the following we will identify, as usual, the $k$-clique $K$ with its induced complete quiver $Q[K]$. \\

A maximum clique in $Q$ is a clique with the maximum number of vertices and the \textit{clique number} of $Q$, denoted by $\omega(Q)$, is the number of vertices in a maximum clique.\\

\section{ Coloured quivers and coloured mutations}
We recall in this section the definitions of  coloured quivers and coloured quivers mutation given by Buan and Thomas in \cite{BT}.

\subsection{Coloured quivers}

Let $m$ be a positive integer. An $m$-coloured quiver  $Q$ is the data of two sets, $Q_0$ (the vertices) and $Q_1$ (the arrows); two maps $s, t : Q_1 \rightarrow Q_0$ that assign to each arrow $\alpha$ its source $s(\alpha)$ and its target $t(\alpha)$; and a colouration function $\kappa:Q_1 \rightarrow \{0,1,\ldots, m\}$ which associates to each arrow  $\alpha$ the colour $\kappa(\alpha)$. We write $\alpha: s(\alpha)\overset{(c)}{\longrightarrow}  t(\alpha)$ for the arrow $\alpha$ from $s(\alpha)$ to $t(\alpha)$ of colour $\kappa(\alpha)=c$.\\

Let $q_{ij}^{(c)}$ denote the number of arrows from
$i$ to $j$ of colour $c$. We will say that $Q$ has no loops if $q_{ii}^{(c)} = 0$ for all $c$. The quiver  $Q$ is said to be monochromatic if $q_{ij}^{(c)} \neq 0$, then $q_{ij}^{(c')} = 0$ for $c \neq c'$. Clearly every simple quiver is monochromatic. Finally, we say that $Q$ is skew-symmetric if $q_{ij}^{(c)} = q_{ji}^{(m-c)}$. \\

From now on, we will consider coloured quivers with the above three additional conditions.\

\begin{obs}
If  $Q$ is a simple quiver. Then,  for each pair of vertices  $i$ and  $j$  in $Q_0$  there is at most one pair of arrows $\xymatrix{i \ar@<0.6ex>^{(c_{ij})}[r] & j\ar@<0.6ex>^{(m-c_{ij})}[l] }$  with  $c_{ij}\in \{0, 1, \dots, m \}$. 
\end{obs}

Since all our coloured quivers are skew-symmetric, the colour of every arrow $\xymatrix{v_i \ar[r]^{(c_{ij})} & v_j}$ determines the colour of the arrow $\xymatrix{v_j \ar[r]^{(c_{ji})} & v_i}$ according to the  equation $c_{ji}=m-c_{ij}$. Then, it will cause no confusion if  we  simply draw $\xymatrix{v_i \ar[r]^{c_{ij}} & v_j}$ instead of $\xymatrix{v_i \ar@<0.6ex>^{(c_{ij})}[r] & v_j\ar@<0.6ex>^{(m-c_{ij})}[l] }$.\\

\subsection{Coloured  mutation}

In such a  coloured quiver $Q$,  we have the following operation $\mu_j$ called coloured quiver mutation at  $j$.  Let $j$ be a vertex in $Q$ and let $\mu_j(Q) = \widetilde{Q}$ be the coloured quiver such that

$$\tilde{q}_{ik}^{(c)} =
\begin{cases}  q_{ik}^{(c+1)} & \text{  if $j =k$} \\
		 q_{ik}^{(c-1)} &\text{  if $j=i$} \\
		 \max \{0, q_{ik}^{(c)} - \sum_{t \neq c} q_{ik}^{(t)} + (q_{ij}^{(c)} - q_{ij}^{(c-1)}) q_{jk}^{(0)}
		 + q_{ij}^{(m)} (q_{jk}^{(c)}  -q_{jk}^{(c+1)}) \} & \text{  if $i \neq j \neq k$}
                 \end{cases}
$$

They also  give an alternative description of coloured quiver mutation at vertex $j$ that we recall here.

\subsubsection*{  \hspace*{4cm}    Alternative algorithm for coloured mutation:}

\begin{enumerate}
\item For each pair of arrows
$$ \xymatrix {i \ar^{(c)}[r] & j\ar^{(0)}[r] &k }$$
with $i\ne k$, the arrow from $i$ to $j$ of arbitrary colour $c$, and the arrow from
$j$ to $k$ of colour $0$, add a pair of arrows: an arrow from $i$ to $k$
of colour $c$, and one from $k$ to $i$ of colour $m-c$.

\item If the quiver is not longer monochromatic, because for some pair of vertices $i$ and $k$ there are arrows from
$i$ to $k$ which have two different colours, cancel the same number of
arrows of each colour, until the  monochromaticity property is satisfied.
\item Add one to the colour of any arrow going into $j$ and subtract one
from the colour of any arrow going out of $j$.
\end{enumerate}

Note that  the operations over the colours performed at step $(3)$ have to be done modulo $m+1$.\\

Buan and Thomas proved that the above algorithm is well-defined and
correctly calculates coloured
quiver mutation as previously defined by themselves.\\


\section{Mutation classes }\label{grafo subyacente}

Two quivers are said to be mutation equivalent if one can be obtained from the other by some sequence of coloured mutations, and viceversa. An equivalence class will be called a mutation class. \\

\begin{defi}
 We will say that the underlying graph of a coloured quiver  $Q$ (with no loops, monochromatic and  skew-symmetric)  is the graph obtained by keeping one edge  $\xymatrix{i \ar@{-}[r] & j}$ for each pair of arrows
$\xymatrix{i \ar@<0.6ex>^{(c)}[r] & j \ar@<0.6ex>^{(m-c)}[l]}$ of the coloured quiver $Q$.
\end{defi}

A natural question within the theory of coloured mutations is whether the initial colour configuration restricts the quiver's equivalence class. While the presence of cycles generally imposes such constraints, the landscape simplifies drastically in the acyclic case. In what follows, we provide a  result showing that, for any tree $T$, mutation equivalence is independent of the initial colouring of the associated quivers.

\begin{lema}
Let $Q$ and $Q'$ be two coloured quivers with the same underlying  graph $T$. If $T$ is a tree then $Q$ and $Q'$ are mutation equivalent.
\end{lema}

\begin{proof}
Let $v$ be a vertex of $T$, and let $w_1, \dots, w_k$ be its neighbours in $T$. Let $Q$ be a quiver whose underlying graph is $ T$, and denote by $c_i$ the color of the arrow $v \to w_i$ in $Q $. We say that $v$ is a \emph{stable vertex} of $Q$ if $$c_1 c_2 \cdots c_k = 0 \quad \text{implies} \quad c_1 = c_2 = \cdots = c_k = 0.$$
Observe that if $v$ is a stable vertex of $Q$, then $\mu_v(Q)$ is again a quiver whose underlying graph is $T$. A mutation $\mu_v : Q \to Q'$ is called \emph{stable} if the vertex $v$ is a stable vertex of $Q$.

We prove by induction on $n = |T_0|$ that if $Q$ and $Q'$ are two coloured quivers whose underlying graphs are isomorphic trees, then there exists a sequence of stable mutations from $Q$ to $Q'$. The result is immediate for $n=1$ and $n=2$.

Assume now that $n>2$, and let $v$ be a leaf of $T$. Let $w$ be the unique neighbor of $v$ in $T$, and let $T'$ be the tree obtained from $T$ by removing the vertex $v$. Denote by $ \widetilde{Q}$ the coloured quiver obtained from $Q$ by deleting the vertex $v$. Since the coloured quivers $\widetilde{Q}$ and $ \widetilde{Q}'$ have the same underlying tree $T'$, by inductive hypothesis there exists a sequence of stable mutations $$\mu_{v_i} : \widetilde{Q}_i \to \widetilde{Q}_{i+1}, \quad 0 \le i < k,$$
such that $\widetilde{Q}_0=\widetilde{Q}$ and $\widetilde{Q}_k = \widetilde{Q}'$. The idea is to use this sequence of mutations to pass from $ Q $ to $ Q' $, but before each occurrence of a mutation at $ w $, we insert suitable mutations at $ v $ so that the resulting sequence remains stable.

Let $ w_0, w_1, \dots, w_r $ be the neighbours of $ w $ in $ T $, where $ w_0 = v $. If $ Q $ is a coloured quiver whose underlying graph is $ T $, we denote by $ c_i(Q) $ the color of the arrow $ w \to w_i $ in $ Q $. Let $ 1 \le i_1 < i_2 < \cdots < i_t $ be the indices such that $ v_{i_j} = w $. Set $ Q_0 = Q $.

Note that the sequence of mutations
$$\xymatrix{
Q_0 = Q \ar[r]^{\mu_{v_0}} & Q_1 \ar[r]^{\mu_{v_1}} & \cdots \ar[r]^{\mu_{v_{i_1-2}}} & Q_{i_1-1} \ar[r]^{\mu_{v_{i_1-1}}} & Q_{i_1}'
}$$
is stable. If $ v_{i_1} = w $ is a stable vertex of $ Q_{i_1}' $, then we set $ Q_{i_1} := Q_{i_1}' $. Otherwise, two cases may occur:

\begin{enumerate}
\item $ c_0(Q_{i_1}') = 0 $, but $ c_i(Q_{i_1}') \ne 0 $ for all $ i = 1, \dots, r $. In this case, we apply the mutation $ \mu_v $ to $ Q_{i_1}' $, and define $ Q_{i_1} := \mu_v(Q_{i_1}') $. It is immediate that $ w $ is a stable vertex of $ Q_{i_1} $.

\item $ c_0(Q_{i_1}') \ne 0 $, but $ c_i(Q_{i_1}') = 0 $ for all $ i = 1, \dots, r $. In this case, we repeatedly apply the mutation $ \mu_v $ starting from $ Q_{i_1}' $ until we obtain a quiver $ Q_{i_1} $ such that $ c_0(Q_{i_1}) = 0 $. Again, $ w $ is a stable vertex of $ Q_{i_1} $.
\end{enumerate}

The sequence of mutations
$$\xymatrix{
Q_{i_1} \ar[r]^{\mu_{v_{i_1}}} & Q_{i_1+1} \ar[r]^{\mu_{v_{i_1+1}}} & \cdots \ar[r]^{\mu_{v_{i_2-2}}} & Q_{i_2-1} \ar[r]^{\mu_{v_{i_2-1}}} & Q_{i_2}'
}$$
is again stable. If $ v_{i_2} = w $ is a stable vertex of $ Q_{i_2}' $, we set $ Q_{i_2} := Q_{i_2}' $. Otherwise, as before, we apply a suitable sequence of mutations at $ v $ to obtain a coloured quiver $ Q_{i_2} $ for which $ w $ is a stable vertex. Repeating this procedure, we can construct a sequence of stable mutations from $ Q $ to $ Q' $.

\end{proof}

If $Q$ is a simple $m$-coloured quiver with  colouration function $\kappa:Q_1 \rightarrow \{0,1,\ldots, m\}$ and $\mathcal{P}$ denote the set of all non-trivial paths in $Q$ the colouration $\kappa$ induces a  \textit{weight function} $w:\mathcal{P}\rightarrow \mathbb{N}$ over $\mathcal{P}$ in the following way: if $p=x_1\cdots x_k x_{k+1}$ is a non-trivial path in $Q$, then $$w(p):= \sum_{i=1}^{k}\kappa(x_{i}x_{i+1}).$$

We also denote the weight of the path $x_1\cdots x_k x_{k+1}$ by $\overline{x_1\cdots x_k x_{k+1}}$.\\

If $\mathcal{O}=(x_1 \ldots x_k)$ is a  $k$-cycle in $Q$, then $w(x_1 x_2\cdots x_k x_{1})=w(x_2 x_3\cdots x_k x_{1}x_2)=\cdots =w(x_k x_1\cdots x_{k-1} x_{k})$. Therefore we can define the colouration of a cycle as follows:\\

\begin{defi}
Let  $\mathcal{O}=(x_1 \ldots x_k)$ a  $k$-cycle in $Q$. The colouration of $\mathcal{O}$ is defined as the minimum possible weight of a $k$-cycle contained in $\mathcal{O}$ ; i.e.
$$  \kappa(\mathcal{O}):= \min\{ w(x_1 x_2\cdots x_k x_{1}), w(x_1x_k\cdots x_2 x_{1}) \}  $$
\end{defi}

\subsection{Mutation class of  $\mathbb{A}_n$-quivers}

Remember that  a clique is a directed quiver in which every pair of distinct vertices is connected by an unique pair of edges (one in each direction). We denote by $\mathcal{C}_r$ a clique with $r$ vertices. \\

Let $\mathcal{Q}_{n,m}^A$ be the $m$-coloured mutation class of $\mathbb{A}_n$ described in \cite{GQP}. This is the class of $m$-coloured simple and connected quivers $Q$ with $n$ vertices and no holes which satisfy the following two additional  conditions:

\begin{enumerate}
\item For each vertex $v$ in $Q_0$ with $z\geq 1$ neighbours, there exists two cliques $\mathcal{C}_r$  and $\mathcal{D}_k$ such that $v \in \mathcal{C}_r \cap \mathcal{D}_k$,  $r+k=z+2$ and $r,k \leq m+2$. In addition, there are not arrows between two vertices $i\in \mathcal{C}_r$ and  $j\in \mathcal{D}_k$.

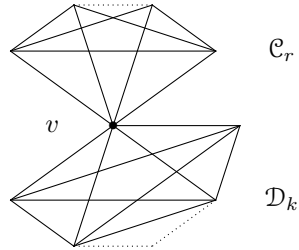
\begin{figure}[H]
\begin{center}

\[
\xy/r4pc/:{\xypolygon10"A"{~<{}~>{}{}}}
*+{{\scriptstyle \bullet}},

\POS"A6" \drop{\begin{array}{llllll} &&&&& v \end{array}}
\POS"A10" \ar@{-}   "A1"
\POS"A9" \ar@{.}   "A10"
\POS"A8" \ar@{.}   "A9"
\POS"A2" \ar@{-}   "A3"
\POS"A7" \ar@{-}   "A8"
\POS"A7" \ar@{-}   "A1"
\POS"A4" \ar@{-}   "A5"
\POS"A3" \ar@{.}   "A4"
\POS"A7" \ar@{-}   "A0"
\POS"A4" \ar@{-}   "A0"
\POS"A2" \ar@{-}   "A0"
\POS"A1" \ar@{-}   "A0"
\POS"A8" \ar@{-}   "A0"
\POS"A10" \ar@{-}   "A0"
\POS"A5" \ar@{-}   "A2"
\POS"A0" \ar@{-}   "A5"
\POS"A0" \ar@{-}   "A3"
\POS"A8" \ar@{-}   "A10"
\POS"A8" \ar@{-}   "A1"
\POS"A7" \ar@{-}   "A10"
\POS"A2" \ar@{-}   "A4"
\POS"A3" \ar@{-}   "A5"
\POS"A2" \drop{\begin{array}{llllll} &&&&& \mathcal{C}_r  \end{array}}
\POS"A10" \drop{\begin{array}{llllll} &&&&& \mathcal{D}_k  \end{array}}

\endxy   \]

\end{center}
\end{figure}

\item For each triangle (i.e, $3$-cycle) $(v_1v_2v_3)$ with $v_1,v_2,v_3$ three different vertices belonging to the same clique $\mathcal{C}$ we have that  the colouration $\kappa((v_1v_2v_3))=m-1$.\\

\end{enumerate}

The following figure shows an example of a $2$-coloured quiver in the class $\mathcal{Q}^A_{13,2}$.

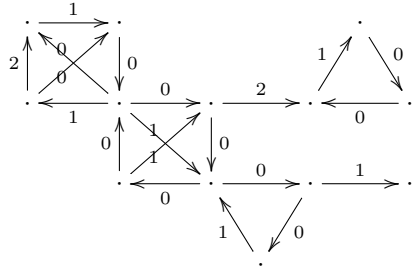
\begin{figure}[H]
\begin{center}
\[
\xymatrix@R=20pt@C=10pt{ \cdot \ar[rr]^{1} && \cdot \ar[d]^0 && &&  & \cdot \ar[dr]^0 & \\
\cdot \ar[u]^2 \ar[rru]^{0} && \cdot \ar[ll]^1 \ar[rr]^0 \ar[llu]^{ 0 } \ar[rrd]_1  && \cdot \ar[rr]^2 \ar[d]^0 && \cdot \ar[ru]^1 && \cdot \ar[ll]^0 \\
&& \cdot \ar[u]^0 \ar[rru]^1  && \cdot  \ar[ll]^0 \ar[rr]^0 && \cdot \ar[rr]^1 \ar[dl]^0 && \cdot \\
&& && & \cdot \ar[lu]^1 & &&
}
\]

\caption{A $2$-coloured quiver in the class $\mathcal{Q}^A_{13,2}$ .}
\label{Q132}
\end{center}
\end{figure}

We denote by $\mathcal{Q}_{m}^A$ the union of all $\mathcal{Q}_{n,m}^A$ for all $n$. Observe that any connected subquiver of a quiver $Q$ in $\mathcal{Q}_{m}^A$ is also in $\mathcal{Q}_{m}^A$.\\

Here we list some important properties of this class.\\

\newpage

\begin{prop}\cite[Lemmas 3.4, 3.7 and 3.8]{GQP}\label{Propiedades de la clase del An} Let $\mathcal{Q}_{n,m}^A$ be the $m$-coloured mutation class of $\mathbb{A}_n$. Then:
\begin{itemize}
  \item [(a)] $\mathcal{Q}_{n,m}^A$ is closed under  coloured quiver  mutation.
  \item [(b)] If $Q\in \mathcal{Q}_{n,m}^A$ and $v$ is any vertex of $Q$ we have $\mu_v^{m+1}(Q)=Q$.
  \item [(c)] For any  triangle $(v_1 v_2 v_3)$ belonging to the class $\mathcal{Q}_{n,m}^A$. We have that $\overline{v_1v_2}\neq \overline{v_1v_3}$. \\
\end{itemize}
\end{prop}

\subsection{Mutation class of type $\mathbb{D}_n$}

In the following, we will consider coloured quivers whose underlying graph is a Dynkin graph of type $\mathbb{D}_n$ with $n\geq 4$.
We will refer to these quivers as coloured quivers of type $\mathbb{D}_n$ or $\mathbb{D}_n$-quivers for short.\\

\begin{figure}[H]
\begin{center}
\[\xymatrix@R=10pt@C=20pt{&&&& n-1 \\ 1 \ar@{-}[r] & 2 \ar@{-}[r] & \cdots \ar@{-}[r] & n-2 \ar@{-}[ru]  \ar@{-}[rd]& \\
&&&& n}
\]
\caption{The Dynkin graph   $\mathbb{D}_n$}
\label{quiverDn}
\end{center}
\end{figure}
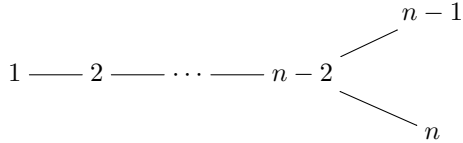

Let $Q$ be the following  coloured quiver with arrows only of colour $0$ and $m$ whose underlying graph is a Dynkin graph of type $\mathbb{D}_n$.

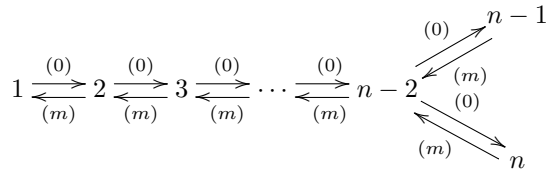
$$\xymatrix@R=15pt@C=20pt{&&&&& n-1 \ar@<0.6ex>^{(m)}[ld]\\
1 \ar@<0.6ex>^{(0)}[r] & 2 \ar@<0.6ex>^{(m)}[l] \ar@<0.6ex>^{(0)}[r] & 3  \ar@<0.6ex>^{(m)}[l]  \ar@<0.6ex>^{(0)}[r] & \cdots   \ar@<0.6ex>^{(m)}[l] \ar@<0.6ex>^{(0)}[r] & n-2  \ar@<0.6ex>^{(m)}[l] \ar@<0.6ex>^{(0)}[rd] \ar@<0.6ex>^{(0)}[ru] & \\
&&&&& n  \ar@<0.6ex>^{(m)}[lu] }$$

We call the mutation class of type $\mathbb{D}_n$ to the set of all quivers
mutation equivalent to $Q$.\
By a result of Torkildsen \cite{Tork}, we know that this mutation class is finite. \

\begin{ejem}  The following figure  shows all non-isomorphic coloured quivers in the mutation class of type $\mathbb{D}_4$ for $m=2$.

\begin{figure}[H]
\begin{center}
$$\begin{array}{ccccc}
  $\xymatrix@R=10pt@C=20pt{ & & 3  \\ 1 & 2 \ar[ru]^0 \ar[rd]^0 \ar[l]^0  & \\ && 4 }$ &
  $\xymatrix@R=10pt@C=20pt{&& 3 \\ 1 & 2 \ar[ru]^1 \ar[rd]^0 \ar[l]^0 \\ & &4}$ &
  $\xymatrix@R=10pt@C=20pt{&& 3 \\ 1 & 2 \ar[ru]^2 \ar[rd]^0 \ar[l]^0 \\ & &4}$ &
  $\xymatrix@R=10pt@C=20pt{&& 3 \ar[dl]^0 \\ 1 \ar[rru]^1 & 2 \ar[l]^0\ar[dr]_0 & \\ && 4 \ar[uu]_1}$ &
  $\xymatrix@R=10pt@C=20pt{& 3 \ar[ld]_0 & \\ 1 \ar[rd]_1 && 4  \ar[lu]_0\\
  & 2 \ar[ur]_0 & }$
   \\
   &&&&\\
  $\xymatrix@R=10pt@C=20pt{ && 3 \\ 1 & 2 \ar[ru]^0 \ar[rd]^1 \ar[l]^1 \\ & &4}$ &
  $\xymatrix@R=10pt@C=20pt{&& 3 \\ 1 & 2 \ar[ru]^1 \ar[rd]^1 \ar[l]^1 \\ & &4}$ &
  $\xymatrix@R=10pt@C=20pt{&& 3 \\ 1 & 2 \ar[ru]^2 \ar[rd]^1 \ar[l]^1 \\ & &4}$ &
  $\xymatrix@R=10pt@C=20pt{ && 3 \ar[dl]^0 \\ 1 \ar[rru]^0 & 2 \ar[l]^1\ar[dr]_1 & \\ && 4 \ar[uu]_0}$  & \\
  &&&& \\

  $\xymatrix@R=10pt@C=20pt{ & & 3  \\ 1 & 2 \ar[ru]^0 \ar[rd]^2 \ar[l]^2  & \\ && 4 }$ &
  $\xymatrix@R=10pt@C=20pt{& & 3  \\ 1 & 2 \ar[ru]^1 \ar[rd]^2 \ar[l]^2  & \\ && 4 }$ &
  $\xymatrix@R=10pt@C=20pt{& & 3  \\ 1 & 2 \ar[ru]^2 \ar[rd]^2 \ar[l]^2  & \\ && 4 }$ &
  $\xymatrix@R=10pt@C=20pt{&& 3 \ar[dl]^1 \\ 1 \ar[rru]^0 & 2 \ar[l]^0\ar[dr]_0 & \\ && 4 \ar[uu]_0}$ &

\end{array}$$
\caption{The mutation class of the 2-coloured quiver $\mathbb{D}_4 $}
\label{mutation class A3}
\end{center}
\end{figure}
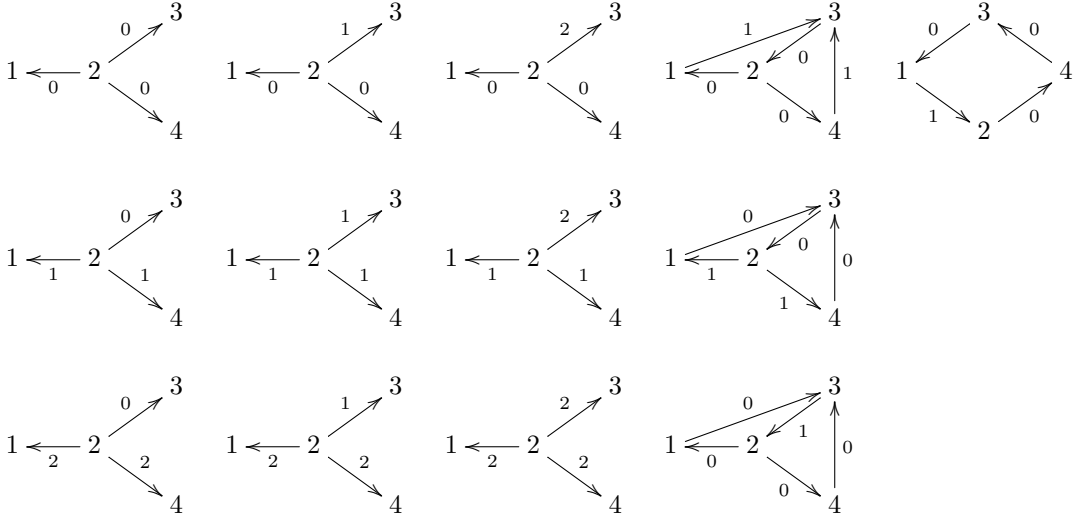

\end{ejem}

\section{The set $\mathcal{Q}_{n,m}^D$}

In this section we define a special class of $m$-coloured quivers with $n$ vertices which will turning  out the coloured mutation class of type $\mathbb{D}_n$.\\

In  the sequel,  given a subquiver $Q'$ we denote by  $\mathcal{A}(Q')$ the set of all arrows belonging to $Q'$ and we adopt the following convention concerning decorations on the names of vertices: $\circledast$  means that the vertex $\ast$ belongs also to a quiver  in the class $\mathcal{Q}_{m}^A$.\\

\begin{defi}\label{la clase}
We define $\mathcal{Q}_{n,m}^D$ to be the set of all quivers $m$-coloured quivers $Q$ with $n$ vertices belonging to one of the following two types:\\

\begin{description}
  \item[Type I] $Q$ has exacly two vertices $a$ and $b$ both of them belonging to at most two $[a,b]$-quasi-complete quivers ${}^aQ^b[x_1,\cdots,x_k]$ and ${}^aQ^b[y_1,\cdots,y_r]$  in such a way that: \\

      \begin{enumerate}
        \item $1 \leq r+k \leq m+1$; \\
        \item $\displaystyle  Q\setminus \{a,b,\mathcal{A}({}^aQ^b[x_1,\cdots,x_k]\cup {}^aQ^b[y_1,\cdots,y_r])\} $ is a disjoint union of connected quivers $Q^1\sqcup \cdots \sqcup Q^{r+k}$ all of them belonging to the class $\mathcal{Q}_{m}^A$; \\
        \item every triangle in $Q$ has colouration $m-1$. \\
        \item if $r,k \geq 1$, every cycle $(y_jax_ib)$ has colouration $\kappa((y_jax_ib))=m-1$.\\
      \end{enumerate}

  \item[Type II] $Q$ has a full subquiver $\mathcal{O}$, called the  \textit{central cycle}, which is an induced $k$-cycle $(x_1\cdots x_k)$ with $k\geq 3$  in such a  way that:\\

       \begin{enumerate}
         \item  The colouration $\kappa(\mathcal{O})=m-1$;\\

         \item  For every arrow $\alpha_i:x_{i}\rightarrow x_{i+1}$ there may exits a complete quiver $Q_{\alpha_i}=Q[x_i,x_{i+1},z_1^{\alpha_i},\cdots,z_{r_{\alpha_i}}^{\alpha_i}]$ of size at most $m+2$ (i.e., $1 \leq r_{\alpha_i}\leq m$), where the indices $i$ are taken modulo $k$;\\

         \item $\displaystyle Q\setminus \{\mathcal{O},\mathcal{A}(\cup_{i=1}^k Q_{\alpha_i})\}$ is a disjoint union of connected quivers $Q^1\sqcup \cdots \sqcup Q^{\alpha}$ all of them belonging to the class $\mathcal{Q}_{m}^A$, where $\alpha=\sum_{i=1}^kr_{\alpha_i}$. \\

         \item Every triangle in $Q$ has colouration $m-1$. \\

         \item If $w(x_1\dots x_k)=m-1$, then every triangle $(x_ix_{i+1}z_j^{\alpha_i})$ has weight $m-1$. \\

         \item If $k=3$ and there exists $1 \leq i \leq 3$ such that $\kappa(\alpha_i)=0$, then $\sum_{j\neq i}r_{\alpha_j}\leq m+1$. \\

         \end{enumerate}

In the sequel we adopt the following convention concerning decorations on the
names of vertices: $\circledast$  means that the vertex $\ast$ belongs also to a quiver $Q^{i}$ in the class $\mathcal{Q}_{m}^A$.

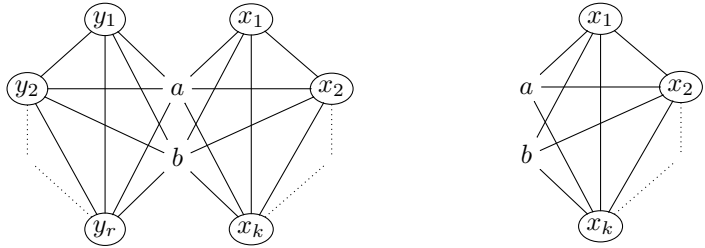
\begin{figure}[H]
\begin{center}
$$\begin{array}{ccccc}
$\SelectTips{eu}{10}\xymatrix@R=14pt@C=14pt{ &&& \\ & *+[o][F]{y_1} \ar@{-}[dr] \ar@{-}[dl] \ar@{-}[ddr]  \ar@{-}[ddd]& &  *+[o][F]{x_1} \ar@{-}[ddl] \ar@{-}[rd] \ar@{-}[ld] \ar@{-}[ddd]&    \\
  *+[o][F]{y_2} \ar@{..}[d] \ar@{-}[rr] \ar@{-}[ddr]\ar@{-}[rrd]& & a  &  &  *+[o][F]{x_2} \ar@{..}[d] \ar@{-}[ll] \ar@{-}[lld] \ar@{-}[ldd] \\
  &  & b &  &    \\
  & *+[o][F]{y_r} \ar@{-}[ruu]\ar@{..}[lu]\ar@{-}[ru] & &  *+[o][F]{x_k} \ar@{-}[luu] \ar@{-}[lu] \ar@{..}[ru]&  \\ &&&  }$
&&&&
$\SelectTips{eu}{10}\xymatrix@R=14pt@C=14pt{ &&&\\  &&  *+[o][F]{x_1} \ar@{-}[ddl] \ar@{-}[rd] \ar@{-}[ld] \ar@{-}[ddd]&    \\
  & a  &  &  *+[o][F]{x_2} \ar@{..}[d] \ar@{-}[ll] \ar@{-}[lld] \ar@{-}[ldd] \\
   & b &  &    \\
  & &  *+[o][F]{x_k} \ar@{-}[luu] \ar@{-}[lu] \ar@{..}[ru]&  \\ &&&  }$
\\
&&&&\\
r,k\geq 1 &&&&  r=0, \  k\geq 1 \\
&&&&\\
\end{array}$$

\caption{Examples of quivers of type I}
\label{tipe I}
\end{center}
\end{figure}

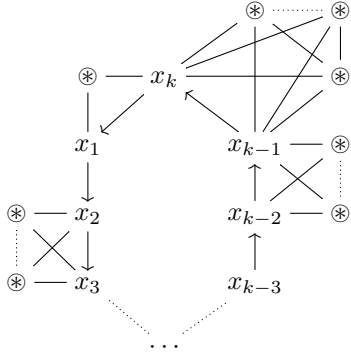
\begin{figure}[H]
\begin{center}
 $\SelectTips{eu}{10}\xymatrix@R=12pt@C=12pt{  &&&  \circledast \ar@{..}[r] \ar@{-}[dd] \ar@{-}[dr] & \circledast \ar@{-}[d] & \\
 & \circledast \ar@{-}[r] \ar@{-}[d] & x_k \ar[dl] \ar@{-}[rru] \ar@{-}[rr] \ar@{-}[ur]& & \circledast \ar@{-}[dl] & \\
& x_{1} \ar[d]  && x_{k-1}  \ar@{-}[ruu] \ar@{-}[r] \ar[ul] \ar@{-}[rd]  & \circledast \ar@{..}[d]& \\
\circledast \ar@{-}[r] \ar@{-}[rd] \ar@{..}[d]  & x_{2} \ar[d] && x_{k-2} \ar[u] \ar@{-}[ru]\ar@{-}[r] & \circledast& \\
\circledast \ar@{-}[r] \ar@{-}[ru] & x_{3}  \ar@{..}[dr] && x_{k-3}  \ar[u] & & \\
 & & \cdots \ar@{..}[ur]& & & \\
 &&&&&}$   \\

\caption{Example of a quiver of type  II}
\label{tipo II}
\end{center}
\end{figure}

\end{description}

\end{defi}
\vspace*{.5cm}

Observe that the subquivers $Q^i$ can be in the set $\mathcal{Q}_{1,m}^A$, i.e., they can have only one vertex.\\

\begin{obs}\label{no puedo tener 2 flechas de color cero hacia a y b}
For type  I,  observe that neither the triangle $(ax_ib)$ nor $(ay_jb)$  appear  because we do not have an edge $\xymatrix{ a \ar@{-}[r]& b}$ in $Q$. As a consequence, the arrows $x_i\rightarrow a$ and $x_i\rightarrow b$ can have the same colour. The same is true for the arrows $y_i\rightarrow a$ and $y_i\rightarrow b$.  However, if $r,k \geq 1$, every cycle $(y_jax_ib)$ has colouration $\kappa((y_jax_ib))=m-1$; and consequently $\overline{x_ia} \neq \overline{x_ib}$.
\end{obs}

\begin{obs}
In the specific case where $m=1$, we recover type I of \cite{Va} when  $r=0$ and $k=1$,  type II when $r=0$ and $k=2$ and type III when $r=k=1$.  Finally, our type II  corresponds to the type IV of \cite{Va}.
\end{obs}

\begin{lema}\label{unico arbol}
The only trees in the class $\mathcal{Q}_{n,m}^D$  are the coloured quivers of type $\mathbb{D}_{n}$.
\end{lema}

\begin{proof}

Let $Q$ be a tree in the class $\mathcal{Q}_{n,m}^D$. Since $Q$ is acyclic, it cannot be of type II. If $Q$ is of type I and $r, k > 1$, then $Q$ would contain the 4-cycle $(x_1ay_1b)$ (if $r, k \neq 0$) or the 3-cycle $(x_1ax_2)$ (if $r=0$ and $k>1$). Therefore, $r=0$ and $k=1$. We conclude by observing that the only tree in the class $\mathcal{Q}_{m}^A$ is an $\mathbb{A}$-quiver itself, due to the absence of triangles. Consequently, only an $\mathbb{A}$-quiver can be attached at vertex $x_1$.
\end{proof}

\section{Properties of the class $\mathcal{Q}_{n,m}^D$}

Recall that given a connected quiver $Q$, its Euler characteristic, $\chi(Q) = |Q_1| -|Q_0| + 1$, is the rank of its first homology group, $Q$ being viewed as a graph. Hereafter, the notation $\chi(Q)$, where $Q$ is an $m$-coloured quiver, will denote the Euler characteristic of the associated underlying graph of $Q$ (as defined in Section \ref{grafo subyacente}).\\

It is clear that for a connected and planar coloured quiver $Q$ its Euler characteristic is zero if and only if $Q$ is a tree. Therefore the quivers in the class
$\mathcal{Q}_{n,m}^D$ with  Euler characteristic exactly zero are the coloured quivers of type $\mathbb{D}_{n}$.\\

Recall that an unicycle quiver is defined as a quiver whose underlying graph contains exactly one cycle. A connected and planar quiver $Q$ is an unicycle quiver if and only if its Euler characteristic, $\chi(Q)$, is equal to 1. Consequently the quivers in the class $\mathcal{Q}_{n,m}^D$ with  Euler characteristic exactly one are the $n$-cycles or the quiver  of type I:

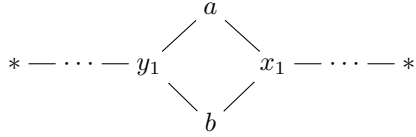
$$\SelectTips{eu}{10}\xymatrix@R=10pt@C=10pt{  &&& a    \ar@{-}[dl]  \ar@{-}[dr]  &&&&\\    \ast \ar@{-}[r] & \cdots  \ar@{-}[r]  & y_1 & & x_1   \ar@{-}[r] & \cdots  \ar@{-}[r] & \ast  \\ &&& b    \ar@{-}[ul]  \ar@{-}[ur]  &&&&}$$

Moreover, it is easy to see that the Euler characteristic of an $[a,b]$-quasi-complete quiver ${}^aQ^b[x_1,\cdots,x_k]$ and the Euler characteristic of a complete quiver $Q[x_1,\cdots,x_k]$ of size $k$ are

 $$ \chi({}^aQ^b[x_1,\cdots,x_k])=\frac{(k-1)(k +2) }{2} \  \  \text{and} \ \  \chi(Q[x_1,\cdots,x_k])=\frac{(k- 1) (k - 2)}{2}  \  \  \text{respectively}.$$

Following the notation of Definition \ref{la clase} we can compute the Euler characteristic of a quiver of type I as

 \begin{eqnarray*}
\chi(Q) & = & 1+ \chi({}^aQ^b[x_1,\cdots,x_k])+\chi({}^aQ^b[y_1,\cdots,y_r])+ \sum_{i=1}^{r+k}\chi(Q^i)\\
  & = &  1+ \frac{(k-1)(k +2) }{2}+ \frac{(r-1)(r+2) }{2}+\sum_{i=1}^{r+k}\chi(Q^i)
\end{eqnarray*}

    Analogously the Euler characteristic of a quiver of type II is:

    \begin{eqnarray*}
    \chi(Q) & = & \chi(\mathcal{O})+\sum_{i=1}^k\chi(Q_{\alpha_i})+ \sum_{i=1}^{t}\chi(Q^i)\\
     & = & 1+ \sum_{i=1}^k \frac{(r_{\alpha_i} - 1) (r_{\alpha_i} - 2)}{2}+\sum_{i=1}^{t}\chi(Q^i)
    \end{eqnarray*}

   \vspace*{.5cm}

 We will show in the following section that we can mutate a quiver  in $\mathcal{Q}_{n,m}^D$ reducing its  Euler characteristic. \\


Remember that the Minkowski sum of $A,B \subseteq \mathbb{Z}$ is $A+B := \{a+b: a\in A, b\in B\}$. In particular, the $n$-fold Minkowski sum of a set $A\subseteq \mathbb{Z}$ with itself is denoted by $nA=A+A+\cdots+A$ ($n$ terms). If $A\subseteq \mathbb{Z}$ and $h\in \mathbb{Z}$, the translate of $A$ by $h$ is $A+h:= \{a+h: a\in A\}$.\\

The following proposition follows directly from the definition of the class $\mathcal{Q}_{n,m}^D$. \\

\begin{prop}\label{colores de los 3 ciclos}
Let $Q$ be the quiver $ {}^aQ^b[x_1,\cdots,x_k] \cup {}^aQ^b[y_1,\cdots,y_r]$ with $r+k\geq 2$. Then:
\begin{itemize}
  \item [(a)]If $\mathcal{C}=(x_1x_2x_3)$ is a $3$-cycle in $Q$ then $\kappa(\mathcal{C})=m-1$ and $\overline{x_1x_2}\neq \overline{x_1x_3}$.
  \item [(b)]If $\mathcal{C}=(x_iay_jb)$, then $\kappa(\mathcal{C})=m-1$ and $\overline{ax_i}\neq \overline{ay_j}$.
\end{itemize}
\end{prop}

\vspace*{.5cm}

We need the following lemma which computes the possible colouration values of a $l$-cycle in a quiver of type I with $l\geq 4$ and with the vertices $a$ and $b$ being part of the cycle. Let $A=\{m-1,2m+1\}$ and $B=\{m-1,3m+1\}$ be the sets of the possible values of the weight of the cycles $(z_1z_2z_3)$ and $(x_iby_ja)$ respectively. \\

\begin{lema}\label{coloracion de un ciclo}
Let $Q$ be the quiver $ {}^aQ^b[x_1,\cdots,x_k] \cup {}^aQ^b[y_1,\cdots,y_r]$ with $r+k\geq 2$. Let $\mathcal{C}=(z_1\ldots z_l)$ be a $l$-cycle in $Q$, with $l\geq 4$ and $a,b\in \mathcal{C}_0$. Let $A=\{m-1,2m+1\}$ and $B=\{m-1,3m+1\}$. Then, $\overline{(z_1\cdots z_l)} \in (l-4)A + B -(l-4)m.$
\end{lema}

\begin{proof}
We can assume without loss of generality that $z_1=a$. Therefore $b=z_i$ for some $3\leq i < l$.  Since $Q$ is skew-symmetric,  $\overline{z_{j}a}+\overline{az_{j}}=m$ for all $j\neq i$ and  $2\leq j \leq l$. Then, the $l$-cycle $(a z_2 \cdots z_l)$ has weight $\overline{(a z_2\cdots z_l)} = \sum_{j=2}^{i-2} \overline{(az_{j}z_{j+1})} + \overline{(az_{i-1}b z_{i+1})} + \sum_{j=i+1}^{l-1} \overline{(az_{j}z_{j+1})}  - (l-4)m \in (l-4)A + B -(l-4)m$.

\end{proof}

\begin{obs}
The minimum value of $(l-4)A + B -(l-4)m $ is $(l-4)(m-1) + (m-1) -(l-4)m = m+3-l$ and the second minimum possible value is $(l-5)(m-1) + (2m+1) + (m-1) -(l-4)m=2m+5-l$. 
\end{obs}

In particular, we have the following corollary.

\begin{coro}
Let $Q$ be the quiver ${}^aQ^b[x_1,\cdots,x_k] \cup {}^aQ^b[y_1,\cdots,y_r]$ with $r+k\geq 2$. Let $\mathcal{C}=(z_1\ldots z_l)$ be a $l$-cycle in $Q$, with $l\geq 4$ and $a,b\in \mathcal{C}_0$ such that $\overline{(z_1\cdots z_l)}<2m+5-l$, then  $\overline{(z_1\cdots z_l)}=m+3-l$.
\end{coro}

Proposition \ref{colores de los 3 ciclos} can be generalised in the following way.\\

\begin{teo}\label{theo-coloracion ciclos}
If $Q$ is the quiver ${}^aQ^b[x_1,\cdots,x_k] \cup {}^aQ^b[y_1,\cdots,y_r]$ with $r+k\geq 2$ and $\mathcal{C}=(z_1\ldots z_l)$ is an $l$-cycle in $Q$, with $l\geq 4$ and $a,b\in \mathcal{C}_0$ then $\kappa(\mathcal{C})=m+3-l$.
\end{teo}

\begin{proof}
The result for $l=4$ follows from Proposition \ref{colores de los 3 ciclos}. Now assume that the result is true for some value of $l\geq 4$ and consider a $(l+1)$-cycle $\mathcal{C}=(z_0z_1z_2\ldots z_{l})$ in $Q$. Then, reordering indices if necessary, we can suppose  that $\overline{(z_0z_2\cdots z_l)}=m+3-l$ , $z_0=a$ and $z_{i+1}=b$ for some $i\leq l-2$. Then, we also have  $\overline{(z_0z_1z_{i+1}\cdots z_l)}=m+3-(l+2-i)$.\\

 Now we consider the cycles $c_1=(z_0z_{1}z_2 z_3 \cdots z_l)$, $c_2=(z_0z_2 z_{1} z_3 \cdots z_l)$, $\ldots$ , $c_{i}=(z_0z_2 z_3 \cdots z_iz_{1}z_{i+1}\cdots z_l)$ and calculate the sum weight of these cycles.
\begin{eqnarray*}
 \sum_{j=1}^{i} \overline{c_j} &=&  \sum_{j=2}^{i} (\overline{z_jz_{1}} + \overline{z_{1}z_{j}}) +  (i-1)(\overline{z_0z_2} + \overline{z_2z_3} + \cdots + \overline{z_{i}z_{i+1}}) + i ( \overline{z_{i+1}z_{i+2}} + \cdots +\overline{z_{l}z_0} ) +\overline{z_0z_{1}} + \overline{z_1z_{i+1}}  \\
                               &=&  m\cdot (i-1) + (i-1)\cdot(\overline{z_0z_2} + \overline{z_2z_3} + \cdots +\overline{z_{l}z_0}) +(\overline{z_0z_{1}} + \overline{z_1z_{i+1}} +\overline{z_{i+1}z_{i+2}} + \cdots +\overline{z_{l}z_0})\\
                               &=&  m\cdot(i-1) + (i-1) \cdot (m+3-l)+m+3-(l+2-i) \\
                               &=& i\cdot(2m+4-l) - (m+2)
\end{eqnarray*}
Then, the average weight of those cycles is $\frac{1}{i}\sum_{j=1}^{i} \overline{c_j}= 2m+4-l - \frac{m+2}{i}< 2m+4-l=2m+5-(l+1)$. This implies the existence of some cycle $c_j$ with $1\leq j \leq i$ such that $\overline{c_j}<2m+5-(l+1)$. By Lemma \ref{coloracion de un ciclo}, the only possibility is $\overline{c_j}=m+3-(l+1)$ which is the minimum possible weight of a $(l+1)$-cycle. Therefore  the colouration of the $(l+1)$-cycle $\mathcal{C}$ is $\kappa(\mathcal{C})=m+3-(l+1)$.

\end{proof}

Observe that the maximal length of a cycle in ${}^aQ^b[x_1,\cdots,x_k] \cup {}^aQ^b[y_1,\cdots,y_r]$ is $l=r+k+2$. If, in addition $r+k$ takes its maximal value $m+1$, then we have the following:\\

\begin{coro}\label{1<r+k=m+1 entonces flecha de color 0}
Let $Q$ be the quiver ${}^aQ^b[x_1,\cdots,x_k] \cup {}^aQ^b[y_1,\cdots,y_r]$ with $2 \leq r+k = m+1$. Then, there exists an $(m+3)$-cycle $\mathcal{C}=(z_1\ldots z_{m+3})$ in $Q$, with  $\overline{(z_1\cdots z_{m+3})}=0$. In particular, for every vertex $z\in \{x_1,\cdots,x_k,a,b,y_1,\cdots,y_r\}$ there exits a vertex $z'\in \{x_1,\cdots,x_k,a,b,y_1,\cdots,y_r\}\setminus \{z\}$ such that the colour  $\overline{zz'}=0$.
\end{coro}

Observe that every Hamiltonian cycle in ${}^aQ^b[x_1,\cdots,x_k] \cup {}^aQ^b[y_1,\cdots,y_r]$ has length $l=r+k+2$. If $2 \leq r+k < m+1$,  Theorem \ref{theo-coloracion ciclos} implies the following: \\

\begin{coro}
Let $Q$ be the quiver ${}^aQ^b[x_1,\cdots,x_k] \cup {}^aQ^b[y_1,\cdots,y_r]$ with $2 \leq r+k < m+1$. Then, for every Hamiltonian cycle  $\mathcal{C}=(z_1 \ldots z_{r+k+2})$ in $Q$,   the colouration  $\overline{(z_1\cdots z_{r+k+2})}\neq 0$.
\end{coro}


If $Q$ is a quiver of type I with $r=0$ we have the analogous result of Corollary \ref{1<r+k=m+1 entonces flecha de color 0}.

\begin{lema}\label{r=0 y k=m+1 entonces flecha de color 0}
Assume that $Q$ is a quiver of type I with $r=0$ and $k=m+1$. Then, for every vertex $z\in \{x_1,\cdots,x_{m+1},a,b\}$ there exits a vertex $z'\in \{x_1,\cdots,x_{m+1},a,b\}\setminus \{z\}$ such that the colour  $\overline{zz'}=0$.
\end{lema}

\begin{proof}
Assume that $z=a$ ( or $z=b$). According to Proposition \ref{colores de los 3 ciclos}, all the $m+1$ colours $\overline{ax_1}, \cdots, \overline{ax_{m+1}} $ are distinct. The claim follows since there are exactly $m+1$ different colours in the set $\{0, \cdots , m\}$.\\
If instead $z=x_i$ ( for some $1\leq i \leq m+1$) we can assume without loss of generality that $z=x_1$. Then, by Proposition \ref{colores de los 3 ciclos},  the $m+1$ colours $\overline{x_1a}, \overline{x_1x_2}\cdots, \overline{x_1x_{m+1}} $ are distinct. The same reasoning as above implies the claim.
\end{proof}

\medskip

A fundamental property of the class $\mathcal{Q}_{n,m}^D$  is the following.

\begin{prop}\label{cerrado por mutaciones}
The class $\mathcal{Q}_{n,m}^D$ is closed under  coloured quiver  mutation.
\end{prop}

\begin{proof}

Let $Q$ be a quiver in the class $\mathcal{Q}_{n,m}^D$ and  let $v$ be a vertex in $Q_0$. If $Q'$ is the quiver obtained after mutating $Q$ at the vertex $v$, we will show that $Q'$ belongs to the class $\mathcal{Q}_{n,m}^D$.\\

If there are not arrows of colour zero leaving the vertex $v$, the mutation at the vertex $v$ simply changes the colour of the arrows arriving or leaving $v$ and preserves without changes all the others colours. As in the $\mathbb{A}$-case, if there is a triangle $(vwu)$ in $Q$ with colouration $\kappa((vwu))=m-1$ the corresponding triangle in $Q'$ has the same colouration.\\

For the remaining cases, note that by Proposition \ref{Propiedades de la clase del An}, the class $\mathcal{Q}_{n,m}^A$ is closed under coloured quiver  mutation. Therefore, it suffices to consider the case where $v$ belongs to the set of vertices $\{a,b,x_1, \cdots, x_k,y_1,\cdots, y_r\}$ if $Q$ is of type I, or to the set  $\{x_1, \cdots, x_k\}\cup \{z : \text{there exits an arrow } zx_i \}$ if $Q$ is of type II. Consequently, the proof naturally splits into four cases.\\


(1) \emph{$Q$ is a quiver of type I and $v=a$}: In this case, the only possibility is that there exists a unique arrow of colour $0$ incident to some $x_i$ or $y_j$, according to Proposition \ref{colores de los 3 ciclos}. We may assume, without loss of generality, that the arrow points to $x_1$. The following figure illustrates this situation:\\

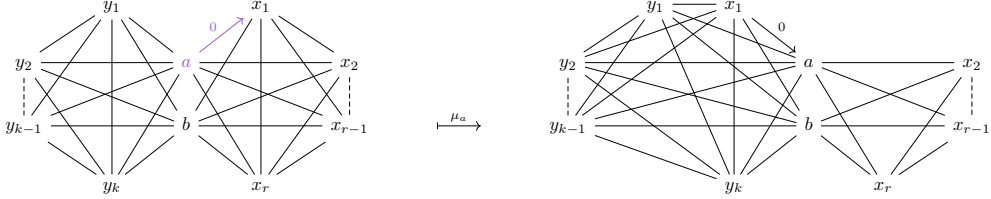
\adjustbox{scale=.7,center}{
\begin{tikzcd}
	& {y_1} && {{x_1}} &&&&& {y_1} & {x_1} \\
	{y_2} && \textcolor{rgb,255:red,153;green,92;blue,214}{a} && {x_2} &&& {y_2} &&& a && {x_2} \\
	{y_{k-1}} && b && {x_{r-1}} & {} & {} & {y_{k-1}} &&& b && {x_{r-1}} \\
	& {y_k} && {x_r} &&&&&& {y_k} && {x_r} \\
	&&&& {}
	\arrow[no head, from=1-2, to=2-1]
	\arrow[no head, from=1-2, to=3-1]
	\arrow[no head, from=1-2, to=3-3]
	\arrow[no head, from=1-2, to=4-2]
	\arrow[no head, from=1-4, to=2-5]
	\arrow[no head, from=1-4, to=4-4]
	\arrow[no head, from=1-9, to=1-10]
	\arrow[no head, from=1-9, to=2-8]
	\arrow[no head, from=1-9, to=2-11]
	\arrow[no head, from=1-9, to=3-8]
	\arrow[no head, from=1-9, to=3-11]
	\arrow[no head, from=1-9, to=4-10]
	\arrow[no head, from=1-10, to=2-8]
	\arrow["0", from=1-10, to=2-11]
	\arrow[no head, from=1-10, to=3-8]
	\arrow[no head, from=1-10, to=3-11]
	\arrow[no head, from=1-10, to=4-10]
	\arrow[no head, from=2-1, to=2-3]
	\arrow[dashed, no head, from=2-1, to=3-1]
	\arrow[no head, from=2-1, to=3-3]
	\arrow[no head, from=2-1, to=4-2]
	\arrow[no head, from=2-3, to=1-2]
	\arrow["0", color={rgb,255:red,153;green,92;blue,214}, from=2-3, to=1-4]
	\arrow[no head, from=2-3, to=2-5]
	\arrow[no head, from=2-3, to=3-5]
	\arrow[no head, from=2-3, to=4-4]
	\arrow[dashed, no head, from=2-5, to=3-5]
	\arrow[no head, from=2-5, to=4-4]
	\arrow[no head, from=2-8, to=2-11]
	\arrow[dashed, no head, from=2-8, to=3-8]
	\arrow[no head, from=2-8, to=3-11]
	\arrow[no head, from=2-8, to=4-10]
	\arrow[no head, from=2-11, to=2-13]
	\arrow[no head, from=2-11, to=3-13]
	\arrow[no head, from=2-11, to=4-10]
	\arrow[no head, from=2-11, to=4-12]
	\arrow[dashed, no head, from=2-13, to=3-13]
	\arrow[no head, from=2-13, to=4-12]
	\arrow[no head, from=3-1, to=2-3]
	\arrow[no head, from=3-1, to=3-3]
	\arrow[no head, from=3-1, to=4-2]
	\arrow[no head, from=3-3, to=1-4]
	\arrow[no head, from=3-3, to=2-5]
	\arrow[no head, from=3-3, to=3-5]
	\arrow[no head, from=3-3, to=4-2]
	\arrow[no head, from=3-3, to=4-4]
	\arrow[no head, from=3-5, to=1-4]
	\arrow["{{{{{{{{\mu_{a}}}}}}}}}", maps to, from=3-6, to=3-7]
	\arrow[no head, from=3-8, to=2-11]
	\arrow[no head, from=3-8, to=3-11]
	\arrow[no head, from=3-8, to=4-10]
	\arrow[no head, from=3-11, to=2-13]
	\arrow[no head, from=3-11, to=3-13]
	\arrow[no head, from=3-11, to=4-12]
	\arrow[no head, from=3-13, to=4-12]
	\arrow[no head, from=4-2, to=2-3]
	\arrow[no head, from=4-4, to=3-5]
	\arrow[no head, from=4-10, to=3-11]
\end{tikzcd}}

It is clear that after mutating at $a$, the condition about the sum of the sizes of the $[a,b]$-quasi-complete subquivers is preserved. That is,  $(r - 1) + (k + 1) = r + k \leq m+1$ holds.

\smallskip


(2) \emph{$Q$ is a quiver of type I and $v=x_i$}: Without loss of generality we can assume that $v=x_1$. We have the following general situation:

      \adjustbox{scale=.7,center}
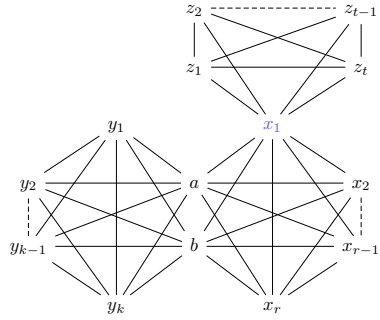
{
 \begin{tikzcd}
	&&&&&&&& \\
	&& {z_2} && {z_{t-1}} &&&&&&  \\
	&& {z_1} && {z_t} &&&& \\
	& {y_1} && \textcolor{rgb,255:red,92;green,92;blue,214}{{x_1}} && {} & {} &&&  \\
	{y_2} && a && {x_2} &&&&  \\
	{y_{k-1}} && b && {x_{r-1}} &&&&  \\
	& {y_k} && {x_r} &&&&\\
	&&&& {}
	\arrow[dashed, no head, from=2-3, to=2-5]
	\arrow[no head, from=2-3, to=3-3]
	\arrow[no head, from=2-3, to=3-5]
	\arrow[no head, from=2-3, to=4-4]
	\arrow[no head, from=2-5, to=3-5]
	\arrow[no head, from=2-5, to=4-4]
	\arrow[no head, from=3-3, to=2-5]
	\arrow[no head, from=3-3, to=3-5]
	\arrow[no head, from=3-3, to=4-4]
	\arrow[no head, from=4-2, to=5-1]
	\arrow[no head, from=4-2, to=5-3]
	\arrow[no head, from=4-2, to=6-1]
	\arrow[no head, from=4-2, to=6-3]
	\arrow[no head, from=4-2, to=7-2]
	\arrow[no head, from=4-4, to=3-5]
	\arrow[no head, from=4-4, to=5-5]
	\arrow[no head, from=4-4, to=7-4]
	\arrow[no head, from=5-1, to=5-3]
	\arrow[dashed, no head, from=5-1, to=6-1]
	\arrow[no head, from=5-1, to=6-3]
	\arrow[no head, from=5-1, to=7-2]
	\arrow[no head, from=5-3, to=4-4]
    \arrow[no head, from=5-3, to=6-5]
	\arrow[no head, from=5-3, to=5-5]
    \arrow[no head, from=5-3, to=7-4]
    \arrow[no head, from=5-5, to=7-4]
	\arrow[dashed, no head, from=5-5, to=6-5]
	\arrow[no head, from=6-1, to=5-3]
	\arrow[no head, from=6-1, to=6-3]
	\arrow[no head, from=6-1, to=7-2]
	\arrow[no head, from=6-3, to=4-4]
	\arrow[no head, from=6-3, to=5-5]
	\arrow[no head, from=6-3, to=6-5]
	\arrow[no head, from=6-3, to=7-2]
	\arrow[no head, from=6-3, to=7-4]
	\arrow[no head, from=6-5, to=4-4]
    \arrow[no head, from=6-5, to=7-4]
	\arrow[no head, from=7-2, to=5-3]
	\end{tikzcd}}

We begin by assuming there is a unique arrow of color 0 starting at $x_1$. This arrow may point to another $x_i$ in the  $[a, b]$-quasi-complete quiver ${}^aQ^b[x_1,\cdots,x_r] $, to the vertex $a$ (or $b$), or to a vertex $z_j$ within the complete subquiver $Q[x_1,z_1, \ldots, z_t]$. We analyse these possibilities as follows: \\

\begin{itemize}

\item[(a)] The arrow of color $0$ points to another $x_i$. In this case, the complete subquiver $Q[x_1, z_1, \dots, z_t]$ must have size $t + 1 \leq m + 1$; otherwise, if its size were $m + 2$, there would be another $0$-coloured arrow $x_1 \to z_i$. Consequently, after mutating at $x_1$, the new complete subquiver $Q'[x_1, z_1, \dots, z_t, x_r]$ has size at most $m+2$. Furthermore, since $r+k \leq m+1$ by hypothesis, it follows that $k + (r-1) \leq m+1$, preserving the required conditions.\\

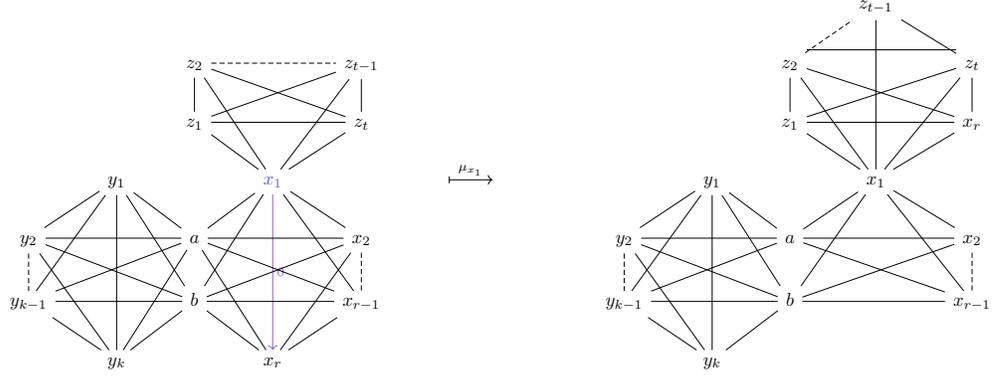
 \adjustbox{scale=.7,center}{
 \begin{tikzcd}
	&&&&&&&&&&& {z_{t-1}} \\
	&& {z_2} && {z_{t-1}} &&&&&& {z_2} && {z_{t}} \\
	&& {z_1} && {z_t} &&&&&& {z_1} && {x_r} \\
	& {y_1} && \textcolor{rgb,255:red,92;green,92;blue,214}{{x_1}} && {} & {} &&& {y_1} && {x_1} \\
	{y_2} && a && {x_2} &&&& {y_2} && a && {x_2} \\
	{y_{k-1}} && b && {x_{r-1}} &&&& {y_{k-1}} && b && {x_{r-1}} \\
	& {y_k} && {x_r} &&&&&& {y_k} \\
	&&&& {}
	\arrow[no head, from=1-12, to=2-13]
	\arrow[no head, from=1-12, to=4-12]
	\arrow[dashed, no head, from=2-3, to=2-5]
	\arrow[no head, from=2-3, to=3-3]
	\arrow[no head, from=2-3, to=3-5]
	\arrow[no head, from=2-3, to=4-4]
	\arrow[no head, from=2-5, to=3-5]
	\arrow[no head, from=2-5, to=4-4]
	\arrow[dashed, no head, from=2-11, to=1-12]
	\arrow[draw=none, from=2-11, to=2-13]
	\arrow[shift left=3, no head, from=2-11, to=2-13]
	\arrow[no head, from=2-11, to=3-11]
	\arrow[no head, from=2-11, to=3-13]
	\arrow[no head, from=2-11, to=4-12]
	\arrow[no head, from=2-13, to=3-13]
	\arrow[no head, from=2-13, to=4-12]
	\arrow[no head, from=3-3, to=2-5]
	\arrow[no head, from=3-3, to=3-5]
	\arrow[no head, from=3-3, to=4-4]
	\arrow[no head, from=3-11, to=2-13]
	\arrow[no head, from=3-11, to=3-13]
	\arrow[no head, from=4-2, to=5-1]
	\arrow[no head, from=4-2, to=5-3]
	\arrow[no head, from=4-2, to=6-1]
	\arrow[no head, from=4-2, to=6-3]
	\arrow[no head, from=4-2, to=7-2]
	\arrow[no head, from=4-4, to=3-5]
	\arrow[no head, from=4-4, to=5-5]
	\arrow["0", color={rgb,255:red,153;green,92;blue,214}, from=4-4, to=7-4]
	\arrow["{{{{{\mu_{x_1}}}}}}", maps to, from=4-6, to=4-7]
	\arrow[no head, from=4-10, to=5-11]
	\arrow[no head, from=4-10, to=6-9]
	\arrow[no head, from=4-10, to=6-11]
	\arrow[no head, from=4-10, to=7-10]
	\arrow[no head, from=4-12, to=3-11]
	\arrow[no head, from=4-12, to=3-13]
	\arrow[no head, from=4-12, to=5-13]
	\arrow[no head, from=4-12, to=6-11]
	\arrow[no head, from=4-12, to=6-13]
	\arrow[no head, from=5-1, to=5-3]
	\arrow[dashed, no head, from=5-1, to=6-1]
	\arrow[no head, from=5-1, to=6-3]
	\arrow[no head, from=5-1, to=7-2]
	\arrow[no head, from=5-3, to=4-4]
	\arrow[no head, from=5-3, to=5-5]
	\arrow[no head, from=5-3, to=6-5]
	\arrow[no head, from=5-3, to=7-4]
	\arrow[dashed, no head, from=5-5, to=6-5]
	\arrow[no head, from=5-5, to=7-4]
	\arrow[no head, from=5-9, to=4-10]
	\arrow[no head, from=5-9, to=5-11]
	\arrow[dashed, no head, from=5-9, to=6-9]
	\arrow[no head, from=5-9, to=6-11]
	\arrow[no head, from=5-9, to=7-10]
	\arrow[no head, from=5-11, to=4-12]
	\arrow[no head, from=5-11, to=5-13]
	\arrow[no head, from=5-11, to=6-13]
	\arrow[dashed, no head, from=5-13, to=6-13]
	\arrow[no head, from=6-1, to=5-3]
	\arrow[no head, from=6-1, to=6-3]
	\arrow[no head, from=6-1, to=7-2]
	\arrow[no head, from=6-3, to=4-4]
	\arrow[no head, from=6-3, to=5-5]
	\arrow[no head, from=6-3, to=6-5]
	\arrow[no head, from=6-3, to=7-2]
	\arrow[no head, from=6-3, to=7-4]
	\arrow[no head, from=6-5, to=4-4]
	\arrow[no head, from=6-9, to=5-11]
	\arrow[no head, from=6-9, to=6-11]
	\arrow[no head, from=6-9, to=7-10]
	\arrow[no head, from=6-11, to=5-13]
	\arrow[no head, from=6-11, to=6-13]
	\arrow[no head, from=7-2, to=5-3]
	\arrow[no head, from=7-4, to=6-5]
	\arrow[no head, from=7-10, to=6-11]
\end{tikzcd}}

\item [(b)] The arrow of color $0$ points to some vertex $z_i$. Note that this case cannot occur if $r + k = m + 1$. Indeed, by Corollary \ref{1<r+k=m+1 entonces flecha de color 0}, $x_1$ would be required to have an additional $0$-coloured arrow incident to a vertex $x_j$ in the $[a,b]$-quasi-complete subquiver ${}^aQ^b[x_1,\cdots,x_r] $, contradicting our assumption of a unique $0$-coloured arrow issuing from $x_1$. The situation is illustrated below.\\

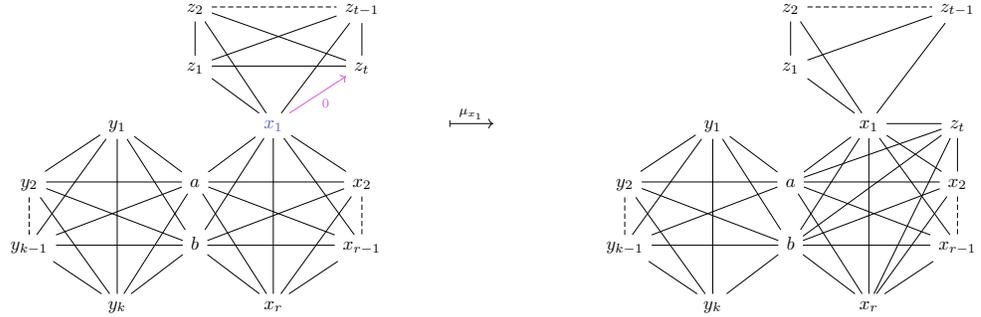
 \adjustbox{scale=.7,center}{
\begin{tikzcd}
	&& {z_2} && {z_{t-1}} &&&&&& {z_2} && {z_{t-1}} \\
	&& {z_1} && {z_t} &&&&&& {z_1} \\
	& {y_1} && \textcolor{rgb,255:red,92;green,92;blue,214}{{x_1}} && {} & {} &&& {y_1} && {x_1} & {z_t} \\
	{y_2} && a && {x_2} &&&& {y_2} && a && {         x_2} \\
	{y_{k-1}} && b && {x_{r-1}} &&&& {y_{k-1}} && b && {x_{r-1}} \\
	& {y_k} && {x_r} &&&&&& {y_k} && {x_{r}} \\
	&&&& {}
	\arrow[dashed, no head, from=1-3, to=1-5]
	\arrow[no head, from=1-3, to=2-3]
	\arrow[no head, from=1-3, to=2-5]
	\arrow[no head, from=1-3, to=3-4]
	\arrow[no head, from=1-5, to=2-5]
	\arrow[no head, from=1-5, to=3-4]
	\arrow[draw=none, from=1-11, to=1-13]
	\arrow[dashed, no head, from=1-11, to=1-13]
	\arrow[no head, from=1-11, to=2-11]
	\arrow[no head, from=1-11, to=3-12]
	\arrow[no head, from=1-13, to=3-12]
	\arrow[no head, from=2-3, to=1-5]
	\arrow[no head, from=2-3, to=2-5]
	\arrow[no head, from=2-3, to=3-4]
	\arrow[no head, from=2-11, to=1-13]
	\arrow[no head, from=3-2, to=4-1]
	\arrow[no head, from=3-2, to=4-3]
	\arrow[no head, from=3-2, to=5-1]
	\arrow[no head, from=3-2, to=5-3]
	\arrow[no head, from=3-2, to=6-2]
	\arrow["0"', color={rgb,255:red,214;green,92;blue,214}, from=3-4, to=2-5]
	\arrow[no head, from=3-4, to=4-5]
	\arrow[no head, from=3-4, to=6-4]
	\arrow["{{{{{\mu_{x_1}}}}}}", maps to, from=3-6, to=3-7]
	\arrow[no head, from=3-10, to=4-11]
	\arrow[no head, from=3-10, to=5-9]
	\arrow[no head, from=3-10, to=5-11]
	\arrow[no head, from=3-10, to=6-10]
	\arrow[no head, from=3-12, to=2-11]
	\arrow[no head, from=3-12, to=3-13]
	\arrow[no head, from=3-12, to=4-13]
	\arrow[no head, from=3-12, to=5-11]
	\arrow[no head, from=3-12, to=5-13]
	\arrow[no head, from=3-12, to=6-12]
	\arrow[no head, from=3-13, to=4-11]
	\arrow[no head, from=3-13, to=4-13]
	\arrow[no head, from=3-13, to=5-11]
	\arrow[no head, from=3-13, to=6-12]
	\arrow[no head, from=4-1, to=4-3]
	\arrow[dashed, no head, from=4-1, to=5-1]
	\arrow[no head, from=4-1, to=5-3]
	\arrow[no head, from=4-1, to=6-2]
	\arrow[no head, from=4-3, to=3-4]
	\arrow[no head, from=4-3, to=4-5]
	\arrow[no head, from=4-3, to=5-5]
	\arrow[no head, from=4-3, to=6-4]
	\arrow[dashed, no head, from=4-5, to=5-5]
	\arrow[no head, from=4-5, to=6-4]
	\arrow[no head, from=4-9, to=3-10]
	\arrow[no head, from=4-9, to=4-11]
	\arrow[dashed, no head, from=4-9, to=5-9]
	\arrow[no head, from=4-9, to=5-11]
	\arrow[no head, from=4-9, to=6-10]
	\arrow[no head, from=4-11, to=3-12]
	\arrow[no head, from=4-11, to=4-13]
	\arrow[no head, from=4-11, to=5-13]
	\arrow[no head, from=4-11, to=6-12]
	\arrow[dashed, no head, from=4-13, to=5-13]
	\arrow[no head, from=4-13, to=6-12]
	\arrow[no head, from=5-1, to=4-3]
	\arrow[no head, from=5-1, to=5-3]
	\arrow[no head, from=5-1, to=6-2]
	\arrow[no head, from=5-3, to=3-4]
	\arrow[no head, from=5-3, to=4-5]
	\arrow[no head, from=5-3, to=5-5]
	\arrow[no head, from=5-3, to=6-2]
	\arrow[no head, from=5-3, to=6-4]
	\arrow[no head, from=5-5, to=3-4]
	\arrow[no head, from=5-9, to=4-11]
	\arrow[no head, from=5-9, to=5-11]
	\arrow[no head, from=5-9, to=6-10]
	\arrow[no head, from=5-11, to=4-13]
	\arrow[no head, from=5-11, to=5-13]
	\arrow[no head, from=5-11, to=6-12]
	\arrow[no head, from=5-13, to=6-12]
	\arrow[no head, from=6-2, to=4-3]
	\arrow[no head, from=6-4, to=5-5]
	\arrow[no head, from=6-10, to=5-11]
\end{tikzcd}}

Under this configuration, after mutation at vertex $x_1$, the resulting $[a,b]$-quasi-complete subquivers fulfill the required condition  $(r + 1) + k \leq m + 1$, and the complete subquiver $Q'[x_1, z_1, \dots, z_{t-1}]$ has size $t \leq m + 1$.\\

\item[(c)] The unique arrow of color $0$ points to $a$ (or $b$). Here, we must have $t + 1 \leq m + 1$, as otherwise a $0$-coloured arrow $x_1 \to z_i$ would exist for some $i$.

Moreover, if $k \geq 1$, the $r + 1$ colours  $\overline{x_1x_2},\ldots, \overline{x_1x_r}, \overline{x_1a}, \overline{x_1b}$ of the arrows starting at $x_1$ must be distinct due to Proposition \ref{no puedo tener 2 flechas de color cero hacia a y b} and Lemma \ref{colores de los 3 ciclos}. This implies $r + 1 \leq m + 1$, so $r \leq m$. The mutation $\mu_{x_1}$ acts as illustrated in the following figure, transforming the quiver $Q$ of type I into a quiver $Q'$ of type II.\\

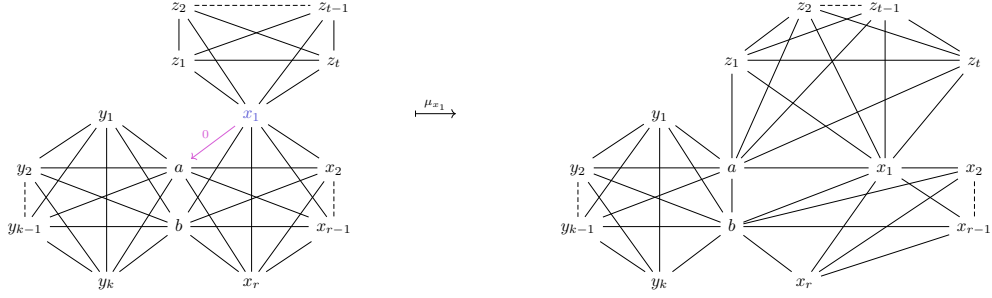
 \adjustbox{scale=.65,center}{
\begin{tikzcd}
	&& {z_2} && {z_{t-1}} &&&&&&& {z_2} & {z_{t-1}} \\
	&& {z_1} && {z_t} &&&&&& {z_1} &&& {z_t} \\
	& {y_1} && \textcolor{rgb,255:red,92;green,92;blue,214}{{x_1}} && {} & {} &&& {y_1} \\
	{y_2} && a && {x_2} &&&& {y_2} && a && {x_1} & {         x_2} \\
	{y_{k-1}} && b && {x_{r-1}} &&&& {y_{k-1}} && b &&& {x_{r-1}} \\
	& {y_k} && {x_r} &&&&&& {y_k} && {x_{r}} \\
	&&&& {}
	\arrow[dashed, no head, from=1-3, to=1-5]
	\arrow[no head, from=1-3, to=2-3]
	\arrow[no head, from=1-3, to=2-5]
	\arrow[no head, from=1-3, to=3-4]
	\arrow[no head, from=1-5, to=2-5]
	\arrow[no head, from=1-5, to=3-4]
	\arrow[draw=none, from=1-12, to=1-13]
	\arrow[dashed, no head, from=1-12, to=1-13]
	\arrow[no head, from=1-12, to=2-11]
	\arrow[no head, from=1-12, to=2-14]
	\arrow[no head, from=1-12, to=4-11]
	\arrow[no head, from=1-12, to=4-13]
	\arrow[no head, from=1-13, to=2-14]
	\arrow[no head, from=1-13, to=4-11]
	\arrow[no head, from=1-13, to=4-13]
	\arrow[no head, from=2-3, to=1-5]
	\arrow[no head, from=2-3, to=2-5]
	\arrow[no head, from=2-3, to=3-4]
	\arrow[no head, from=2-11, to=1-13]
	\arrow[no head, from=2-11, to=2-14]
	\arrow[no head, from=2-11, to=4-11]
	\arrow[no head, from=2-14, to=4-11]
	\arrow[no head, from=3-2, to=4-1]
	\arrow[no head, from=3-2, to=4-3]
	\arrow[no head, from=3-2, to=5-1]
	\arrow[no head, from=3-2, to=5-3]
	\arrow[no head, from=3-2, to=6-2]
	\arrow[no head, from=3-4, to=2-5]
	\arrow["0"', color={rgb,255:red,214;green,92;blue,214}, from=3-4, to=4-3]
	\arrow[no head, from=3-4, to=4-5]
	\arrow[no head, from=3-4, to=6-4]
	\arrow["{{{{{\mu_{x_1}}}}}}", maps to, from=3-6, to=3-7]
	\arrow[no head, from=3-10, to=4-11]
	\arrow[no head, from=3-10, to=5-9]
	\arrow[no head, from=3-10, to=5-11]
	\arrow[no head, from=3-10, to=6-10]
	\arrow[no head, from=4-1, to=4-3]
	\arrow[dashed, no head, from=4-1, to=5-1]
	\arrow[no head, from=4-1, to=5-3]
	\arrow[no head, from=4-1, to=6-2]
	\arrow[no head, from=4-3, to=4-5]
	\arrow[no head, from=4-3, to=5-5]
	\arrow[no head, from=4-3, to=6-4]
	\arrow[dashed, no head, from=4-5, to=5-5]
	\arrow[no head, from=4-5, to=6-4]
	\arrow[no head, from=4-9, to=3-10]
	\arrow[no head, from=4-9, to=4-11]
	\arrow[dashed, no head, from=4-9, to=5-9]
	\arrow[no head, from=4-9, to=5-11]
	\arrow[no head, from=4-9, to=6-10]
	\arrow[no head, from=4-11, to=4-13]
	\arrow[no head, from=4-11, to=5-11]
	\arrow[no head, from=4-13, to=2-11]
	\arrow[no head, from=4-13, to=2-14]
	\arrow[no head, from=4-13, to=4-14]
	\arrow[no head, from=4-13, to=5-11]
	\arrow[no head, from=4-13, to=5-14]
	\arrow[no head, from=4-13, to=6-12]
	\arrow[dashed, no head, from=4-14, to=5-14]
	\arrow[no head, from=4-14, to=6-12]
	\arrow[no head, from=5-1, to=4-3]
	\arrow[no head, from=5-1, to=5-3]
	\arrow[no head, from=5-1, to=6-2]
	\arrow[no head, from=5-3, to=3-4]
	\arrow[no head, from=5-3, to=4-5]
	\arrow[no head, from=5-3, to=5-5]
	\arrow[no head, from=5-3, to=6-2]
	\arrow[no head, from=5-3, to=6-4]
	\arrow[no head, from=5-5, to=3-4]
	\arrow[no head, from=5-9, to=4-11]
	\arrow[no head, from=5-9, to=5-11]
	\arrow[no head, from=5-9, to=6-10]
	\arrow[no head, from=5-11, to=4-14]
	\arrow[no head, from=5-11, to=5-14]
	\arrow[no head, from=5-11, to=6-12]
	\arrow[draw=none, from=5-14, to=6-12]
	\arrow[no head, from=6-2, to=4-3]
	\arrow[no head, from=6-4, to=5-5]
	\arrow[no head, from=6-10, to=5-11]
	\arrow[no head, from=6-12, to=5-14]
\end{tikzcd}}

Since $r \leq m$ and $t \leq m$, the new complete subquivers $Q'[z_1, \dots, z_t, x_1, a]$, $Q'[x_1, \dots, x_r, b]$, and $Q'[a, b, y_1, \dots, y_k]$ have sizes $t+2 \leq m+2$, $r+1 \leq m+1$, and $k+2 \leq m+2$, respectively.\\

If $k = 0$, we have the following situation with $r \leq m + 1$ and $t\leq m$.\\

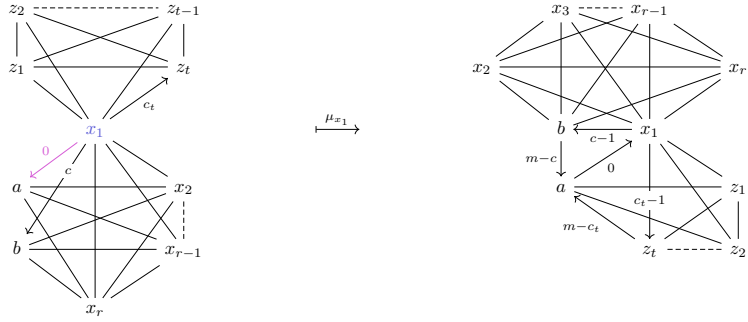
\adjustbox{scale=.7,center}{
\begin{tikzcd}
	{z_2} && {z_{t-1}} &&&&&& {x_3} & {x_{r-1}} & \\
	{z_1} && {z_t} &&&&& {x_2} &&& {x_r} \\
	& \textcolor{rgb,255:red,92;green,92;blue,214}{{{{{{x_1}}}}}} &&& {} & {} &&& b & {x_1} \\
	a && {x_2} &&&&&& a && {z_1} \\
	b && {x_{r-1}} &&&&&&& {z_t} & {         z_2} \\
	& {x_r} \\
	&& {}
	\arrow[dashed, no head, from=1-1, to=1-3]
	\arrow[no head, from=1-1, to=2-1]
	\arrow[no head, from=1-1, to=2-3]
	\arrow[no head, from=1-1, to=3-2]
	\arrow[no head, from=1-3, to=2-3]
	\arrow[no head, from=1-3, to=3-2]
	\arrow[dashed, no head, from=1-9, to=1-10]
	\arrow[no head, from=1-9, to=2-8]
	\arrow[draw=none, from=1-9, to=2-11]
	\arrow[no head, from=1-9, to=2-11]
	\arrow[no head, from=1-9, to=3-9]
	\arrow[no head, from=1-9, to=3-10]
	\arrow[draw=none, from=1-9, to=4-9]
	\arrow[no head, from=1-10, to=2-8]
	\arrow[no head, from=1-10, to=2-11]
	\arrow[no head, from=1-10, to=3-9]
	\arrow[no head, from=1-10, to=3-10]
	\arrow[no head, from=2-1, to=1-3]
	\arrow[no head, from=2-1, to=2-3]
	\arrow[no head, from=2-1, to=3-2]
	\arrow[no head, from=2-8, to=2-11]
	\arrow[no head, from=2-8, to=3-9]
	\arrow[no head, from=2-11, to=3-10]
	\arrow["{c_t}"', from=3-2, to=2-3]
	\arrow["0"', color={rgb,255:red,214;green,92;blue,214}, from=3-2, to=4-1]
	\arrow[no head, from=3-2, to=4-3]
	\arrow["c"{description, pos=0.3}, from=3-2, to=5-1]
	\arrow[no head, from=3-2, to=6-2]
	\arrow["{{{{{{{{{\mu_{x_1}}}}}}}}}}", maps to, from=3-5, to=3-6]
	\arrow[no head, from=3-9, to=2-11]
	\arrow["{\small{m-c}}"', from=3-9, to=4-9]
	\arrow[no head, from=3-10, to=2-8]
	\arrow["{\tiny{c-1}}", from=3-10, to=3-9]
	\arrow[no head, from=3-10, to=4-11]
	\arrow["{\tiny{c_t-1}}"{description, pos=0.6}, from=3-10, to=5-10]
	\arrow[no head, from=3-10, to=5-11]
	\arrow[no head, from=4-1, to=4-3]
	\arrow[no head, from=4-1, to=5-3]
	\arrow[no head, from=4-1, to=6-2]
	\arrow[dashed, no head, from=4-3, to=5-3]
	\arrow[no head, from=4-3, to=6-2]
	\arrow["{\small{0}}"', from=4-9, to=3-10]
	\arrow[no head, from=4-9, to=4-11]
	\arrow[no head, from=4-9, to=5-11]
	\arrow[no head, from=4-11, to=5-10]
	\arrow[no head, from=4-11, to=5-11]
	\arrow[no head, from=5-1, to=4-3]
	\arrow[no head, from=5-1, to=5-3]
	\arrow[no head, from=5-1, to=6-2]
	\arrow[no head, from=5-3, to=3-2]
	\arrow["{\tiny{m-c_t}}", from=5-10, to=4-9]
	\arrow[dashed, no head, from=5-11, to=5-10]
	\arrow[no head, from=6-2, to=5-3]
\end{tikzcd}}

If $r = m + 1$, there would have to be an additional $0$-coloured arrow from $x_1$ to $b$ since the $r$ colours  $\overline{x_1x_2},\ldots, \overline{x_1x_r}, \overline{x_1a}$ of the arrows starting at $x_1$ must be distinct due to  Proposition \ref{colores de los 3 ciclos}. In particular, none of the colours $\overline{x_1x_i}$ is $0$. Since the same holds for the $r$ colours  $\overline{x_1x_2},\ldots, \overline{x_1x_r}, \overline{x_1b}$, it follows that the arrow $x_1 \rightarrow b$ must be the arrow of colour $0$. However, as we are considering the case of a unique $0$-coloured arrow starting at $x_1$, it follows that $r \leq m$.

Consequently, the new complete subquiver $Q'[a, x_1, z_1, \dots, z_t]$ has size $t + 2 \leq m + 2$. Since $r \leq m$, the new complete subquiver $Q'[b, x_1, \dots, x_r]$ has size at most $m + 1$. A check of the induced cycle $(ax_1z_t)$ shows that condition (6) holds, with all other cycles following the same pattern. Thus, $Q'$ is indeed of type II.\\

\end{itemize}

Moving forward, suppose that there are two arrows of colour $0$ emanating from $x_1$. According to Proposition \ref{colores de los 3 ciclos}, these two arrows may be directed toward an $x_i$ and a $z_j$, toward some $x_i$ and either $a$ or $b$, or toward both $a$ and $b$ simultaneously. We explore these possibilities in the following discussion.

\begin{itemize}

\item[(a)]

 There are two arrows $x_1\rightarrow x_r$ and $x_1\rightarrow z_t$ of colour $0$. The mutation $\mu_{x_1}$, in this case,
acts as illustrated in the following figure:\\

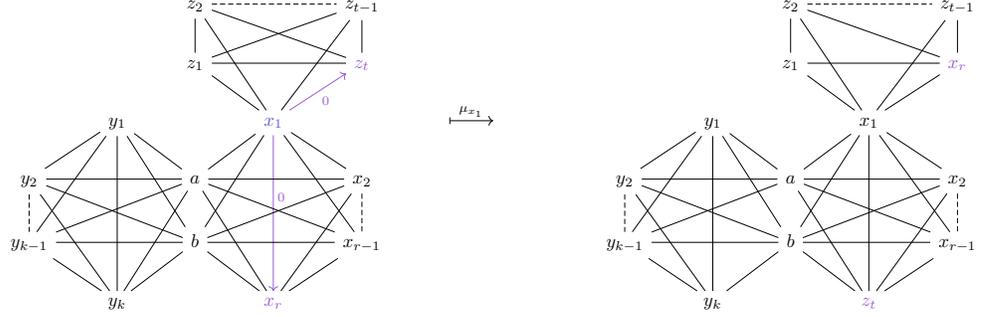
\adjustbox{scale=.7,center}{
\begin{tikzcd}
	&& {z_2} && {z_{t-1}} &&&&&& {z_2} && {z_{t-1}} \\
	&& {z_1} && \textcolor{rgb,255:red,153;green,92;blue,214}{{z_t}} &&&&&& {z_1} && \textcolor{rgb,255:red,153;green,92;blue,214}{{x_{r}}} \\
	& {y_1} && \textcolor{rgb,255:red,92;green,92;blue,214}{{{x_1}}} && {} & {} &&& {y_1} && {x_1} \\
	{y_2} && a && {x_2} &&&& {y_2} && a && {x_2} \\
	{y_{k-1}} && b && {x_{r-1}} &&&& {y_{k-1}} && b && {x_{r-1}} \\
	& {y_k} && \textcolor{rgb,255:red,153;green,92;blue,214}{{x_r}} &&&&&& {y_k} && \textcolor{rgb,255:red,153;green,92;blue,214}{{z_t}} \\
	&&&& {}
	\arrow[dashed, no head, from=1-3, to=1-5]
	\arrow[no head, from=1-3, to=2-3]
	\arrow[no head, from=1-3, to=2-5]
	\arrow[no head, from=1-3, to=3-4]
	\arrow[no head, from=1-5, to=2-5]
	\arrow[no head, from=1-5, to=3-4]
	\arrow[dashed, no head, from=1-11, to=1-13]
	\arrow[no head, from=1-11, to=2-11]
	\arrow[draw=none, from=1-11, to=2-13]
	\arrow[no head, from=1-11, to=2-13]
	\arrow[no head, from=1-11, to=3-12]
	\arrow[no head, from=1-13, to=2-13]
	\arrow[no head, from=1-13, to=3-12]
	\arrow[no head, from=2-3, to=1-5]
	\arrow[no head, from=2-3, to=2-5]
	\arrow[no head, from=2-3, to=3-4]
	\arrow[no head, from=2-11, to=2-13]
	\arrow[no head, from=2-13, to=3-12]
	\arrow[no head, from=3-2, to=4-1]
	\arrow[no head, from=3-2, to=4-3]
	\arrow[no head, from=3-2, to=5-1]
	\arrow[no head, from=3-2, to=5-3]
	\arrow[no head, from=3-2, to=6-2]
	\arrow["0"', color={rgb,255:red,153;green,92;blue,214}, from=3-4, to=2-5]
	\arrow[no head, from=3-4, to=4-5]
	\arrow["0"{pos=0.4}, color={rgb,255:red,153;green,92;blue,214}, from=3-4, to=6-4]
	\arrow["{{{{{{\mu_{x_1}}}}}}}", maps to, from=3-6, to=3-7]
	\arrow[no head, from=3-10, to=4-11]
	\arrow[no head, from=3-10, to=5-9]
	\arrow[no head, from=3-10, to=5-11]
	\arrow[no head, from=3-10, to=6-10]
	\arrow[no head, from=3-12, to=2-11]
	\arrow[no head, from=3-12, to=4-13]
	\arrow[no head, from=3-12, to=5-11]
	\arrow[no head, from=3-12, to=5-13]
	\arrow[no head, from=3-12, to=6-12]
	\arrow[no head, from=4-1, to=4-3]
	\arrow[dashed, no head, from=4-1, to=5-1]
	\arrow[no head, from=4-1, to=5-3]
	\arrow[no head, from=4-1, to=6-2]
	\arrow[no head, from=4-3, to=3-4]
	\arrow[no head, from=4-3, to=4-5]
	\arrow[no head, from=4-3, to=5-5]
	\arrow[no head, from=4-3, to=6-4]
	\arrow[dashed, no head, from=4-5, to=5-5]
	\arrow[no head, from=4-5, to=6-4]
	\arrow[no head, from=4-9, to=3-10]
	\arrow[no head, from=4-9, to=4-11]
	\arrow[dashed, no head, from=4-9, to=5-9]
	\arrow[no head, from=4-9, to=5-11]
	\arrow[no head, from=4-9, to=6-10]
	\arrow[no head, from=4-11, to=3-12]
	\arrow[no head, from=4-11, to=4-13]
	\arrow[no head, from=4-11, to=5-13]
	\arrow[no head, from=4-11, to=6-12]
	\arrow[dashed, no head, from=4-13, to=5-13]
	\arrow[no head, from=4-13, to=6-12]
	\arrow[no head, from=5-1, to=4-3]
	\arrow[no head, from=5-1, to=5-3]
	\arrow[no head, from=5-1, to=6-2]
	\arrow[no head, from=5-3, to=3-4]
	\arrow[no head, from=5-3, to=4-5]
	\arrow[no head, from=5-3, to=5-5]
	\arrow[no head, from=5-3, to=6-2]
	\arrow[no head, from=5-3, to=6-4]
	\arrow[no head, from=5-5, to=3-4]
	\arrow[no head, from=5-9, to=4-11]
	\arrow[no head, from=5-9, to=5-11]
	\arrow[no head, from=5-9, to=6-10]
	\arrow[no head, from=5-11, to=4-13]
	\arrow[no head, from=5-11, to=5-13]
	\arrow[no head, from=5-11, to=6-12]
	\arrow[no head, from=5-13, to=6-12]
	\arrow[no head, from=6-2, to=4-3]
	\arrow[no head, from=6-4, to=5-5]
	\arrow[no head, from=6-10, to=5-11]
\end{tikzcd}}

Note that, after mutating at $x_1$, the vertices $x_r$ and $z_t$ exchange their positions, implying
that the sizes of the complete and quasi-complete subquivers involved remain unchanged.

\item[(b)] There are two arrows $x_1\rightarrow a$ and
 $x_1\rightarrow z_t$ of colour $0$. The mutation
$\mu_{x_1}$, in this case, acts as illustrated in the
following figure by transforming the type I quiver $Q$ into the type II quiver $Q'$:\\

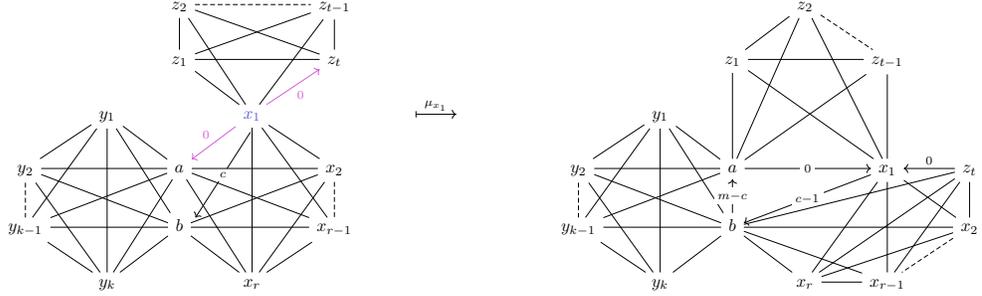
\adjustbox{scale=.65,center}{
\begin{tikzcd}
	&& {z_2} && {z_{t-1}} &&&&&&& {z_2} && \\
	&& {z_1} && {z_t} &&&&&& {z_1} && {z_{t-1}} \\
	& {y_1} && \textcolor{rgb,255:red,92;green,92;blue,214}{{{{x_1}}}} && {} & {} &&& {y_1} \\
	{y_2} && a && {x_2} &&&& {y_2} && a && {x_1} & {z_t} \\
	{y_{k-1}} && b && {x_{r-1}} &&&& {y_{k-1}} && b &&& {         x_2} \\
	& {y_k} && {x_r} &&&&&& {y_k} && {x_{r}} & {x_{r-1}} \\
	&&&& {}
	\arrow[dashed, no head, from=1-3, to=1-5]
	\arrow[no head, from=1-3, to=2-3]
	\arrow[no head, from=1-3, to=2-5]
	\arrow[no head, from=1-3, to=3-4]
	\arrow[no head, from=1-5, to=2-5]
	\arrow[no head, from=1-5, to=3-4]
	\arrow[no head, from=1-12, to=2-11]
	\arrow[draw=none, from=1-12, to=2-13]
	\arrow[dashed, no head, from=1-12, to=2-13]
	\arrow[no head, from=1-12, to=4-11]
	\arrow[no head, from=1-12, to=4-13]
	\arrow[no head, from=2-3, to=1-5]
	\arrow[no head, from=2-3, to=2-5]
	\arrow[no head, from=2-3, to=3-4]
	\arrow[no head, from=2-11, to=2-13]
	\arrow[no head, from=2-11, to=4-11]
	\arrow[no head, from=2-13, to=4-11]
	\arrow[no head, from=2-13, to=4-13]
	\arrow[no head, from=3-2, to=4-1]
	\arrow[no head, from=3-2, to=4-3]
	\arrow[no head, from=3-2, to=5-1]
	\arrow[no head, from=3-2, to=5-3]
	\arrow[no head, from=3-2, to=6-2]
	\arrow["0"', color={rgb,255:red,214;green,92;blue,214}, from=3-4, to=2-5]
	\arrow["0"', color={rgb,255:red,214;green,92;blue,214}, from=3-4, to=4-3]
	\arrow[no head, from=3-4, to=4-5]
	\arrow["c"{description}, shift left=2, from=3-4, to=5-3]
	\arrow[no head, from=3-4, to=6-4]
	\arrow["{{{{{{{\mu_{x_1}}}}}}}}", maps to, from=3-6, to=3-7]
	\arrow[no head, from=3-10, to=4-11]
	\arrow[no head, from=3-10, to=5-9]
	\arrow[no head, from=3-10, to=5-11]
	\arrow[no head, from=3-10, to=6-10]
	\arrow[no head, from=4-1, to=4-3]
	\arrow[dashed, no head, from=4-1, to=5-1]
	\arrow[no head, from=4-1, to=5-3]
	\arrow[no head, from=4-1, to=6-2]
	\arrow[no head, from=4-3, to=4-5]
	\arrow[no head, from=4-3, to=5-5]
	\arrow[no head, from=4-3, to=6-4]
	\arrow[dashed, no head, from=4-5, to=5-5]
	\arrow[no head, from=4-5, to=6-4]
	\arrow[no head, from=4-9, to=3-10]
	\arrow[no head, from=4-9, to=4-11]
	\arrow[dashed, no head, from=4-9, to=5-9]
	\arrow[no head, from=4-9, to=5-11]
	\arrow[no head, from=4-9, to=6-10]
	\arrow["0"{description}, from=4-11, to=4-13]
	\arrow[no head, from=4-13, to=2-11]
	\arrow["{\small{c-1}}"{description}, from=4-13, to=5-11]
	\arrow[no head, from=4-13, to=5-14]
	\arrow[no head, from=4-13, to=6-12]
	\arrow[no head, from=4-13, to=6-13]
	\arrow["0"', from=4-14, to=4-13]
	\arrow[no head, from=4-14, to=5-11]
	\arrow[no head, from=4-14, to=5-14]
	\arrow[no head, from=4-14, to=6-12]
	\arrow[no head, from=4-14, to=6-13]
	\arrow[no head, from=5-1, to=4-3]
	\arrow[no head, from=5-1, to=5-3]
	\arrow[no head, from=5-1, to=6-2]
	\arrow[no head, from=5-3, to=4-5]
	\arrow[no head, from=5-3, to=5-5]
	\arrow[no head, from=5-3, to=6-2]
	\arrow[no head, from=5-3, to=6-4]
	\arrow[no head, from=5-5, to=3-4]
	\arrow[no head, from=5-9, to=4-11]
	\arrow[no head, from=5-9, to=5-11]
	\arrow[no head, from=5-9, to=6-10]
	\arrow["{\small{m-c}}"{description}, from=5-11, to=4-11]
	\arrow[no head, from=5-11, to=5-14]
	\arrow[no head, from=5-11, to=6-12]
	\arrow[no head, from=5-11, to=6-13]
	\arrow[no head, from=5-14, to=6-12]
	\arrow[dashed, no head, from=5-14, to=6-13]
	\arrow[no head, from=6-2, to=4-3]
	\arrow[no head, from=6-4, to=5-5]
	\arrow[no head, from=6-10, to=5-11]
	\arrow[no head, from=6-12, to=6-13]
	\arrow[draw=none, from=6-13, to=6-12]
\end{tikzcd}}

It is evident that the central cycle $(ax_1b)$ in $Q'$ possesses coloration  $\kappa(ax_1b)=0 + (c-1) + (m-c)= m-1$. Given that $r+k \leq m+1$, it follows that if $k \geq 1$, then $r\leq m$; consequently, the size of the new complete subquiver $Q'[b, x_1, \cdots, x_r, z_t]$ is bounded by $m+2$. Analogously, since $r \geq 1$, it follows that $k \leq m$; consequently, the size of the new complete subquiver $Q'[a,b, y_1, \dots, y_k]$ is also bounded by $m+2$. Moreover, the size of the complete $Q'[a, x_1, z_1, \dots, z_{t-1}]$ remains identical to that of $Q[x_1, z_1, \dots, z_t]$. \\

 \item [(c)] There are two arrows $x_1\rightarrow a$ and $x_1\rightarrow b$ of colour $0$. This is only possible if $k = 0$, implying $r \leq m + 1$. Note that $t$ cannot be $m + 1$, as this would necessitate a third arrow of colour $0$ toward a vertex $z_i$. Given our assumption of only two such arrows, we conclude that $t \leq m$. The mutation $\mu_{x_1}$ acts as illustrated in the
following figure:\\

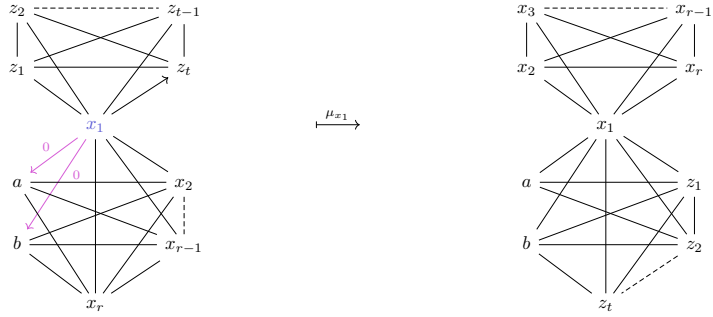
\adjustbox{scale=.7,center}{
\begin{tikzcd}
	{z_2} && {z_{t-1}} &&&&&& {x_3} && {x_{r-1}} \\
	{z_1} && {z_t} &&&&&& {x_2} && {x_r} \\
	& \textcolor{rgb,255:red,92;green,92;blue,214}{{{x_1}}} &&& {} & {} &&&& {x_1} \\
	a && {x_2} &&&&&& a && {z_1} \\
	b && {x_{r-1}} &&&&&& b && {         z_2} \\
	& {x_r} &&&&&&&& {z_t} \\
	&& {}
	\arrow[dashed, no head, from=1-1, to=1-3]
	\arrow[no head, from=1-1, to=2-1]
	\arrow[no head, from=1-1, to=2-3]
	\arrow[no head, from=1-1, to=3-2]
	\arrow[no head, from=1-3, to=2-3]
	\arrow[no head, from=1-3, to=3-2]
	\arrow[dashed, no head, from=1-9, to=1-11]
	\arrow[no head, from=1-9, to=2-9]
	\arrow[draw=none, from=1-9, to=2-11]
	\arrow[no head, from=1-9, to=2-11]
	\arrow[no head, from=1-9, to=3-10]
	\arrow[draw=none, from=1-9, to=4-9]
	\arrow[no head, from=1-11, to=2-9]
	\arrow[no head, from=1-11, to=2-11]
	\arrow[no head, from=1-11, to=3-10]
	\arrow[no head, from=2-1, to=1-3]
	\arrow[no head, from=2-1, to=2-3]
	\arrow[no head, from=2-1, to=3-2]
	\arrow[no head, from=2-9, to=2-11]
	\arrow[no head, from=2-11, to=3-10]
	\arrow[from=3-2, to=2-3]
	\arrow["0"', color={rgb,255:red,214;green,92;blue,214}, from=3-2, to=4-1]
	\arrow[no head, from=3-2, to=4-3]
	\arrow["0"{pos=0.3}, color={rgb,255:red,214;green,92;blue,214}, from=3-2, to=5-1]
	\arrow[no head, from=3-2, to=6-2]
	\arrow["{{{{{{\mu_{x_1}}}}}}}", maps to, from=3-5, to=3-6]
	\arrow[no head, from=3-10, to=2-9]
	\arrow[no head, from=3-10, to=4-11]
	\arrow[no head, from=3-10, to=5-9]
	\arrow[no head, from=3-10, to=5-11]
	\arrow[no head, from=3-10, to=6-10]
	\arrow[no head, from=4-1, to=4-3]
	\arrow[no head, from=4-1, to=5-3]
	\arrow[no head, from=4-1, to=6-2]
	\arrow[dashed, no head, from=4-3, to=5-3]
	\arrow[no head, from=4-3, to=6-2]
	\arrow[no head, from=4-9, to=3-10]
	\arrow[no head, from=4-9, to=4-11]
	\arrow[no head, from=4-9, to=5-11]
	\arrow[no head, from=4-11, to=5-9]
	\arrow[no head, from=4-11, to=5-11]
	\arrow[no head, from=4-11, to=6-10]
	\arrow[no head, from=5-1, to=4-3]
	\arrow[no head, from=5-1, to=5-3]
	\arrow[no head, from=5-1, to=6-2]
	\arrow[no head, from=5-3, to=3-2]
	\arrow[no head, from=5-9, to=5-11]
	\arrow[no head, from=5-9, to=6-10]
	\arrow[dashed, no head, from=5-11, to=6-10]
	\arrow[no head, from=6-2, to=5-3]
\end{tikzcd}}

It follows that the size of the new complete subquiver $Q'[x_1, \dots, x_r]$ is at most $m+2$, while the new $[a,b]$-quasi-complete ${}^aQ'^b[x_1, z_1,\cdots,z_t]$ reaches a size of no more than $m+3$.\\

\end{itemize}

Lastly, we address the possibility of three arrows  of colour $0$ emanating from $x_1$. Such a scenario arises only when $k=0$ and $t, r \leq m+1$. The  resulting action of the mutation $\mu_{x_1}$ is depicted in the following figure.

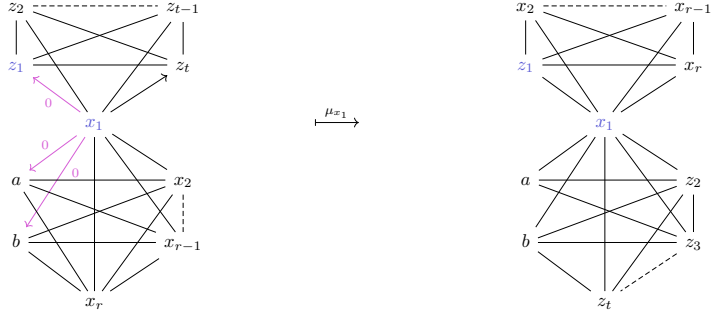
\adjustbox{scale=.7,center}{
\begin{tikzcd}
	{z_2} && {z_{t-1}} &&&&&& {x_2} && {x_{r-1}} \\
	\textcolor{rgb,255:red,92;green,92;blue,214}{{z_1}} && {z_t} &&&&&& \textcolor{rgb,255:red,92;green,92;blue,214}{{z_1}} && {x_r} \\
	& \textcolor{rgb,255:red,92;green,92;blue,214}{{{{x_1}}}} &&& {} & {} &&&& \textcolor{rgb,255:red,92;green,92;blue,214}{{x_1}} \\
	a && {x_2} &&&&&& a && {z_2} \\
	b && {x_{r-1}} &&&&&& b && {         z_3} \\
	& {x_r} &&&&&&&& {z_t} \\
	&& {}
	\arrow[dashed, no head, from=1-1, to=1-3]
	\arrow[no head, from=1-1, to=2-1]
	\arrow[no head, from=1-1, to=2-3]
	\arrow[no head, from=1-1, to=3-2]
	\arrow[no head, from=1-3, to=2-3]
	\arrow[no head, from=1-3, to=3-2]
	\arrow[dashed, no head, from=1-9, to=1-11]
	\arrow[no head, from=1-9, to=2-9]
	\arrow[draw=none, from=1-9, to=2-11]
	\arrow[no head, from=1-9, to=2-11]
	\arrow[no head, from=1-9, to=3-10]
	\arrow[draw=none, from=1-9, to=4-9]
	\arrow[no head, from=1-11, to=2-9]
	\arrow[no head, from=1-11, to=2-11]
	\arrow[no head, from=1-11, to=3-10]
	\arrow[no head, from=2-1, to=1-3]
	\arrow[no head, from=2-1, to=2-3]
	\arrow[no head, from=2-9, to=2-11]
	\arrow[no head, from=2-11, to=3-10]
	\arrow["0", color={rgb,255:red,214;green,92;blue,214}, from=3-2, to=2-1]
	\arrow[from=3-2, to=2-3]
	\arrow["0"', color={rgb,255:red,214;green,92;blue,214}, from=3-2, to=4-1]
	\arrow[no head, from=3-2, to=4-3]
	\arrow["0"{pos=0.3}, color={rgb,255:red,214;green,92;blue,214}, from=3-2, to=5-1]
	\arrow[no head, from=3-2, to=6-2]
	\arrow["{{{{{{{\mu_{x_1}}}}}}}}", maps to, from=3-5, to=3-6]
	\arrow[no head, from=3-10, to=2-9]
	\arrow[no head, from=3-10, to=4-11]
	\arrow[no head, from=3-10, to=5-9]
	\arrow[no head, from=3-10, to=5-11]
	\arrow[no head, from=3-10, to=6-10]
	\arrow[no head, from=4-1, to=4-3]
	\arrow[no head, from=4-1, to=5-3]
	\arrow[no head, from=4-1, to=6-2]
	\arrow[dashed, no head, from=4-3, to=5-3]
	\arrow[no head, from=4-3, to=6-2]
	\arrow[no head, from=4-9, to=3-10]
	\arrow[no head, from=4-9, to=4-11]
	\arrow[no head, from=4-9, to=5-11]
	\arrow[no head, from=4-11, to=5-9]
	\arrow[no head, from=4-11, to=5-11]
	\arrow[no head, from=4-11, to=6-10]
	\arrow[no head, from=5-1, to=4-3]
	\arrow[no head, from=5-1, to=5-3]
	\arrow[no head, from=5-1, to=6-2]
	\arrow[no head, from=5-3, to=3-2]
	\arrow[no head, from=5-9, to=5-11]
	\arrow[no head, from=5-9, to=6-10]
	\arrow[dashed, no head, from=5-11, to=6-10]
	\arrow[no head, from=6-2, to=5-3]
\end{tikzcd}}

 It follows that the size of the new complete subquiver $Q'[x_1, \dots, x_r, z_1]$ is at most $m+2$, while the new $[a,b]$-quasi-complete ${}^aQ'^b[x_1, z_2,\cdots,z_t]$ reaches a size of no more than $m+3$.\\


(3) \emph{$Q$ is a quiver of type II and $v=v_r$ a vertex such that there is an arrow $v_r\rightarrow x_1$, for some vertex $x_1$ in the central cycle}: Assume that the weight of the central cycle $(x_1x_2\cdots x_k)$ is $m-1$. In this case, the mutation $\mu_{v_r}$ acts as illustrated in the following figure:

    \adjustbox{scale=.7,center}
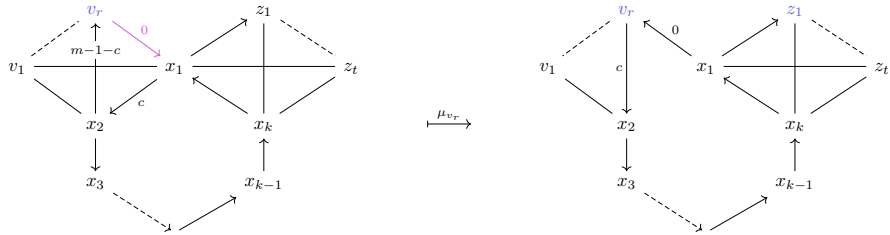
{
    \begin{tikzcd}
	& \textcolor{rgb,255:red,92;green,92;blue,214}{{v_r}} && {{z_1}} &&&&& \textcolor{rgb,255:red,92;green,92;blue,214}{{v_r}} && \textcolor{rgb,255:red,92;green,92;blue,214}{{{z_1}}} & \\
	{v_1} && {{{{{x_1}}}}} && {z_t} &&& {v_1} && {{x_1}} && {z_t} \\
	& {x_2} && {x_k} && {} & {} && {x_2} && {x_k} \\
	& {x_3} && {x_{k-1}} &&&&& {x_3} && {        x_{k-1}} \\
	&& {} &&&&&&& {} \\
	&&& {}
	\arrow["{\small{0}}", color={rgb,255:red,214;green,92;blue,214}, from=1-2, to=2-3]
	\arrow[dashed, no head, from=1-4, to=2-5]
	\arrow[no head, from=1-4, to=3-4]
	\arrow[dashed, no head, from=1-9, to=2-8]
	\arrow["{\small{c}}"', from=1-9, to=3-9]
	\arrow[dashed, no head, from=1-11, to=2-12]
	\arrow[no head, from=1-11, to=3-11]
	\arrow[dashed, no head, from=2-1, to=1-2]
	\arrow[no head, from=2-1, to=2-3]
	\arrow[from=2-3, to=1-4]
	\arrow[no head, from=2-3, to=2-5]
	\arrow["c", from=2-3, to=3-2]
	\arrow[no head, from=2-5, to=3-4]
	\arrow[no head, from=2-8, to=3-9]
	\arrow["{\small{0}}"', from=2-10, to=1-9]
	\arrow[from=2-10, to=1-11]
	\arrow[no head, from=2-10, to=2-12]
	\arrow[no head, from=2-12, to=3-11]
	\arrow["{\small{m-1-c}}"{description, pos=0.7}, from=3-2, to=1-2]
	\arrow[no head, from=3-2, to=2-1]
	\arrow[from=3-2, to=4-2]
	\arrow[from=3-4, to=2-3]
	\arrow["{{{{{{{{{{\mu_{v_r}}}}}}}}}}}", maps to, from=3-6, to=3-7]
	\arrow[from=3-9, to=4-9]
	\arrow[draw=none, from=3-11, to=2-10]
	\arrow[from=3-11, to=2-10]
	\arrow[dashed, from=4-2, to=5-3]
	\arrow[from=4-4, to=3-4]
	\arrow[dashed, from=4-9, to=5-10]
	\arrow[from=4-11, to=3-11]
	\arrow[from=5-3, to=4-4]
	\arrow[from=5-10, to=4-11]
\end{tikzcd}}

The mutation $\mu_{v_r}$ transforms the central cycle $(x_1 \ldots x_k)$ into the  bigger central cycle $(x_1 v_r x_2\ldots x_k)$ in $Q^\prime$ also with weight $m-1$.  Clearly, the complete subquiver $Q[x_1,x_2,v_1, \cdots, v_r]$ changes to a smaller complete $Q'[x_2,v_1, \cdots, v_r]$ and the complete subquiver $Q[x_1,z_1, \cdots, z_t, x_k]$ remains unchanged. \\

(4) \emph{$Q$ is a quiver of type II and $v=x_1$ a vertex in the central cycle}:  We once again face two scenarios, depending on whether $x_1$ has one or two outgoing arrows of colour $0$. Starting with the former, suppose there is a single colour $0$ arrow  departing from $x_1$.  When the arrow is a central one and $k > 3$ the following situation arises:\\

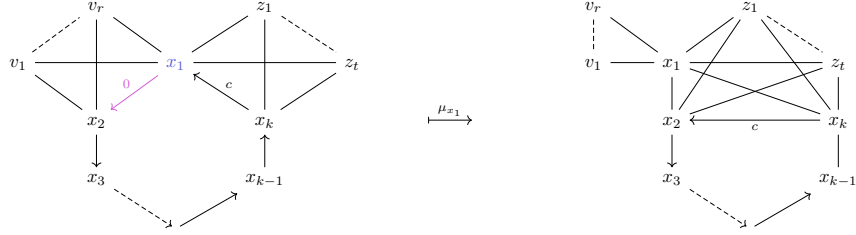
\adjustbox{scale=.7,center}{
\begin{tikzcd}
	& {v_r} && {z_1} &&&&& {v_r} && {z_1} \\
	{v_1} && \textcolor{rgb,255:red,92;green,92;blue,214}{{{{{x_1}}}}} && {z_t} &&&& {v_1} & {x_1} && {z_t} \\
	& {x_2} && {x_k} && {} & {} &&& {x_2} && {x_k} \\
	& {x_3} && {x_{k-1}} &&&&&& {x_3} && {        x_{k-1}} \\
	&& {} &&&&&&&& {} \\
	&&& {}
	\arrow[no head, from=1-2, to=2-3]
	\arrow[no head, from=1-2, to=3-2]
	\arrow[no head, from=1-4, to=2-3]
	\arrow[dashed, no head, from=1-4, to=2-5]
	\arrow[no head, from=1-4, to=3-4]
	\arrow[dashed, no head, from=1-9, to=2-9]
	\arrow[no head, from=1-9, to=2-10]
	\arrow[no head, from=1-11, to=2-10]
	\arrow[dashed, no head, from=1-11, to=2-12]
	\arrow[no head, from=1-11, to=3-10]
	\arrow[no head, from=1-11, to=3-12]
	\arrow[dashed, no head, from=2-1, to=1-2]
	\arrow[no head, from=2-1, to=2-3]
	\arrow[no head, from=2-3, to=2-5]
	\arrow["0"', color={rgb,255:red,214;green,92;blue,214}, from=2-3, to=3-2]
	\arrow[no head, from=2-5, to=3-4]
	\arrow[no head, from=2-9, to=2-10]
	\arrow[no head, from=2-10, to=2-12]
	\arrow[no head, from=2-10, to=3-10]
	\arrow[no head, from=2-12, to=3-12]
	\arrow[no head, from=3-2, to=2-1]
	\arrow[from=3-2, to=4-2]
	\arrow["c"', from=3-4, to=2-3]
	\arrow["{{{{{{{{\mu_{x_1}}}}}}}}}", maps to, from=3-6, to=3-7]
	\arrow[no head, from=3-10, to=2-12]
	\arrow[from=3-10, to=4-10]
	\arrow[no head, from=3-12, to=2-10]
	\arrow["c", from=3-12, to=3-10]
	\arrow[no head, from=3-12, to=4-12]
	\arrow[dashed, from=4-2, to=5-3]
	\arrow[from=4-4, to=3-4]
	\arrow[dashed, from=4-10, to=5-11]
	\arrow[from=5-3, to=4-4]
	\arrow[from=5-11, to=4-12]
\end{tikzcd}}

In this case, since there is a single colour $0$ arrow  departing from $x_1$ the size of the complete subquiver $Q[x_1,z_1,\dots,z_t,x_k]$ must be at most $m+1$. Therefore,  the complete subquiver $Q'[x_1,z_1,\dots,z_t,x_k, x_2]$ has size at most $m+2$. In addition, the central $k$-cycle $(x_1\dots x_k)$  it is transformed into a central $(k-1)$-cycle $(x_2\dots x_k)$ with the same weight. \\

In contrast, for the case where the arrow is a central one and $k = 3$, the situation is as follows:\\

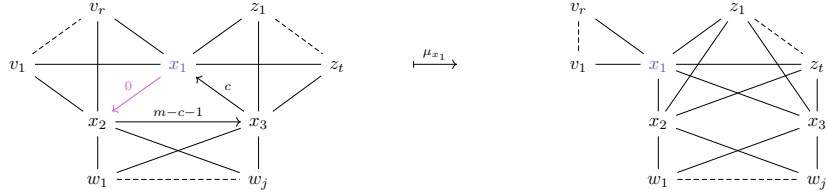
\adjustbox{scale=.7,center}{
\begin{tikzcd}
	& {v_r} && {z_1} &&&&& {v_r} && {z_1} \\
	{v_1} && \textcolor{rgb,255:red,92;green,92;blue,214}{{{{{x_1}}}}} && {z_t} & {} & {} && {v_1} & \textcolor{rgb,255:red,92;green,92;blue,214}{{x_1}} && {z_t} \\
	& {x_2} && {x_3} &&&&&& {x_2} && {x_3} \\
	& {w_1} && {w_j} &&&&&& {w_1} && {w_j} \\
	&& {} &&&&&&&& {} \\
	&&& {}
	\arrow[no head, from=1-2, to=2-3]
	\arrow[no head, from=1-2, to=3-2]
	\arrow[dashed, no head, from=1-4, to=2-5]
	\arrow[no head, from=1-4, to=3-4]
	\arrow[dashed, no head, from=1-9, to=2-9]
	\arrow[no head, from=1-9, to=2-10]
	\arrow[no head, from=1-11, to=2-10]
	\arrow[dashed, no head, from=1-11, to=2-12]
	\arrow[no head, from=1-11, to=3-10]
	\arrow[no head, from=1-11, to=3-12]
	\arrow[dashed, no head, from=2-1, to=1-2]
	\arrow[no head, from=2-1, to=2-3]
	\arrow[no head, from=2-3, to=1-4]
	\arrow[no head, from=2-3, to=2-5]
	\arrow["0"', color={rgb,255:red,214;green,92;blue,214}, from=2-3, to=3-2]
	\arrow[no head, from=2-5, to=3-4]
	\arrow["{{{{{{{{\mu_{x_1}}}}}}}}}", maps to, from=2-6, to=2-7]
	\arrow[no head, from=2-9, to=2-10]
	\arrow[no head, from=2-10, to=2-12]
	\arrow[no head, from=2-10, to=3-10]
	\arrow[no head, from=2-12, to=3-12]
	\arrow[no head, from=3-2, to=2-1]
	\arrow["{m-c-1}", from=3-2, to=3-4]
	\arrow[no head, from=3-2, to=4-2]
	\arrow[no head, from=3-2, to=4-4]
	\arrow["c"', from=3-4, to=2-3]
	\arrow[no head, from=3-10, to=2-12]
	\arrow[no head, from=3-10, to=4-10]
	\arrow[no head, from=3-10, to=4-12]
	\arrow[no head, from=3-12, to=2-10]
	\arrow[no head, from=3-12, to=4-12]
	\arrow[no head, from=4-2, to=3-4]
	\arrow[dashed, no head, from=4-2, to=4-4]
	\arrow[no head, from=4-4, to=3-4]
	\arrow[no head, from=4-10, to=3-12]
	\arrow[dashed, no head, from=4-10, to=4-12]
\end{tikzcd}}

Since $c \leq m-1$ and no arrow of colour $0$ can start at $x_1$, condition (5) of Definition \ref{la clase} implies that the number of  possible colours for arrows starting at $x_1$ is bounded by  $m-1-c$. It follows that $t \leq m-c-1$. Similarly, as an arrow of colour $0$ is permitted to start at $x_3$, we have $j \leq m - (m-c-1) = c+1$, which yields $t+j \leq m$. Consequently, the new quasi-complete subquivers ${}^{x_2}Q'^{x_3}[x_1, z_1, \dots, z_t]$ and ${}^{x_2}Q'^{x_3}[w_1, \dots, w_j]$ satisfy $j+(t+1) \leq m+1$, as required.\\

If, on the contrary, the $0$-coloured arrow is not central, the situation changes as follows:\\

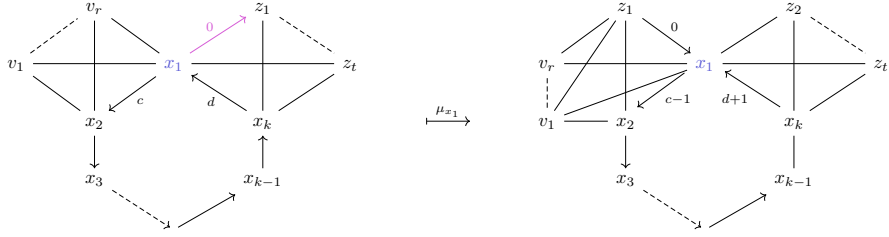
\adjustbox{scale=.7,center}{
\begin{tikzcd}
	& {v_r} && {z_1} &&&&& {z_1} && {z_2} & \\
	{v_1} && \textcolor{rgb,255:red,92;green,92;blue,214}{{{{{{{x_1}}}}}}} && {z_t} &&& {v_r} && \textcolor{rgb,255:red,92;green,92;blue,214}{{{{x_1}}}} && {z_t} \\
	& {x_2} && {x_k} && {} & {} & {v_1} & {x_2} && {x_k} \\
	& {x_3} && {x_{k-1}} &&&&& {x_3} && {        x_{k-1}} \\
	&& {} &&&&&&& {} \\
	&&& {}
	\arrow[no head, from=1-2, to=2-3]
	\arrow[no head, from=1-2, to=3-2]
	\arrow[dashed, no head, from=1-4, to=2-5]
	\arrow[no head, from=1-4, to=3-4]
	\arrow[no head, from=1-9, to=2-8]
	\arrow["0", from=1-9, to=2-10]
	\arrow[no head, from=1-9, to=3-8]
	\arrow[no head, from=1-9, to=3-9]
	\arrow[no head, from=1-11, to=2-10]
	\arrow[dashed, no head, from=1-11, to=2-12]
	\arrow[draw=none, from=1-11, to=3-9]
	\arrow[no head, from=1-11, to=3-11]
	\arrow[dashed, no head, from=2-1, to=1-2]
	\arrow[no head, from=2-1, to=2-3]
	\arrow["0", color={rgb,255:red,214;green,92;blue,214}, from=2-3, to=1-4]
	\arrow[no head, from=2-3, to=2-5]
	\arrow["c", from=2-3, to=3-2]
	\arrow[no head, from=2-5, to=3-4]
	\arrow[no head, from=2-8, to=2-10]
	\arrow[dashed, no head, from=2-8, to=3-8]
	\arrow[no head, from=2-10, to=2-12]
	\arrow["{ \small{c-1}}", shift right, from=2-10, to=3-9]
	\arrow[no head, from=2-12, to=3-11]
	\arrow[no head, from=3-2, to=2-1]
	\arrow[from=3-2, to=4-2]
	\arrow["d", from=3-4, to=2-3]
	\arrow["{{{{{{{{{{\mu_{x_1}}}}}}}}}}}", maps to, from=3-6, to=3-7]
	\arrow[no head, from=3-8, to=2-10]
	\arrow[no head, from=3-8, to=3-9]
	\arrow[from=3-9, to=4-9]
	\arrow[draw=none, from=3-11, to=2-10]
	\arrow["{\small{d+1}}", shift right, from=3-11, to=2-10]
	\arrow[no head, from=3-11, to=4-11]
	\arrow[dashed, from=4-2, to=5-3]
	\arrow[from=4-4, to=3-4]
	\arrow[dashed, from=4-9, to=5-10]
	\arrow[from=5-3, to=4-4]
	\arrow[from=5-10, to=4-11]
\end{tikzcd}}

Clearly, the central cycle preserves both its size and its colouration. As in the previous case, the sizes of the complete subquivers involved continue to satisfy the required conditions.\\

Next, we consider the situation where two colour $0$ arrows emanate from $x_1$, neither of which is central.

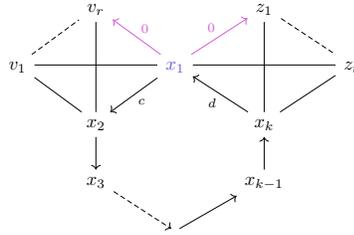
\adjustbox{scale=.7,center}{
\begin{tikzcd}
	& {v_r} && {z_1} && \\
	{v_1} && \textcolor{rgb,255:red,92;green,92;blue,214}{{{{{{{x_1}}}}}}} && {z_t} \\
	& {x_2} && {x_k} && {} \\
	& {x_3} && {x_{k-1}} \\
	&& {} \\
	&&& {}
	\arrow[no head, from=1-2, to=3-2]
	\arrow[dashed, no head, from=1-4, to=2-5]
	\arrow[no head, from=1-4, to=3-4]
	\arrow[dashed, no head, from=2-1, to=1-2]
	\arrow[no head, from=2-1, to=2-3]
	\arrow["0"', color={rgb,255:red,214;green,92;blue,214}, from=2-3, to=1-2]
	\arrow["0", color={rgb,255:red,214;green,92;blue,214}, from=2-3, to=1-4]
	\arrow[no head, from=2-3, to=2-5]
	\arrow["c", from=2-3, to=3-2]
	\arrow[no head, from=2-5, to=3-4]
	\arrow[no head, from=3-2, to=2-1]
	\arrow[from=3-2, to=4-2]
	\arrow["d", from=3-4, to=2-3]
	\arrow[dashed, from=4-2, to=5-3]
	\arrow[from=4-4, to=3-4]
	\arrow[from=5-3, to=4-4]
\end{tikzcd}}

In view of condition (5), the cycles $(x_1 x_2 v_i)$ must have a weight of $m-1$; this requirement precludes the existence of a colour $0$ arrow from $x_1$ to some $v_i$. Therefore, this case is impossible. \\

We now turn to the scenario where $x_1$ has two outgoing colour $0$ arrows, one of which is central, under the assumption that $k > 3$.

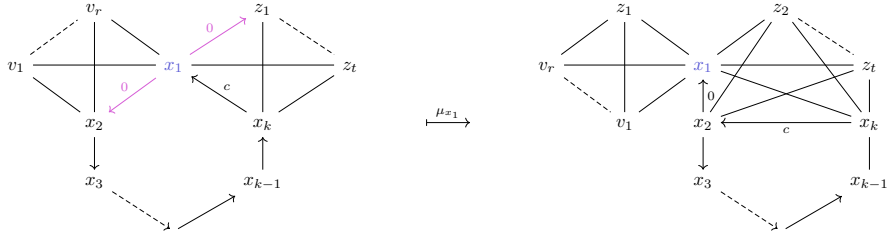
\adjustbox{scale=.7,center}{
\begin{tikzcd}
	& {v_r} && {z_1} &&&&& {z_1} && {z_2} & \\
	{v_1} && \textcolor{rgb,255:red,92;green,92;blue,214}{{{{{{x_1}}}}}} && {z_t} &&& {v_r} && \textcolor{rgb,255:red,92;green,92;blue,214}{{{x_1}}} && {z_t} \\
	& {x_2} && {x_k} && {} & {} && {v_1} & {x_2} && {x_k} \\
	& {x_3} && {x_{k-1}} &&&&&& {x_3} && {        x_{k-1}} \\
	&& {} &&&&&&&& {} \\
	&&& {}
	\arrow[no head, from=1-2, to=2-3]
	\arrow[no head, from=1-2, to=3-2]
	\arrow[dashed, no head, from=1-4, to=2-5]
	\arrow[no head, from=1-4, to=3-4]
	\arrow[no head, from=1-9, to=2-8]
	\arrow[no head, from=1-9, to=2-10]
	\arrow[no head, from=1-9, to=3-9]
	\arrow[no head, from=1-11, to=2-10]
	\arrow[dashed, no head, from=1-11, to=2-12]
	\arrow[no head, from=1-11, to=3-10]
	\arrow[no head, from=1-11, to=3-12]
	\arrow[dashed, no head, from=2-1, to=1-2]
	\arrow[no head, from=2-1, to=2-3]
	\arrow["0", color={rgb,255:red,214;green,92;blue,214}, from=2-3, to=1-4]
	\arrow[no head, from=2-3, to=2-5]
	\arrow["0"', color={rgb,255:red,214;green,92;blue,214}, from=2-3, to=3-2]
	\arrow[no head, from=2-5, to=3-4]
	\arrow[no head, from=2-8, to=2-10]
	\arrow[dashed, no head, from=2-8, to=3-9]
	\arrow[no head, from=2-10, to=2-12]
	\arrow[no head, from=2-12, to=3-12]
	\arrow[no head, from=3-2, to=2-1]
	\arrow[from=3-2, to=4-2]
	\arrow["c"', from=3-4, to=2-3]
	\arrow["{{{{{{{{{\mu_{x_1}}}}}}}}}}", maps to, from=3-6, to=3-7]
	\arrow[no head, from=3-9, to=2-10]
	\arrow["0"', from=3-10, to=2-10]
	\arrow[no head, from=3-10, to=2-12]
	\arrow[from=3-10, to=4-10]
	\arrow[no head, from=3-12, to=2-10]
	\arrow["c", from=3-12, to=3-10]
	\arrow[no head, from=3-12, to=4-12]
	\arrow[dashed, from=4-2, to=5-3]
	\arrow[from=4-4, to=3-4]
	\arrow[dashed, from=4-10, to=5-11]
	\arrow[from=5-3, to=4-4]
	\arrow[from=5-11, to=4-12]
\end{tikzcd}}

After mutating at $x_1$, the vertices $x_2$ and $z_1$ exchange their positions, implying that the sizes of the complete subquivers involved remain unchanged. The central $k$-cycle $(x_1\dots x_k)$  it is transformed into a central $(k-1)$-cycle $(x_2\dots x_k)$ with the same weight. \\

Finally, we address the scenario where $x_1$ has two outgoing arrows of colour $0$, one of which is central and $k=3$. The mutation $\mu_{x_1}$, in this case, acts as illustrated in the
following figure by transforming the type II quiver $Q$ into the type I quiver $Q'$: \\

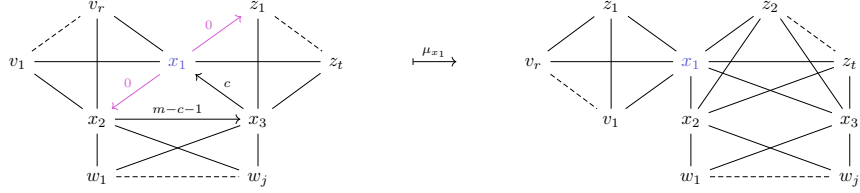
\adjustbox{scale=.7,center}{
\begin{tikzcd}
	& {v_r} && {z_1} &&&&& {z_1} && {z_2} \\
	{v_1} && \textcolor{rgb,255:red,92;green,92;blue,214}{{{{{x_1}}}}} && {z_t} & {} & {} & {v_r} && \textcolor{rgb,255:red,92;green,92;blue,214}{{x_1}} && {z_t} \\
	& {x_2} && {x_3} &&&&& {v_1} & {x_2} && {x_3} \\
	& {w_1} && {w_j} &&&&&& {w_1} && {w_j} \\
	&& {} &&&&&&&& {} \\
	&&& {}
	\arrow[no head, from=1-2, to=2-3]
	\arrow[no head, from=1-2, to=3-2]
	\arrow[dashed, no head, from=1-4, to=2-5]
	\arrow[no head, from=1-4, to=3-4]
	\arrow[no head, from=1-9, to=2-8]
	\arrow[no head, from=1-9, to=2-10]
	\arrow[no head, from=1-9, to=3-9]
	\arrow[no head, from=1-11, to=2-10]
	\arrow[dashed, no head, from=1-11, to=2-12]
	\arrow[no head, from=1-11, to=3-10]
	\arrow[no head, from=1-11, to=3-12]
	\arrow[dashed, no head, from=2-1, to=1-2]
	\arrow[no head, from=2-1, to=2-3]
	\arrow["0", color={rgb,255:red,214;green,92;blue,214}, from=2-3, to=1-4]
	\arrow[no head, from=2-3, to=2-5]
	\arrow["0"', color={rgb,255:red,214;green,92;blue,214}, from=2-3, to=3-2]
	\arrow[no head, from=2-5, to=3-4]
	\arrow["{{{{{{{{\mu_{x_1}}}}}}}}}", maps to, from=2-6, to=2-7]
	\arrow[no head, from=2-8, to=2-10]
	\arrow[dashed, no head, from=2-8, to=3-9]
	\arrow[no head, from=2-10, to=2-12]
	\arrow[no head, from=2-10, to=3-10]
	\arrow[no head, from=2-12, to=3-12]
	\arrow[no head, from=3-2, to=2-1]
	\arrow["{m-c-1}", from=3-2, to=3-4]
	\arrow[no head, from=3-2, to=4-2]
	\arrow[no head, from=3-2, to=4-4]
	\arrow["c"', from=3-4, to=2-3]
	\arrow[no head, from=3-9, to=2-10]
	\arrow[no head, from=3-10, to=2-12]
	\arrow[no head, from=3-10, to=4-10]
	\arrow[no head, from=3-10, to=4-12]
	\arrow[no head, from=3-12, to=2-10]
	\arrow[no head, from=3-12, to=4-12]
	\arrow[no head, from=4-2, to=3-4]
	\arrow[dashed, no head, from=4-2, to=4-4]
	\arrow[no head, from=4-4, to=3-4]
	\arrow[no head, from=4-10, to=3-12]
	\arrow[dashed, no head, from=4-10, to=4-12]
\end{tikzcd}}

By applying condition (6) to $Q$, we obtain $j + t \leq m + 1$, as required.

\end{proof}

For $m\geq 2$ the coloured mutation is not invertible in general. In \cite{GQP}, we provided an example showing that the set of coloured quivers with invertible mutations is significantly broader than the class $\mathcal{Q}_{n,m}^A$. Having recently identified another such family in \cite{GR}, we now demonstrate that this property extends to the class $\mathcal{Q}_{n,m}^D$ as well.

\begin{lema}\label{mutacion invertible}
The coloured quiver  mutation restricted to the  class $\mathcal{Q}_{n,m}^D$ is invertible.
Moreover if $Q\in \mathcal{Q}_{n,m}^D$ and $v$ is any vertex of $Q$ we have $\mu_v^{m+1}(Q)=Q$.
\end{lema}

\begin{proof}
Let $Q$ be a quiver belonging to the class $Q_{n,m}^{D}$ and let $v$ be a vertex of $Q$. It suffices to show that $\mu_v^{\,m+1}(Q)=Q$. If $v$ has at most one neighbour, the claim is immediate. Otherwise, let $u_1,u_2$ be two neighbours of $v$, and set $Q' = Q[\{v,u_1,u_2\}]$. Since $\mu_v$ acts trivially outside the neighbourhood of $v$, it is enough to prove that $\mu_v^{\,m+1}(Q')=Q'$. By Proposition~\ref{cerrado por mutaciones}, the class $Q_{n,m}^{D}$ is closed under coloured quiver mutation, so in particular every triangle has colouring $m-1$ throughout. The conclusion now follows as in Lemma~4.8 of \cite{GQP}.
\end{proof}

\section{Proof that $\mathcal{Q}_{n,m}^D$ is the coloured mutation class of $\mathbb{D}_n$}

In the following  assume $k$ to be at least equal to $3$.\\

\begin{lema}
Let $\mathcal{O}_1$ and $\mathcal{O}_2$ be two coloured quivers with the same underlying $k$-cycle $\mathcal{O}$ with $k\geq 3$.
If $\kappa (\mathcal{O}_1) =\kappa (\mathcal{O}_2)$  then $\mathcal{O}_1$ and $\mathcal{O}_2$ are mutation equivalent.
\end{lema}

\begin{proof}
Assume that $\kappa (\mathcal{O}_1) =\kappa (\mathcal{O}_2)=c$.  The result follow from showing that for any decomposition $c=r_1+r_2+\cdots + r_k$ with $r_i\geq 0$ the following two coloured cycles are mutation equivalent.

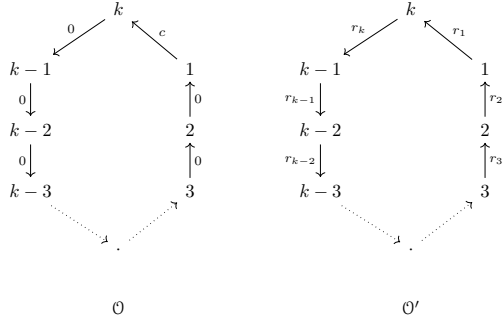
\adjustbox{scale=.7,center}{
\begin{tikzcd}
	& k &&&& k \\
	{k-1} && 1 && {k-1} && 1 \\
	{k-2} && 2 && {k-2} && 2 \\
	{k-3} && 3 && {k-3} && 3 \\
	& \cdot &&&& \cdot \\
	& \mathcal{O} &&&& \mathcal{O'}
	\arrow["0"', from=1-2, to=2-1]
	\arrow["{r_k}"', from=1-6, to=2-5]
	\arrow["0"', from=2-1, to=3-1]
	\arrow["c"', from=2-3, to=1-2]
	\arrow["{r_{k-1}}"', from=2-5, to=3-5]
	\arrow["{r_1}"', from=2-7, to=1-6]
	\arrow["0"', from=3-1, to=4-1]
	\arrow["0"', from=3-3, to=2-3]
	\arrow["{r_{k-2}}"', from=3-5, to=4-5]
	\arrow["{r_2}"', from=3-7, to=2-7]
	\arrow[dotted, from=4-1, to=5-2]
	\arrow["0"', from=4-3, to=3-3]
	\arrow[dotted, from=4-5, to=5-6]
	\arrow["{r_3}"', from=4-7, to=3-7]
	\arrow[dotted, from=5-2, to=4-3]
	\arrow[dotted, from=5-6, to=4-7]
\end{tikzcd}}

If we apply the mutation $\mu_1^{r_2+\cdots+r_k}$ to the cycle $\mathcal{O}$ we obtain the cycle $\mathcal{O}^1$

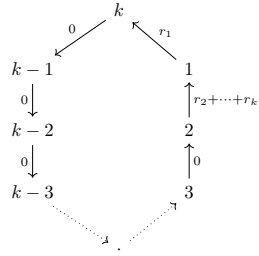
\adjustbox{scale=.7,center}{
\begin{tikzcd}
	& k && \\
	{k-1} && 1 & \\
	{k-2} && 2 & \\
	{k-3} && 3 & \\
	& \cdot &&& \\
	\arrow["0"', from=1-2, to=2-1]
	\arrow["0"', from=2-1, to=3-1]
	\arrow["r_1"', from=2-3, to=1-2]
	\arrow["0"', from=3-1, to=4-1]
	\arrow["r_2+\cdots+r_k"', from=3-3, to=2-3]
    \arrow["0"', from=4-3, to=3-3]
    \arrow[dotted, from=4-1, to=5-2]
    \arrow[dotted, from=5-2, to=4-3]
\end{tikzcd}}

Then, if we apply the mutation $\mu_2^{r_3+\cdots+r_k}$ to the cycle $\mathcal{O}^1$ we obtain the cycle $\mathcal{O}^2$

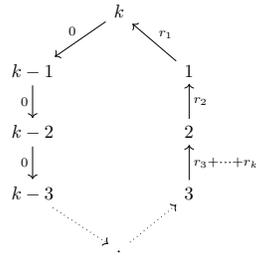
\adjustbox{scale=.7,center}{
\begin{tikzcd}
	& k && \\
	{k-1} && 1 & \\
	{k-2} && 2 & \\
	{k-3} && 3 & \\
	& \cdot &&& \\
	\arrow["0"', from=1-2, to=2-1]
	\arrow["0"', from=2-1, to=3-1]
	\arrow["r_1"', from=2-3, to=1-2]
	\arrow["0"', from=3-1, to=4-1]
	\arrow["r_2"', from=3-3, to=2-3]
    \arrow["r_3+\cdots+r_k"', from=4-3, to=3-3]
    \arrow[dotted, from=4-1, to=5-2]
    \arrow[dotted, from=5-2, to=4-3]
\end{tikzcd}}

Finally, a simple iteration of this procedure gives the desired cycle $\mathcal{O}^k=\mathcal{O}'$.

\end{proof}

\begin{obs}
In particular all $m$-coloured $k$-cycles with colouration $m-1$ are mutation equivalent.
\end{obs}

\begin{lema}
The $m$-coloured $n$-cycles with colouration $m-1$ are mutation equivalent to $\mathbb{D}_n$.
\end{lema}

\begin{proof}
It is enough to observe that we can get the $m$-coloured $n$-cycle

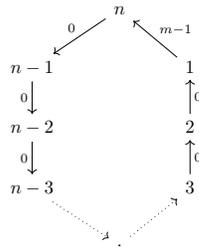
\adjustbox{scale=.7,center}{
\begin{tikzcd}
	& n && \\
	{n-1} && 1 & \\
	{n-2} && 2 & \\
	{n-3} && 3 & \\
	& \cdot &&& \\
	\arrow["0"', from=1-2, to=2-1]
	\arrow["0"', from=2-1, to=3-1]
	\arrow["m-1"', from=2-3, to=1-2]
	\arrow["0"', from=3-1, to=4-1]
	\arrow["0"', from=3-3, to=2-3]
    \arrow["0"', from=4-3, to=3-3]
    \arrow[dotted, from=4-1, to=5-2]
    \arrow[dotted, from=5-2, to=4-3]
\end{tikzcd}}

 by performing the composition  $\mu_{n-1}\mu_{1}\mu_2 \cdots \mu_{n-3}\mu_{n-2}\mu_{n-1}$  to the quiver

 $$\xymatrix@R=13pt@C=18pt{&&&&& n \ar[ld]_0\\
1  & 2 \ar[l]^0 & \cdots n-4 \ar[l]^0 & n-3 \ar[l]^0  \ar[l]^0  & n-1  \ar[l]^0   & \\
&&&&& n-2  \ar[lu]^0 }$$

\end{proof}

\begin{lema}\cite[Lemma 5.4 and Remark 5.5]{GQP}\label{coloracion del An}
All colourations of $\mathbb{A}_t$ are mutation equivalent and in fact,  the quiver $ \xymatrix{ *+[o][F]{Q'} \ar^{0}[r] & t-1 \ar^{0}[r] &  \cdots  \ar^{0}[r] & 2 \ar^{0}[r] & 1  }$ is mutation equivalent to the  quiver $ \xymatrix{ *+[o][F]{Q'} \ar^{c_{t-1}}[r] & t-1 \ar^{c_{t-2}}[r] &  \cdots  \ar^{c_2}[r] & 2 \ar^{c_1}[r] & 1  }$  where the decoration $\SelectTips{eu}{10}\xymatrix@R=.2pc@C=.8pc{ *+[o][F]{Q'}}$ means that $Q'$ can be any subquiver.
\end{lema}

\begin{lema}\label{An mas largo}
 The coloured quiver $$ \xymatrix@R=13pt@C=9pt{ v_2^k \ar@{..}[rr]  & &  v_2^1& & v_2  \ar[ll]  \ar[rd]&& v_1 \ar[dl]^{c_1} \ar[ll]_{c_2}  && v_1^1 \ar[ll] &&   v_1^t   \ar@{..}[ll]\\
 &&&&&  *+[o][F]{v}  &&&    }$$ is mutation equivalent to a coloured quiver of the form

  $$ \xymatrix@R=13pt@C=9pt{ v_2^k \ar@{..}[rr]  & &  v_2^1& & v_2  \ar[ll]  && v_1  \ar[ll]  && v_1^1 \ar[ll] &&  v_1^{t}\ar@{..}[ll]   &&  *+[o][F]{v} \ar[ll]    }$$

where the decoration $\SelectTips{eu}{10}\xymatrix@R=.2pc@C=.8pc{ *+[o][F]{v}}$ means that the vertex $v$ can be part of any subquiver.
\end{lema}

\begin{proof} This lemma is a special case of \cite[Lemma 4.3]{GR}.

\end{proof}

In the sequel if a quiver $Q$ belong to the class $\mathcal{Q}_{m}^A$ we will assume that $Q$ is an $\mathbb{A}_t$-quiver (for some value of $t$). Therefore, we  will simply write that $Q$ is an $\mathbb{A}$-quiver. \\

\begin{lema}\label{tipo II --- Dn}
All $m$-coloured quivers of type II are mutation equivalent to $\mathbb{D}_n$.
\end{lema}

\begin{proof}
Let $Q$ be a quiver of type II. By the previous lemma it is enough to show that we can get from $Q$, performing a sequence of mutations that does not change the type of the quiver,  an $m$-coloured $n$-cycle $Q'$ with colouration $m-1$; and equivalently an $m$-coloured quiver in the class $\mathcal{Q}_{n,m}^D$ (of type II) with Euler characteristic equal to one. \\

Assume that $Q$ is not a $n$-cycle, and equivalently $Q$ has a central $(r+2)$-cycle $(x_1x_2 \cdots x_r ab)$ and at least a complete subquiver $Q[a,b,y_1,\cdots ,y_k]$  (with $k \geq 1$) as in the following figure:\\

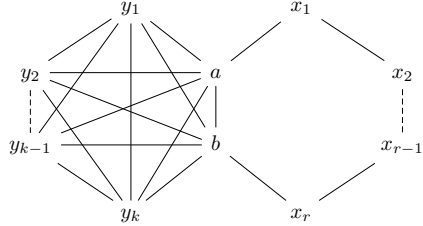
 \adjustbox{scale=.8,center}{
\begin{tikzcd}
         &&&&\\
         &&&&\\
	& {y_1} && {x_1} && \\
	{y_2} && a && {x_2} &\\
	{y_{k-1}} && b && {x_{r-1}} &\\
	& {y_k} && {x_r} && \\
	&&&& {}
	\arrow[no head, from=3-2, to=4-1]
	\arrow[no head,  from=3-2, to=4-3]
	\arrow[no head, from=3-2, to=5-1]
	\arrow[no head, from=3-2, to=5-3]
	\arrow[no head, from=3-2, to=6-2]
	\arrow[no head, from=3-4, to=4-5]
	\arrow[no head, from=4-1, to=4-3]
	\arrow[dashed, no head, from=4-1, to=5-1]
	\arrow[no head, from=4-1, to=5-3]
	\arrow[no head, from=4-1, to=6-2]
	\arrow[no head, from=4-3, to=3-4]
	\arrow[no head, from=4-3, to=5-3]
	\arrow[dashed, no head, from=4-5, to=5-5]
	\arrow[no head, from=5-1, to=4-3]
	\arrow[no head, from=5-1, to=5-3]
	\arrow[no head, from=5-1, to=6-2]
	\arrow[no head, from=5-3, to=6-2]
	\arrow[no head, from=5-3, to=6-4]
	\arrow[no head, from=6-2, to=4-3]
	\arrow[no head, from=6-4, to=5-5]
\end{tikzcd}}

Our goal is to demonstrate that it is possible to apply a sequence of mutations to $Q$, leading to a resulting quiver with a strictly decreased Euler characteristic. Given that the minimal Euler characteristic for a type II quiver is $1$, we must eventually obtain a quiver which is an $n$-cycle.\\

By \cite[Lemma 4.3]{GR} every quiver in the class $Q^A_{m}$ is mutation equivalent to an $\mathbb{A}$-quiver. Then, following the notation of Definition \ref{la clase},  we can assume that every subquiver $Q^i $, attached to the vertex $y_i$,  of $Q$ is an $\mathbb{A}$-quiver. \\

We will denote by $ \xymatrix@R=15pt@C=15pt{ y_i \ar@{-}[r] & y_i^1 \ar@{-}[r] & y_i^2 \ar@{..}[r] &  y_i^{t_i}} $  the $\mathbb{A}_{t_i+1}$-quiver  $Q^i $ attached to the vertex $y_i$.  Furthermore, by Lemma \ref{coloracion del An}, we may assume a specific and convenient colouring.\\

Let $\{\overline{y_1y_2}, \cdots, \overline{y_1y_k}, \overline{y_1a}, \overline{y_1b} \}$ the set of all the colours of the arrows going out from the vertex $y_1$; and let $c$ be the minimum  of these colours. We can assume, without loss of generality that $c=0$. Therefore,  we have two cases: either there exists an arrow of colour zero $y_1\rightarrow y_j$ ($j\neq 1$) or there exists an arrow of colour zero $y_1\rightarrow a$ (or $b$).

\begin{enumerate}

  \item \underline{There exists an arrow of colour zero $y_1\rightarrow a$}: If this is the case, we can mutate at vertex $y_1$ obtaining:

  \adjustbox{scale=.7,center}
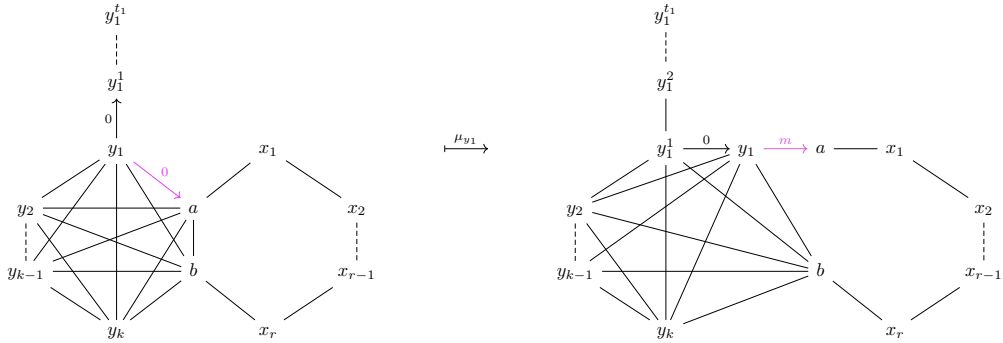
{
\begin{tikzcd}
	& {y_1^{t_1}} &&&&&&& {y_1^{t_1}} \\
	& {y_1^1} &&&&&&& {y_1^2} \\
	& {y_1} && {x_1} && {} & {} && {y_1^1} & {y_1} & a & {x_1} \\
	{y_2} && a && {x_2} &&& {y_2} &&&&& {x_2} \\
	{y_{k-1}} && b && {x_{r-1}} &&& {y_{k-1}} &&& b && {x_{r-1}} \\
	& {y_k} && {x_r} &&&&& {y_k} &&& {x_r} \\
	&&&& {}
	\arrow[dashed, no head, from=1-9, to=2-9]
	\arrow[dashed, no head, from=2-2, to=1-2]
	\arrow[no head, from=2-9, to=3-9]
	\arrow["0", from=3-2, to=2-2]
	\arrow[no head, from=3-2, to=4-1]
	\arrow["0", color={rgb,255:red,229;green,62;blue,244}, from=3-2, to=4-3]
	\arrow[no head, from=3-2, to=5-1]
	\arrow[no head, from=3-2, to=5-3]
	\arrow[no head, from=3-2, to=6-2]
	\arrow[no head, from=3-4, to=4-5]
	\arrow["{{{{{\mu_{y_1}}}}}}", maps to, from=3-6, to=3-7]
	\arrow["0", from=3-9, to=3-10]
	\arrow[no head, from=3-9, to=6-9]
	\arrow["m", color={rgb,255:red,214;green,92;blue,214}, from=3-10, to=3-11]
	\arrow[no head, from=3-10, to=5-8]
	\arrow[no head, from=3-10, to=5-11]
	\arrow[no head, from=3-10, to=6-9]
	\arrow[no head, from=3-11, to=3-12]
	\arrow[no head, from=3-12, to=4-13]
	\arrow[no head, from=4-1, to=4-3]
	\arrow[dashed, no head, from=4-1, to=5-1]
	\arrow[no head, from=4-1, to=5-3]
	\arrow[no head, from=4-1, to=6-2]
	\arrow[no head, from=4-3, to=3-4]
	\arrow[no head, from=4-3, to=5-3]
	\arrow[dashed, no head, from=4-5, to=5-5]
	\arrow[no head, from=4-8, to=3-9]
	\arrow[no head, from=4-8, to=3-10]
	\arrow[dashed, no head, from=4-8, to=5-8]
	\arrow[no head, from=4-8, to=5-11]
	\arrow[no head, from=4-8, to=6-9]
	\arrow[dashed, no head, from=4-13, to=5-13]
	\arrow[no head, from=5-1, to=4-3]
	\arrow[no head, from=5-1, to=5-3]
	\arrow[no head, from=5-1, to=6-2]
	\arrow[no head, from=5-3, to=6-2]
	\arrow[no head, from=5-3, to=6-4]
	\arrow[no head, from=5-8, to=5-11]
	\arrow[no head, from=5-8, to=6-9]
	\arrow[no head, from=5-11, to=3-9]
	\arrow[no head, from=5-11, to=6-12]
	\arrow[no head, from=6-2, to=4-3]
	\arrow[no head, from=6-4, to=5-5]
	\arrow[no head, from=6-9, to=5-11]
	\arrow[no head, from=6-12, to=5-13]
\end{tikzcd}}

In this mutated quiver the central cycle increases its length by $1$ while  the complete quiver containing the vertex  $y_1$ remains of the same size $k+2$. Consequently, the Euler characteristic is unchanged.

Then, we can assume that  $y_1^1$ has valence $k+1$ (seen in the underlying graph). Otherwise, we continue mutating at the vertices   $y_1^2, \cdots , y_1^{t_1}$ in a  procedure where the size of the  central cycle  grows without changing the Euler characteristic of the quiver. \\

Finally, a mutation at vertex $y_1^1$ increases the  length of the central cycle  by $1$ and  reduces by $1$ the size of the complete quiver containing the vertex  $y_1^1$ as in the figure below. \\


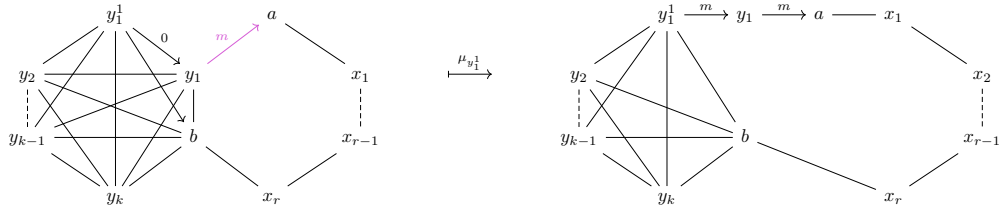
\adjustbox{scale=.7,center}{
\begin{tikzcd}
	& {y_1^1} && a && {} & {} && {y_1^1} & {y_1} & a & {x_1} \\
	{y_2} && {y_1} && {x_1} &{}&{}& {y_2} &&&&& {x_2} \\
	{y_{k-1}} && b && {x_{r-1}} &&& {y_{k-1}} && b &&& {x_{r-1}} \\
	& {y_k} && {x_r} &&&&& {y_k} &&& {x_r} \\
	&&&& {}
	\arrow[no head, from=1-2, to=2-1]
	\arrow["0", from=1-2, to=2-3]
	\arrow[no head, from=1-2, to=3-1]
	\arrow[from=1-2, to=3-3]
	\arrow[no head, from=1-2, to=4-2]
	\arrow[no head, from=1-4, to=2-5]
	\arrow["{{{{{{\mu_{y_1^1}}}}}}}", maps to, from=2-6, to=2-7]
	\arrow["m", from=1-9, to=1-10]
	\arrow[no head, from=1-9, to=3-8]
	\arrow[no head, from=1-9, to=4-9]
	\arrow["m", from=1-10, to=1-11]
	\arrow[no head, from=1-11, to=1-12]
	\arrow[no head, from=1-12, to=2-13]
	\arrow[no head, from=2-1, to=2-3]
	\arrow[dashed, no head, from=2-1, to=3-1]
	\arrow[no head, from=2-1, to=3-3]
	\arrow[no head, from=2-1, to=4-2]
	\arrow["m", color={rgb,255:red,214;green,92;blue,214}, from=2-3, to=1-4]
	\arrow[no head, from=2-3, to=3-3]
	\arrow[dashed, no head, from=2-5, to=3-5]
	\arrow[no head, from=2-8, to=1-9]
	\arrow[dashed, no head, from=2-8, to=3-8]
	\arrow[no head, from=2-8, to=3-10]
	\arrow[no head, from=2-8, to=4-9]
	\arrow[dashed, no head, from=2-13, to=3-13]
	\arrow[no head, from=3-1, to=2-3]
	\arrow[no head, from=3-1, to=3-3]
	\arrow[no head, from=3-1, to=4-2]
	\arrow[no head, from=3-3, to=4-2]
	\arrow[no head, from=3-3, to=4-4]
	\arrow[no head, from=3-8, to=3-10]
	\arrow[no head, from=3-8, to=4-9]
	\arrow[no head, from=3-10, to=1-9]
	\arrow[no head, from=3-10, to=4-12]
	\arrow[no head, from=4-2, to=2-3]
	\arrow[no head, from=4-4, to=3-5]
	\arrow[no head, from=4-9, to=3-10]
	\arrow[no head, from=4-12, to=3-13]
\end{tikzcd}}

Therefore, the value of the Euler characteristic  was decreased by  $k$.  \\

  \item \underline{There exists an arrow of colour zero $y_1\rightarrow y_j$}: Observe that we have this case if $k>1$ and we can assume that $j=2$. Therefore, we can mutate at $y_1$ obtaining the following:

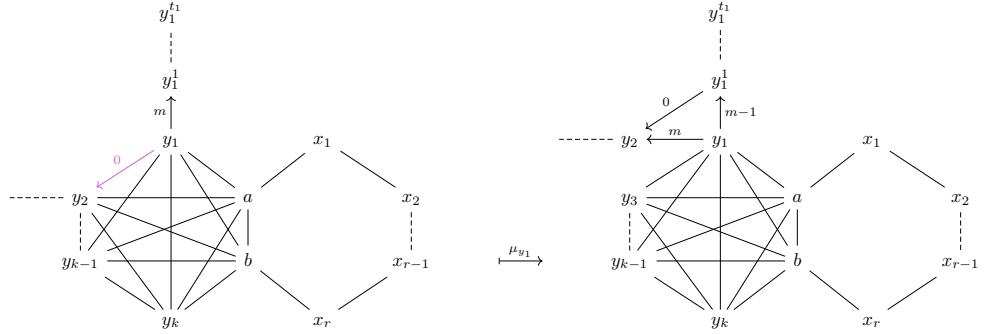
\adjustbox{scale=.7,center}{
\begin{tikzcd}
	&& {y_1^{t_1}} &&&&&&& {y_1^{t_1}} \\
	&& {y_1^1} &&&&&&& {y_1^1} \\
	&& {y_1} && {x_1} &&& {} & {y_2} & {y_1} && {x_1} \\
	{} & {y_2} && a && {x_2} &&& {y_3} && a && {x_2} \\
	& {y_{k-1}} && b && {x_{r-1}} & {} & {} & {y_{k-1}} && b && {x_{r-1}} \\
	&& {y_k} && {x_r} &&&&& {y_k} && {x_r} \\
	&&&&& {}
	\arrow[dashed, no head, from=1-10, to=2-10]
	\arrow[dashed, no head, from=2-3, to=1-3]
	\arrow["0"', from=2-10, to=3-9]
	\arrow["m", from=3-3, to=2-3]
	\arrow["0"', color={rgb,255:red,214;green,92;blue,214},  from=3-3, to=4-2]
	\arrow[no head, from=3-3, to=4-4]
	\arrow[no head, from=3-3, to=5-2]
	\arrow[no head, from=3-3, to=5-4]
	\arrow[no head, from=3-3, to=6-3]
	\arrow[no head, from=3-5, to=4-6]
	\arrow[dashed, no head, from=3-9, to=3-8]
	\arrow["{{{{m-1}}}}"', from=3-10, to=2-10]
	\arrow["m"', from=3-10, to=3-9]
	\arrow[no head, from=3-10, to=4-11]
	\arrow[no head, from=3-10, to=5-9]
	\arrow[no head, from=3-10, to=5-11]
	\arrow[no head, from=3-10, to=6-10]
	\arrow[no head, from=3-12, to=4-13]
	\arrow[dashed, no head, from=4-2, to=4-1]
	\arrow[no head, from=4-2, to=4-4]
	\arrow[dashed, no head, from=4-2, to=5-2]
	\arrow[no head, from=4-2, to=5-4]
	\arrow[no head, from=4-2, to=6-3]
	\arrow[no head, from=4-4, to=3-5]
	\arrow[no head, from=4-4, to=5-2]
	\arrow[no head, from=4-4, to=5-4]
	\arrow[no head, from=4-4, to=6-3]
	\arrow[dashed, no head, from=4-6, to=5-6]
	\arrow[no head, from=4-9, to=3-10]
	\arrow[no head, from=4-9, to=4-11]
	\arrow[dashed, no head, from=4-9, to=5-9]
	\arrow[no head, from=4-9, to=5-11]
	\arrow[no head, from=4-9, to=6-10]
	\arrow[no head, from=4-11, to=3-12]
	\arrow[no head, from=4-11, to=6-10]
	\arrow[no head, from=4-11, to=5-11]
	\arrow[dashed, no head, from=4-13, to=5-13]
	\arrow[no head, from=5-2, to=5-4]
	\arrow[no head, from=5-2, to=6-3]
	\arrow[no head, from=5-4, to=6-3]
	\arrow[no head, from=5-4, to=6-5]
	\arrow["{{{{\mu_{y_1}}}}}", maps to, from=5-7, to=5-8]
	\arrow[no head, from=5-9, to=4-11]
	\arrow[no head, from=5-9, to=5-11]
	\arrow[no head, from=5-9, to=6-10]
	\arrow[no head, from=5-11, to=6-12]
	\arrow[no head, from=6-5, to=5-6]
	\arrow[no head, from=6-10, to=5-11]
	\arrow[no head, from=6-12, to=5-13]
\end{tikzcd}}

The quiver $\mu_{y_1}(Q)$ is, by  Lemma \ref{An mas largo},  mutation equivalent to the quiver: \\

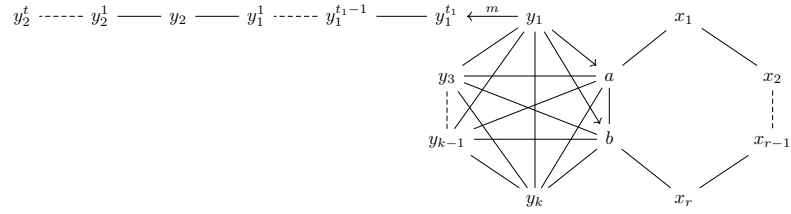
\adjustbox{scale=.7,center}{
\begin{tikzcd}
	{y_2^t} & {y_2^1} & {y_2} & {y_1^1} & {y_1^{t_1-1}} & {y_1^{t_1}} & {y_1} && {x_1} \\
	&&&&& {y_3} && a && {x_2} \\
	&&&&& {y_{k-1}} && b && {x_{r-1}} \\
	&&&&&& {y_k} && {x_r} \\
	&& {}
	\arrow[dashed, no head, from=1-2, to=1-1]
	\arrow[no head, from=1-3, to=1-2]
	\arrow[no head, from=1-4, to=1-3]
	\arrow[dashed, no head, from=1-5, to=1-4]
	\arrow[no head, from=1-6, to=1-5]
	\arrow["m"', from=1-7, to=1-6]
	\arrow[from=1-7, to=2-8]
	\arrow[no head, from=1-7, to=3-6]
	\arrow[from=1-7, to=3-8]
	\arrow[no head, from=1-7, to=4-7]
	\arrow[no head, from=1-9, to=2-10]
	\arrow[no head, from=2-6, to=1-7]
	\arrow[no head, from=2-6, to=2-8]
	\arrow[dashed, no head, from=2-6, to=3-6]
	\arrow[no head, from=2-6, to=3-8]
	\arrow[no head, from=2-6, to=4-7]
	\arrow[no head, from=2-8, to=1-9]
	\arrow[no head, from=2-8, to=3-6]
	\arrow[no head, from=2-8, to=3-8]
	\arrow[dashed, no head, from=2-10, to=3-10]
	\arrow[no head, from=3-6, to=3-8]
	\arrow[no head, from=3-6, to=4-7]
	\arrow[no head, from=3-8, to=4-9]
	\arrow[no head, from=4-7, to=2-8]
	\arrow[no head, from=4-7, to=3-8]
	\arrow[no head, from=4-9, to=3-10]
\end{tikzcd}}

 with the same central cycle but with a complete subquiver  $Q[a,b,y_1,y_3, \cdots, y_k]$ of size $k+1$ . Therefore, the value of the Euler characteristic, after these mutations,   was decreased by  $k$.  \\

 \end{enumerate}

\end{proof}

\begin{lema}\label{tipo I eq a tipo II}
All $m$-coloured quivers of type I with $r+k > 1$ are mutation equivalent to an $m$-coloured quivers of type II.
\end{lema}

\begin{proof}
Let $Q$ be a quiver of type I with $r+k>1$. By \cite[Lemma 4.3]{GR} every quiver in the class $Q^A_{m}$ is mutation equivalent to an $\mathbb{A}$-quiver. Then, following the notation of Definition \ref{la clase},  we can assume that every subquiver $Q^i $  of $Q$ is an $\mathbb{A}$-quiver. \\

Consider the vertex $y_1$ of  ${}^aQ^b[y_1,\cdots,y_k]$ and let $c_1$ be the minimum of the  colours  $\overline{y_1v}$ of the arrows going out from the vertex $y_1$ ( where $v\in \{a,b,y_2, \cdots, y_k \} $).   We can assume that $c_1=0$ and therefore: either there exists an arrow $y_1\rightarrow a$ (or $y_1\rightarrow b$) of colour zero or  there exists an arrow $y_1\rightarrow y_j$ of colour zero.  \\

If we mutate at the vertex $y_1$ in the first case (i.e., $\overline{y_1a}=0$), we obtain  a quiver of type II as follows: \\
\\

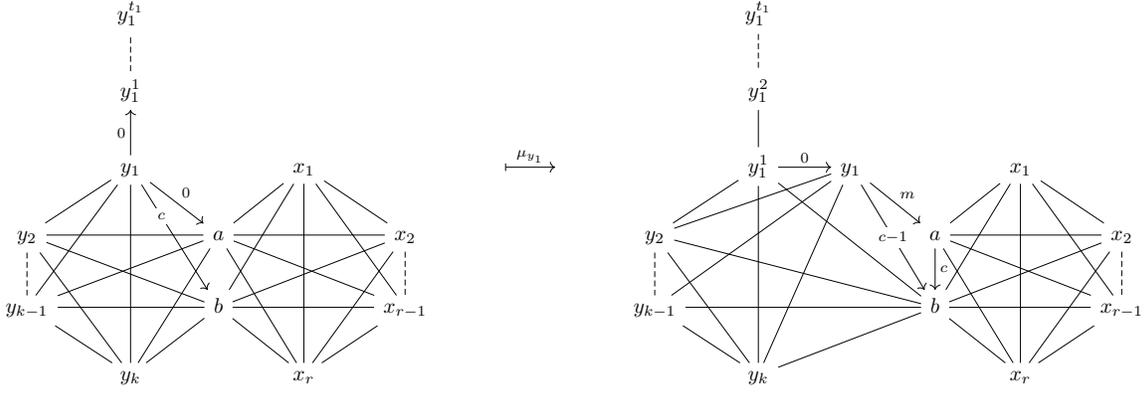
\adjustbox{scale=.8,center}{
\begin{tikzcd}
	& {y_1^{t_1}} &&&&&&& {y_1^{t_1}} \\
	& {y_1^1} &&&&&&& {y_1^2} \\
	& {y_1} && {x_1} && {} & {} && {y_1^1} & {y_1} && {x_1} \\
	{y_2} && a && {x_2} &&& {y_2} &&& a && {x_2} \\
	{y_{k-1}} && b && {x_{r-1}} &&& {y_{k-1}} &&& b && {x_{r-1}} \\
	& {y_k} && {x_r} &&&&& {y_k} &&& {x_r} \\
	&&&& {}
	\arrow[dashed, no head, from=1-9, to=2-9]
	\arrow[dashed, no head, from=2-2, to=1-2]
	\arrow[no head, from=2-9, to=3-9]
	\arrow["0", from=3-2, to=2-2]
	\arrow[no head, from=3-2, to=4-1]
	\arrow["0", from=3-2, to=4-3]
	\arrow[no head, from=3-2, to=5-1]
	\arrow["{{{c}}}"{description, pos=0.3}, from=3-2, to=5-3]
	\arrow[no head, from=3-2, to=6-2]
	\arrow[no head, from=3-4, to=4-5]
	\arrow[no head, from=3-4, to=6-4]
	\arrow["{{{{\mu_{y_1}}}}}", maps to, from=3-6, to=3-7]
	\arrow["0", from=3-9, to=3-10]
	\arrow[no head, from=3-9, to=6-9]
	\arrow["m", from=3-10, to=4-11]
	\arrow[no head, from=3-10, to=5-8]
	\arrow["{{{{c-1}}}}"{description}, from=3-10, to=5-11]
	\arrow[no head, from=3-10, to=6-9]
	\arrow[no head, from=3-12, to=4-13]
	\arrow[no head, from=3-12, to=5-11]
	\arrow[no head, from=3-12, to=5-13]
	\arrow[no head, from=3-12, to=6-12]
	\arrow[no head, from=4-1, to=4-3]
	\arrow[dashed, no head, from=4-1, to=5-1]
	\arrow[no head, from=4-1, to=5-3]
	\arrow[no head, from=4-1, to=6-2]
	\arrow[no head, from=4-3, to=3-4]
	\arrow[no head, from=4-3, to=4-5]
	\arrow[no head, from=4-3, to=5-5]
	\arrow[no head, from=4-3, to=6-4]
	\arrow[dashed, no head, from=4-5, to=5-5]
	\arrow[no head, from=4-5, to=6-4]
	\arrow[no head, from=4-8, to=3-9]
	\arrow[no head, from=4-8, to=3-10]
	\arrow[dashed, no head, from=4-8, to=5-8]
	\arrow[no head, from=4-8, to=5-11]
	\arrow[no head, from=4-8, to=6-9]
	\arrow[no head, from=4-11, to=3-12]
	\arrow[no head, from=4-11, to=4-13]
	\arrow["c", from=4-11, to=5-11]
	\arrow[no head, from=4-11, to=5-13]
	\arrow[no head, from=4-11, to=6-12]
	\arrow[dashed, no head, from=4-13, to=5-13]
	\arrow[no head, from=4-13, to=6-12]
	\arrow[no head, from=5-1, to=4-3]
	\arrow[no head, from=5-1, to=5-3]
	\arrow[no head, from=5-1, to=6-2]
	\arrow[no head, from=5-3, to=3-4]
	\arrow[no head, from=5-3, to=4-5]
	\arrow[no head, from=5-3, to=5-5]
	\arrow[no head, from=5-3, to=6-2]
	\arrow[no head, from=5-3, to=6-4]
	\arrow[no head, from=5-5, to=3-4]
	\arrow[no head, from=5-8, to=5-11]
	\arrow[no head, from=5-8, to=6-9]
	\arrow[no head, from=5-11, to=3-9]
	\arrow[no head, from=5-11, to=4-13]
	\arrow[no head, from=5-11, to=5-13]
	\arrow[no head, from=5-11, to=6-12]
	\arrow[no head, from=6-2, to=4-3]
	\arrow[no head, from=6-4, to=5-5]
	\arrow[no head, from=6-9, to=5-11]
	\arrow[no head, from=6-12, to=5-13]
\end{tikzcd}}

In the second case (i.e., $\overline{y_1y_j}=0$), we can assume without loss of generality that $j=2$ and that the arrow $y_1\rightarrow y_1^1 $ has colour $m$. Therefore, after mutating at the vertex $y_1$, we obtain:  \\

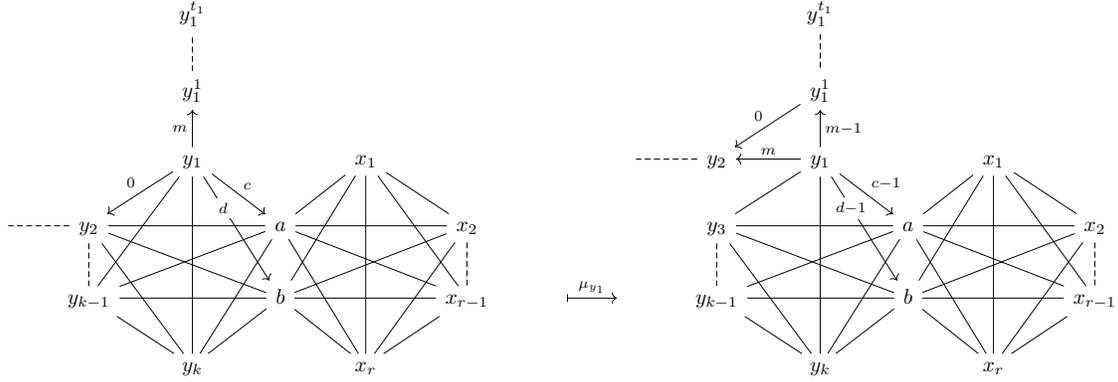
\adjustbox{scale=.8,center}{
\begin{tikzcd}
	&& {y_1^{t_1}} &&&&&&& {y_1^{t_1}} \\
	&& {y_1^1} &&&&&&& {y_1^1} \\
	&& {y_1} && {x_1} &&& {} & {y_2} & {y_1} && {x_1} \\
	{} & {y_2} && a && {x_2} &&& {y_3} && a && {x_2} \\
	& {y_{k-1}} && b && {x_{r-1}} & {} & {} & {y_{k-1}} && b && {x_{r-1}} \\
	&& {y_k} && {x_r} &&&&& {y_k} && {x_r} \\
	&&&&& {}
	\arrow[dashed, no head, from=1-10, to=2-10]
	\arrow[dashed, no head, from=2-3, to=1-3]
	\arrow["0"', from=2-10, to=3-9]
	\arrow["m", from=3-3, to=2-3]
	\arrow["0"', from=3-3, to=4-2]
	\arrow["c", from=3-3, to=4-4]
	\arrow[no head, from=3-3, to=5-2]
	\arrow["d"{description, pos=0.3}, from=3-3, to=5-4]
	\arrow[no head, from=3-3, to=6-3]
	\arrow[no head, from=3-5, to=4-6]
	\arrow[no head, from=3-5, to=5-4]
	\arrow[no head, from=3-5, to=5-6]
	\arrow[no head, from=3-5, to=6-5]
	\arrow[dashed, no head, from=3-9, to=3-8]
	\arrow["{{{m-1}}}"', from=3-10, to=2-10]
	\arrow["m"', from=3-10, to=3-9]
	\arrow["{{{c-1}}}", from=3-10, to=4-11]
	\arrow["{{{d-1}}}"{description, pos=0.3}, from=3-10, to=5-11]
	\arrow[no head, from=3-10, to=6-10]
	\arrow[no head, from=3-12, to=4-13]
	\arrow[no head, from=3-12, to=5-11]
	\arrow[no head, from=3-12, to=5-13]
	\arrow[no head, from=3-12, to=6-12]
	\arrow[dashed, no head, from=4-2, to=4-1]
	\arrow[no head, from=4-2, to=4-4]
	\arrow[dashed, no head, from=4-2, to=5-2]
	\arrow[no head, from=4-2, to=5-4]
	\arrow[no head, from=4-2, to=6-3]
	\arrow[no head, from=4-4, to=3-5]
	\arrow[no head, from=4-4, to=4-6]
	\arrow[no head, from=4-4, to=5-2]
	\arrow[no head, from=4-4, to=5-6]
	\arrow[no head, from=4-4, to=6-3]
	\arrow[no head, from=4-4, to=6-5]
	\arrow[dashed, no head, from=4-6, to=5-6]
	\arrow[no head, from=4-6, to=6-5]
	\arrow[no head, from=4-9, to=3-10]
	\arrow[no head, from=4-9, to=4-11]
	\arrow[dashed, no head, from=4-9, to=5-9]
	\arrow[no head, from=4-9, to=5-11]
	\arrow[no head, from=4-9, to=6-10]
	\arrow[no head, from=4-11, to=3-12]
	\arrow[no head, from=4-11, to=4-13]
	\arrow[no head, from=4-11, to=5-13]
	\arrow[no head, from=4-11, to=6-10]
	\arrow[no head, from=4-11, to=6-12]
	\arrow[dashed, no head, from=4-13, to=5-13]
	\arrow[no head, from=4-13, to=6-12]
	\arrow[no head, from=5-2, to=5-4]
	\arrow[no head, from=5-2, to=6-3]
	\arrow[no head, from=5-4, to=4-6]
	\arrow[no head, from=5-4, to=5-6]
	\arrow[no head, from=5-4, to=6-3]
	\arrow[no head, from=5-4, to=6-5]
	\arrow["{{{\mu_{y_1}}}}", maps to, from=5-7, to=5-8]
	\arrow[no head, from=5-9, to=4-11]
	\arrow[no head, from=5-9, to=5-11]
	\arrow[no head, from=5-9, to=6-10]
	\arrow[no head, from=5-11, to=4-13]
	\arrow[no head, from=5-11, to=5-13]
	\arrow[no head, from=5-11, to=6-12]
	\arrow[no head, from=6-5, to=5-6]
	\arrow[no head, from=6-10, to=5-11]
	\arrow[no head, from=6-12, to=5-13]
\end{tikzcd}}

By Lemma \ref{An mas largo} the quiver $\mu_{y_1}(Q)$ is mutation equivalent to:\\

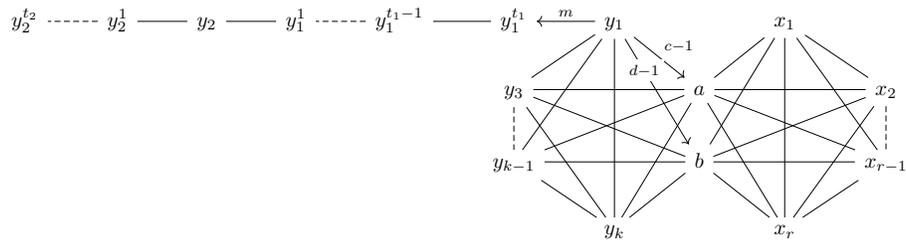
\adjustbox{scale=.8,center}{
\begin{tikzcd}
	{y_2^{t_2}} & {y_2^1} & {y_2} & {y_1^1} & {y_1^{t_1-1}} & {y_1^{t_1}} & {y_1} && {x_1} \\
	&&&&& {y_3} && a && {x_2} \\
	&&&&& {y_{k-1}} && b && {x_{r-1}} \\
	&&&&&& {y_k} && {x_r} \\
	&& {}
	\arrow[dashed, no head, from=1-2, to=1-1]
	\arrow[no head, from=1-3, to=1-2]
	\arrow[no head, from=1-4, to=1-3]
	\arrow[dashed, no head, from=1-5, to=1-4]
	\arrow[no head, from=1-6, to=1-5]
	\arrow["m"', from=1-7, to=1-6]
	\arrow["{{{{c-1}}}}", from=1-7, to=2-8]
	\arrow[no head, from=1-7, to=3-6]
	\arrow["{{{d-1}}}"{description, pos=0.3}, from=1-7, to=3-8]
	\arrow[no head, from=1-7, to=4-7]
	\arrow[no head, from=1-9, to=2-10]
	\arrow[no head, from=1-9, to=3-8]
	\arrow[no head, from=1-9, to=3-10]
	\arrow[no head, from=1-9, to=4-9]
	\arrow[no head, from=2-6, to=1-7]
	\arrow[no head, from=2-6, to=2-8]
	\arrow[dashed, no head, from=2-6, to=3-6]
	\arrow[no head, from=2-6, to=3-8]
	\arrow[no head, from=2-6, to=4-7]
	\arrow[no head, from=2-8, to=1-9]
	\arrow[no head, from=2-8, to=2-10]
	\arrow[no head, from=2-8, to=3-6]
	\arrow[no head, from=2-8, to=3-10]
	\arrow[no head, from=2-8, to=4-9]
	\arrow[dashed, no head, from=2-10, to=3-10]
	\arrow[no head, from=2-10, to=4-9]
	\arrow[no head, from=3-6, to=3-8]
	\arrow[no head, from=3-6, to=4-7]
	\arrow[no head, from=3-8, to=2-10]
	\arrow[no head, from=3-8, to=3-10]
	\arrow[no head, from=3-8, to=4-9]
	\arrow[no head, from=4-7, to=2-8]
	\arrow[no head, from=4-7, to=3-8]
	\arrow[no head, from=4-9, to=3-10]
\end{tikzcd}}

If $c-1$ or $d-1$ is the minimum of the colours of the arrows emanating from $y_1$, we can proceed as in the first case and obtain a quiver of type II.  If, instead, the minimum colour is, for example the colour $\overline{y_1y_3}$ of the arrow $y_1\rightarrow y_3$ we proceed as in the second case, obtaining: \\

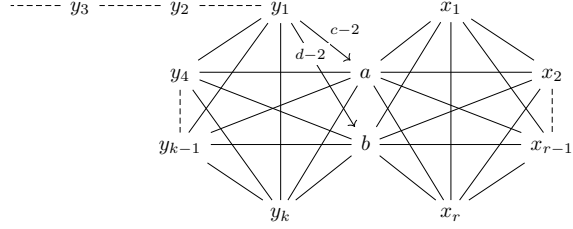
\adjustbox{scale=.8,center}{
\begin{tikzcd}
	& {} & {y_3} & {y_2} & {y_1} && {x_1} \\
	&&& {y_4} && a && {x_2} \\
	&&& {y_{k-1}} && b && {x_{r-1}} \\
	&&&& {y_k} && {x_r} \\
	{}
	\arrow[dashed, no head, from=1-3, to=1-2]
	\arrow[dashed, no head, from=1-4, to=1-3]
	\arrow[dashed, no head, from=1-5, to=1-4]
	\arrow["{{{{{c-2}}}}}", from=1-5, to=2-6]
	\arrow[no head, from=1-5, to=3-4]
	\arrow["{{{{{d-2}}}}}"{description, pos=0.3}, from=1-5, to=3-6]
	\arrow[no head, from=1-5, to=4-5]
	\arrow[no head, from=1-7, to=2-8]
	\arrow[no head, from=1-7, to=3-6]
	\arrow[no head, from=1-7, to=3-8]
	\arrow[no head, from=1-7, to=4-7]
	\arrow[no head, from=2-4, to=1-5]
	\arrow[no head, from=2-4, to=2-6]
	\arrow[dashed, no head, from=2-4, to=3-4]
	\arrow[no head, from=2-4, to=3-6]
	\arrow[no head, from=2-4, to=4-5]
	\arrow[no head, from=2-6, to=1-7]
	\arrow[no head, from=2-6, to=2-8]
	\arrow[no head, from=2-6, to=3-8]
	\arrow[no head, from=2-6, to=4-5]
	\arrow[no head, from=2-6, to=4-7]
	\arrow[dashed, no head, from=2-8, to=3-8]
	\arrow[no head, from=2-8, to=4-7]
	\arrow[no head, from=3-4, to=2-6]
	\arrow[no head, from=3-4, to=3-6]
	\arrow[no head, from=3-4, to=4-5]
	\arrow[no head, from=3-6, to=2-8]
	\arrow[no head, from=3-6, to=3-8]
	\arrow[no head, from=3-6, to=4-7]
	\arrow[no head, from=4-5, to=3-6]
	\arrow[no head, from=4-7, to=3-8]
\end{tikzcd}}

We can iterate this procedure until we get the $[a,b]$-quasi-complete quiver ${}^aQ^b[y_1]$. Assume that  $\overline{y_1b}$ es the minimum of the colours $\{\overline{y_1a},\overline{y_1b}\}$. Therefore, after mutating  $\overline{y_1b}+1$ times at $y_1$, we obtain the following quiver of type II:\\

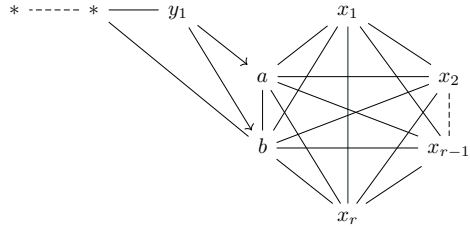
\adjustbox{scale=.8,center}{
\begin{tikzcd}
	& {} & {*} & {*} & {y_1} && {x_1} \\
	&&&&& a && {x_2} \\
	&&&&& b && {x_{r-1}} \\
	&&&&&& {x_r} \\
	{}
	\arrow[dashed, no head, from=1-4, to=1-3]
	\arrow[no head, from=1-4, to=3-6]
	\arrow[no head, from=1-5, to=1-4]
	\arrow[from=1-5, to=2-6]
	\arrow[from=1-5, to=3-6]
	\arrow[no head, from=1-7, to=2-8]
	\arrow[no head, from=1-7, to=3-6]
	\arrow[no head, from=1-7, to=3-8]
	\arrow[color={rgb,255:red,13;green,33;blue,28}, no head, from=1-7, to=4-7]
	\arrow[no head, from=2-6, to=1-7]
	\arrow[no head, from=2-6, to=2-8]
	\arrow[no head, from=2-6, to=3-8]
	\arrow[no head, from=2-6, to=4-7]
	\arrow[no head, from=2-8, to=3-6]
	\arrow[dashed, no head, from=2-8, to=3-8]
	\arrow[no head, from=2-8, to=4-7]
	\arrow[no head, from=3-6, to=2-6]
	\arrow[no head, from=3-6, to=3-8]
	\arrow[no head, from=3-6, to=4-7]
	\arrow[no head, from=4-7, to=3-8]
\end{tikzcd}}

\end{proof}

Follows immediately from \cite[Lemma 4.3]{GR} that an  $m$-coloured quiver of type I with  $r+k = 1$ is mutation equivalent to $\mathbb{D}_n$. Consequently, Lemmas \ref{tipo II --- Dn} and \ref{tipo I eq a tipo II} directly imply the following.\\

\begin{coro}\label{tipo I --- Dn}
All $m$-coloured quivers of type I are mutation equivalent to $\mathbb{D}_n$.
\end{coro}

\vspace*{.5cm}

Now, we are ready to prove one of the main results of this paper.\\

\begin{teo}
An $m$-coloured connected quiver $Q$ is mutation equivalent to $\mathbb{D}_n$ if and only if it belongs to  the class $\mathcal{Q}_{n,m}^D$.
\end{teo}

\begin{proof}

Obviously, all colourations of an  $\mathbb{D}_n$-quiver are in $\mathcal{Q}_{n,m}^D$. Let $Q$ be a coloured quiver belonging to $\mathcal{Q}_{n,m}^D$  which is not a  $\mathbb{D}_n$-quiver.  It follows upon applying Lemma \ref{tipo II --- Dn} and Corollary \ref{tipo I --- Dn} that $Q$ is mutation equivalent to a  $\mathbb{D}_n$-quiver.  The proof is completed by noting that, by Proposition \ref{cerrado por mutaciones}, the class $\mathcal{Q}_n^m$ is closed under coloured mutation.
\end{proof}

\begin{ejem}\label{ejemplo de un carcaj de tipo D15}
 We conclude this section by presenting examples of $m$-coloured quivers in the class $\mathcal{Q}_{n,m}^D$. Specifically, we provide an example for type II followed by one for type I in the class $\mathcal{Q}_{14,2}^D$.
 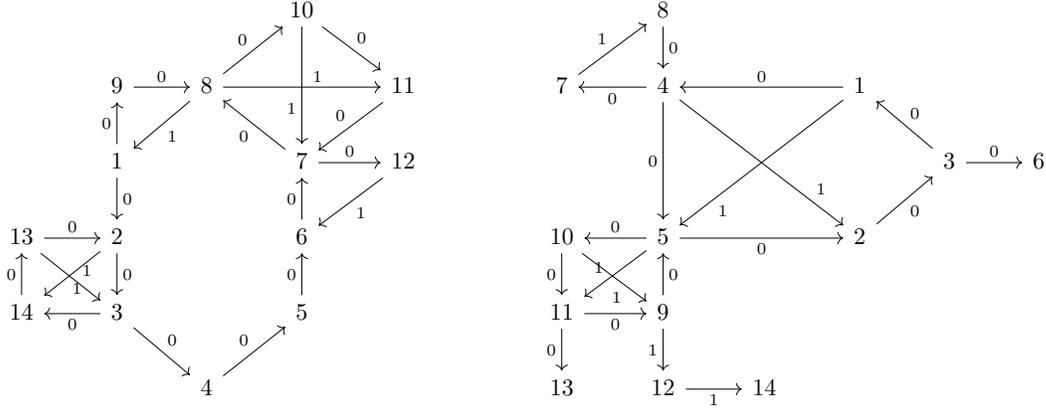
\begin{figure}[H]
\begin{center}

\adjustbox{scale=.9,center}{
\begin{tikzcd}
	&&& 10 &&&& 8 &&&& \\
	& 9 & 8 && 11 && 7 & 4 && 1 \\
	& 1 && 7 & 12 &&&&&& 3 & 6 \\
	13 & 2 && 6 &&& 10 & 5 && 2 \\
	14 & 3 && 5 &&& 11 & 9 \\
	&& 4 &&&& 13 & 12 & 14
	\arrow["0", from=1-4, to=2-5]
	\arrow["1"'{pos=0.7}, from=1-4, to=3-4]
	\arrow["0", from=1-8, to=2-8]
	\arrow["0", from=2-2, to=2-3]
	\arrow["0", from=2-3, to=1-4]
	\arrow["1"{pos=0.6}, from=2-3, to=2-5]
	\arrow["1", from=2-3, to=3-2]
	\arrow["0"', from=2-5, to=3-4]
	\arrow["1", from=2-7, to=1-8]
	\arrow["0", from=2-8, to=2-7]
	\arrow["0"', from=2-8, to=4-8]
	\arrow["1"{pos=0.8}, from=2-8, to=4-10]
	\arrow["0"', from=2-10, to=2-8]
	\arrow["1"{pos=0.8}, from=2-10, to=4-8]
	\arrow["0", from=3-2, to=2-2]
	\arrow["0", from=3-2, to=4-2]
	\arrow["0", from=3-4, to=2-3]
	\arrow["0", from=3-4, to=3-5]
	\arrow["1", from=3-5, to=4-4]
	\arrow["0"', from=3-11, to=2-10]
	\arrow["0", from=3-11, to=3-12]
	\arrow["0", from=4-1, to=4-2]
	\arrow["1"{pos=0.6}, from=4-1, to=5-2]
	\arrow["1"{pos=0.6}, from=4-2, to=5-1]
	\arrow["0", from=4-2, to=5-2]
	\arrow["0", from=4-4, to=3-4]
	\arrow["0"', from=4-7, to=5-7]
	\arrow["1"'{pos=0.7}, from=4-7, to=5-8]
	\arrow["0"', from=4-8, to=4-7]
	\arrow["0"', from=4-8, to=4-10]
	\arrow["1"'{pos=0.6}, from=4-8, to=5-7]
	\arrow["0"', from=4-10, to=3-11]
	\arrow["0", from=5-1, to=4-1]
	\arrow["0", from=5-2, to=5-1]
	\arrow["0", from=5-2, to=6-3]
	\arrow["0", from=5-4, to=4-4]
	\arrow["0"', from=5-7, to=5-8]
	\arrow["0"', from=5-7, to=6-7]
	\arrow["0"', from=5-8, to=4-8]
	\arrow["1"', from=5-8, to=6-8]
	\arrow["0", from=6-3, to=5-4]
	\arrow["1"', from=6-8, to=6-9]
\end{tikzcd}}

\vspace*{0.5cm}

\caption{Example of $2$-coloured quivers in the class $\mathcal{Q}_{14,2}^D$.}
\label{carcaj de tipo D15}
\end{center}
\end{figure}

\end{ejem}

\section{The $0$-coloured part of a quiver in the mutation class of $\mathbb{D}_n$}

As in the $\mathbb{A}_n$-case,  the $0$-coloured part of a quiver in the mutation class of $\mathbb{D}_n$ plays an important role in the study of the $m$-cluster tilted algebra of type $\mathbb{D}_n$. In fact, every quiver of an $m$-cluster tilted algebras of type $\mathbb{D}_n$  can be obtained  by considering the $0$-coloured part of the  quiver in the class $\mathcal{Q}_{n,m}^D$. \\

In particular, the lengths of the cycles in the $0$-coloured part bring information about the possible cycles that an $m$-cluster tilted algebra of type $\mathbb{D}_n$ can have.  \\

 By \cite{GQP} if $Q'$ is the $0$-coloured part of an $m$-coloured quiver $Q$ in the mutation class of $\mathbb{A}_{n}$, then $Q'$ admits only cycles of length $m+2$. The following example shows that this is no longer true for the mutation class of  $\mathbb{D}_n$. In fact, the  quiver of an $m$-cluster tilted algebra of type $\mathbb{D}_n$ can have either one non-oriented cycle of arbitrary length $t \leq n$  or an oriented cycle of lenght $m+3$; and any other cycle must be oriented and have length $m+2$. \\

\begin{ejem}By computing the $0$-coloured part of the $2$-coloured quivers in Figure \ref{carcaj de tipo D15}, we obtain the quivers shown below, which correspond to $2$-cluster tilted algebras of type $\mathbb{D}_{14}$. Note that the resulting quiver may be disconnected.

\begin{figure}[H]
\begin{center}

\adjustbox{scale=.9,center}{
\begin{tikzcd}
	&&& 10 &&&& 8 &&&& \\
	& 9 & 8 && 11 && 7 & 4 && 1 \\
	& 1 && 7 & 12 &&&&&& 3 & 6 \\
	13 & 2 && 6 &&& 10 & 5 && 2 \\
	14 & 3 && 5 &&& 11 & 9 \\
	&& 4 &&&& 13 & 12 & 14
	\arrow[ from=1-4, to=2-5]
	\arrow[ from=1-8, to=2-8]
	\arrow[ from=2-2, to=2-3]
	\arrow[ from=2-3, to=1-4]
	\arrow[ from=2-5, to=3-4]
	\arrow[ from=2-8, to=2-7]
	\arrow[ from=2-8, to=4-8]
	\arrow[ from=2-10, to=2-8]
	\arrow[ from=3-2, to=2-2]
	\arrow[ from=3-2, to=4-2]
	\arrow[ from=3-4, to=2-3]
	\arrow[ from=3-4, to=3-5]
	\arrow[ from=3-11, to=2-10]
	\arrow[ from=3-11, to=3-12]
	\arrow[ from=4-1, to=4-2]
	\arrow[ from=4-2, to=5-2]
	\arrow[ from=4-4, to=3-4]
	\arrow[ from=4-7, to=5-7]
	\arrow[ from=4-8, to=4-7]
	\arrow[ from=4-8, to=4-10]
	\arrow[ from=4-10, to=3-11]
	\arrow[ from=5-1, to=4-1]
	\arrow[ from=5-2, to=5-1]
	\arrow[ from=5-2, to=6-3]
	\arrow[ from=5-4, to=4-4]
	\arrow[ from=5-7, to=5-8]
	\arrow[ from=5-7, to=6-7]
	\arrow[ from=5-8, to=4-8]
	\arrow[ from=6-3, to=5-4]
\end{tikzcd}}

\vspace*{0.5cm}

\caption{The  $0$-coloured part of the quivers of Figure \ref{carcaj de tipo D15}.}
\end{center}
\end{figure}
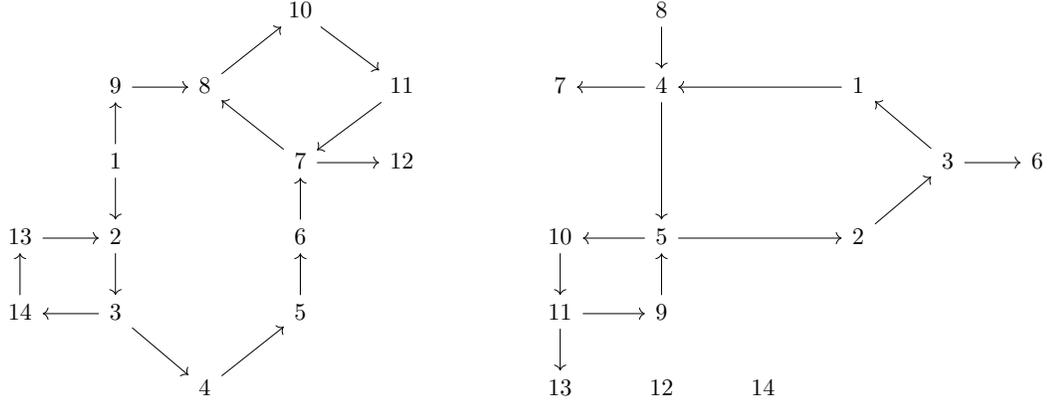

\end{ejem}

We recall in the following remark an important property about the  $0$-coloured part of a quiver in the mutation class of $\mathbb{A}_{n}$.\\

\begin{obs}\cite[Corollary 7.1]{GQP}\label{m+2 ciclos del An} Let $Q$ be an $m$-coloured quiver in the mutation class of $\mathbb{A}_{n}$ and consider its $0$-coloured part $Q'$. Suppose there is a $k$-cycle $c=(x_1x_2\cdots x_k)$ in $Q'$. Then, $k=m+2$.
\end{obs}

We finish this section with two theorems which characterise  the connected quivers of an $m$-cluster tilted algebra of type $\mathbb{D}_n$. \\

\begin{teo} Let $Q$ be a connected quiver. Then $Q$ is a connected component of  the quiver of an $m$-cluster tilted algebra of type $\mathbb{D}_n$ of subtype I if and only if $Q$ is an induced subquiver of a quiver $\tilde{Q}$ that satisfies the following conditions:

\begin{enumerate}
\item $\tilde{Q}$ contains a subgraph $S$ which is an oriented cycle  of length $m + 3$;
 \item  there exist exactly two non-adjacent vertices $a, b$ in $S$ such that all neighbours of $a$ and $b$ necessarily belong to $S$;
 \item for every vertex $x\in \tilde{Q}_0$ the sets $s^{-1}(x)$ and $t^{-1}(x)$ have cardinality at most two;
  \item the only other possible cycles which can occur in $\tilde{Q}$  are oriented cycles of length $m+2$, provided they do not share any arrows with $S$.
\end{enumerate}

\end{teo}

\begin{proof} We must consider the $0$-coloured part of the quivers in the class $\mathcal{Q}_{n,m}^D$ of type I.
Let $\widehat{Q}$ be such a quiver; then $\widehat{Q}$ possesses exactly two vertices, $a$ and $b$, both of which belong to two $[a,b]$-quasi-complete quivers, denoted by ${}^a\widehat{Q}^b[x_1,\dots,x_k]$ and ${}^a\widehat{Q}^b[y_1,\dots,y_r]$, with $1 \leq r+k \leq m+1$. If $1 < r+k = m+1$, Corollary \ref{1<r+k=m+1 entonces flecha de color 0} implies the existence of an oriented cycle of length $m+3$ of colour zero,  containing exactly two distinguished non-adjacent vertices, $a$ and $b$.
Both vertices have valency 2, thereby satisfying conditions $(1)$ and $(2)$. Moreover, Lemma \ref{colores de los 3 ciclos} ensures that condition $(3)$ holds, while $(4)$ follows directly from Remark \ref{m+2 ciclos del An} and condition $(2)$ of Definition \ref{la clase}.
Finally, we analyse the construction of full subquivers from the quivers described above. These arise whenever the cycle $(av_1\dots v_kbv_{k+1}\dots v_{r+k})$  necessarily contains arrows of non-zero colour. This occurs specifically when $1 < r+k < m+1$,  by Theorem \ref{theo-coloracion ciclos}.
Conversely, if $r=0$ and $1 \leq k\leq m+1$, condition $(3)$ of Definition \ref{la clase} implies the existence of  an oriented path of colour $0$ from $a$ to $b$, of the form $a \rightarrow v_1 \rightarrow \dots \rightarrow v_{k} \rightarrow b$, which also satisfies $(1)$ and $(2)$.  If the path contains arrows of non-zero colour we obtain  the required induced subquivers.

\end{proof}

\begin{obs}
It is worth noting that whenever a subquiver of $S$ is considered, the obtained quivers are also quivers of $m$-cluster tilted algebras of type $\mathbb{A}$.
\end{obs}

\begin{ejem} If we compute the $0$-coloured part of the  induced subquiver of type I of Figure \ref{carcaj de tipo D15}  generated by the set of vertices $\{1, \dots,  8, 13, 14\}$, we obtain the quiver of a connected component of a $2$-cluster tilted algebra of type  $\mathbb{D}_{14}$ of subtype I which is also a $2$-cluster tilted algebra of type  $\mathbb{A}_{14}$.

\begin{figure}[H]
\begin{center}
\adjustbox{scale=.8,center}{

\begin{tikzcd}
	&& 8 & \\
	& 1 && 7 \\
	13 & 2 && 6 \\
	14 & 3 && 5 \\
	&& 4
	\arrow[from=2-2, to=3-2]
	\arrow[from=2-4, to=1-3]
	\arrow[from=3-1, to=3-2]
	\arrow[from=3-2, to=4-2]
	\arrow[from=3-4, to=2-4]
	\arrow[from=4-1, to=3-1]
	\arrow[from=4-2, to=4-1]
	\arrow[from=4-2, to=5-3]
	\arrow[from=4-4, to=3-4]
	\arrow[from=5-3, to=4-4]
\end{tikzcd}}

\vspace*{.5cm}

\caption{An example of a $2$-cluster tilted algebra simultaneously of type $\mathbb{D}_{14}$ and $\mathbb{A}_{14}$.}
\end{center}
\end{figure}
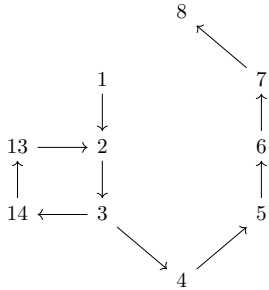

\end{ejem}
Considering a quiver that contains a cycle as shown below; we may fix a planar embedding of it to distinguish between clockwise and counterclockwise oriented arrows. However, it should be noted that this convention is unique only up to reflection, as reversing the embedding would swap the roles of the clockwise and counterclockwise orientations.

\begin{teo} Let $Q$ be a connected quiver.  Then $Q$ is a connected component of  the quiver of an $m$-cluster tilted algebra of type $\mathbb{D}_n$ of subtype II if and only if $Q$ is an induced subquiver of a quiver $\tilde{Q}$ that satisfies the following conditions:

\begin{enumerate}
\item $\tilde{Q}$ contains a non-oriented cycle $\mathcal{C}$  possessing $(m+1)r-(m-1)$ clockwise arrows and at least one counterclockwise arrow.
    The clockwise arrows are partitioned into $r$ disjoint groups, where the $i$-th group consists of $m+1-k_i$ consecutive arrows. The positive integers $k_i$ satisfy: $$\sum_{i=1}^r k_i = m - 1, \quad \text{for } 1 \leq r \leq m - 1.$$

  \item for every vertex $x\in \tilde{Q}_0$ the sets $s^{-1}(x)$ and $t^{-1}(x)$ have cardinality at most two;
  \item the only other possible cycles which can occur in $\tilde{Q}$  are oriented cycles of length $m+2$  that do not include any of the clockwise arrows of $\mathcal{C}$.
\end{enumerate}

Furthermore, $Q$ is obtained from $\tilde{Q}$ by removing $t$ of the $r$ groups of consecutive clockwise arrows, along with any resulting isolated vertices, where $0 \leq t \leq r$.

\end{teo}

\begin{proof}

We must consider the $0$-coloured part of the quivers in the class $\mathcal{Q}_{n,m}^D$ of type II.
Let $\widehat{Q}$ be such a quiver; then $\widehat{Q}$ has a full subquiver $\mathcal{O}$, called the  \textit{central cycle}, which is an induced $k$-cycle $(x_1\cdots x_k)$ with $k\geq 3$ and colouration $\kappa(\mathcal{O})=m-1$.

Suppose that $w(x_1 \dots x_k) = m-1$ and that this orientation is drawn counterclockwise. By writing $m-1 = \sum_{i=1}^r k_i$, we obtain $r$ arrows $\alpha_i$ in the cycle $\mathcal{O}$, with colouration $\kappa(\alpha_i) = k_i$ for $i = 1, \dots, r$.  According to \cite[Theorem 6.7]{GQP}, $d$-cycles have a colouration of $m+2-d$; thus, cycles of length $m+2-k_i$ are coloured by $k_i$. Consequently, for each $i$, there may exist $m+1-k_i$ consecutive arrows of colour $0$ which, together with the arrow $\alpha_i$, form such a cycle. By condition (5) of Definition \ref{la clase}, these arrows must be oriented clockwise. Following this construction for each $i$, we obtain a total of $
\sum_{i=1}^r (m+1-k_i) = (m+1)r - \sum_{i=1}^r k_i = (m+1)r - (m-1)$
clockwise arrows. Upon removing all arrows $\alpha_i$ of non-zero color $k_i$, we recover the cycle $\mathcal{C}$ described in (1).

It is worth noting that (2) and (3) follow directly from conditions (2), (3) and (4) of Definition \ref{la clase}. Furthermore, if these cycles are not fully formed for every arrow $\alpha_i$, deleting the non-zero coloured arrows yields the induced subquivers mentioned above.

\end{proof}

\begin{ejem} An illustration of the aforementioned type is provided for $m=4$, considering the decomposition $3=1+2$.

\begin{figure}[H]
\begin{center}
\adjustbox{scale=.8,center}{
\begin{tikzcd}
	&& 8 &&&&&& 8 && \\
	& 1 && 7 &&&& 1 && 7 \\
	13 & 2 && 6 & {} && 13 & 2 && 6 \\
	14 & 3 && 5 & 11 && 14 & 3 && 5 & 11 \\
	&& 4 &&&&&& 4 \\
	&& 9 && 10 &&&& 9 && 10
	\arrow["0", from=1-3, to=2-2]
	\arrow[from=1-9, to=2-8]
	\arrow["0", from=2-2, to=3-2]
	\arrow["0", from=2-4, to=1-3]
	\arrow[from=2-8, to=3-8]
	\arrow[from=2-10, to=1-9]
	\arrow["0", from=3-1, to=3-2]
	\arrow["2", from=3-2, to=4-2]
	\arrow["0", from=3-4, to=2-4]
	\arrow[from=3-7, to=3-8]
	\arrow[from=3-10, to=2-10]
	\arrow["0", from=4-1, to=3-1]
	\arrow["1"{pos=0.2}, from=4-1, to=3-2]
	\arrow["1"'{pos=0.6}, from=4-2, to=3-1]
	\arrow["0", from=4-2, to=4-1]
	\arrow["0", from=4-2, to=5-3]
	\arrow["0", from=4-4, to=3-4]
	\arrow["0", from=4-4, to=4-5]
	\arrow["2"{description}, from=4-4, to=6-3]
	\arrow["2"{description}, from=4-5, to=5-3]
	\arrow["1"{description}, from=4-5, to=6-3]
	\arrow["0", from=4-5, to=6-5]
	\arrow[from=4-7, to=3-7]
	\arrow[from=4-8, to=4-7]
	\arrow[from=4-8, to=5-9]
	\arrow[from=4-10, to=3-10]
	\arrow[from=4-10, to=4-11]
	\arrow[from=4-11, to=6-11]
	\arrow["1", from=5-3, to=4-4]
	\arrow["0"', from=6-3, to=5-3]
	\arrow["3"{description}, from=6-5, to=4-4]
	\arrow["1"{description}, from=6-5, to=5-3]
	\arrow["0", from=6-5, to=6-3]
	\arrow[from=6-9, to=5-9]
	\arrow[from=6-11, to=6-9]
\end{tikzcd}}

\vspace*{.5cm}
\caption{A $4$-coloured quiver from $\mathcal{Q}_{13,4}^D$ (left) and the resulting $0$-coloured subquiver (right).}
\end{center}
\end{figure}
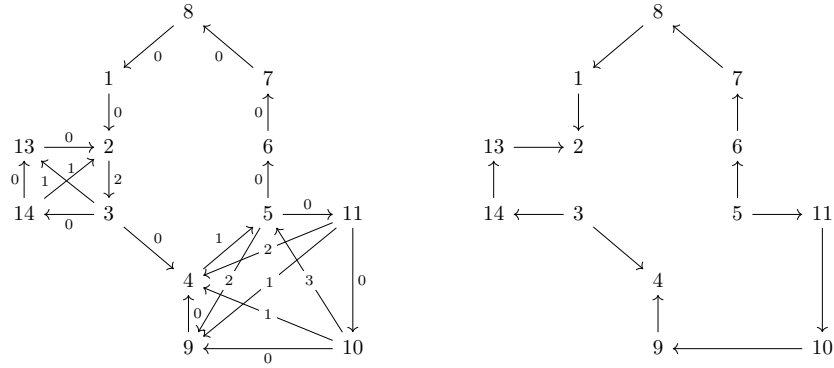

\end{ejem}

\section*{Data availability}
Data sharing not applicable to this article as no datasets were generated or analysed during the current study.

\section*{Conflict of interest}
There are no conflicts of interests or competing interests.

\section*{Acknowledgements}
The authors gratefully  acknowledge financial support from CSIC (\textit{Comisi\'on Sectorial de Investigaci\'on Cient\'ifica}) of Uruguay (grant no. 22520220100067UD).

\end{document}